\documentclass{article}
\usepackage[utf8]{inputenc}
\usepackage{lmodern}
\usepackage[T1]{fontenc}
\usepackage{amsmath,amssymb}
\usepackage{hyperref}
\usepackage{mathrsfs}
\usepackage{amsfonts}
\usepackage{mathtools}
\usepackage{tikz}
\usepackage{authblk}
\usepackage{lmodern}
\usepackage{bbm}
\usepackage{amsthm}
\usepackage{geometry}
\newcommand{\md}[1]{\left\lvert#1\right\rvert}
\newcommand{\norm}[1]{\left\lVert#1\right\rVert}
\newtheorem{theorem}{Theorem}[section]
\newtheorem{corollary}[theorem]{Corollary}
\newtheorem{lemma}[theorem]{Lemma}
\newtheorem{definition}[theorem]{Definition}
\newtheorem{notation}[theorem]{Notation}
\newtheorem{remark}[theorem]{Remark}
\theoremstyle{definition}

\DeclareMathOperator{\sech}{sech}
\DeclareMathOperator{\supp}{supp}

\DeclareMathOperator{\arctanh}{arctanh}

\title{Dynamics of two interacting kinks for the $\phi^{6}$ model}
\author[1]{Abdon Moutinho\thanks{Eletronic address : moutinho@math.univ-paris13.fr}}
\affil[1]{Department of Mathematics: LAGA, Université Sorbonne Paris Nord, 99 Avenue Jean Baptiste Clément, 93430-Villetaneuse, France}
\begin{document}
\date{}
\maketitle
\nocite{*}
\begin{abstract}
    We consider the nonlinear wave equation known as the $\phi^{6}$ model in dimension 1+1. We describe the long time behavior of all the solutions of this model close to a sum of two kinks with energy slightly larger than twice the minimum energy of non constant stationary solutions. We prove orbital stability of two moving kinks. We show for low energy excess $\epsilon$ that these solutions can be described for long time less or equivalent than $-\ln{(\epsilon)}\epsilon^{-\frac{1}{2}}$ as the sum of two moving kinks such that each kink's center is close to an explicit function which is a solution of an ordinary differential system. We give an optimal estimate in the energy norm of the remainder $(g(t),\partial_{t}g(t))$ and we prove that this estimate is achieved during a finite instant $t=T\lesssim -\ln{(\epsilon)}\epsilon^{-\frac{1}{2}}.$  
\end{abstract}
\section{Introduction}
We consider a nonlinear wave equation equation known as the $\phi^{6}$ model. For the potential function $U(\phi)=\phi^{2}(1-\phi^{2})^{2}$ and $\dot U(\phi)=2\phi-8\phi^{3}+6\phi^{5},$ the equation is written as 
\begin{equation}\label{nlww}
    \partial_{t}^{2}\phi(t,x)-\partial_{x}^{2}\phi(t,x)+\dot U(\phi(t,x))=0 \text{, $(t,x) \in \mathbb{R}\times \mathbb{R}$}.
\end{equation}
The potential energy $E_{pot},$ the kinetic energy $E_{kin}$ and total energy $E_{total}$ associated to the equation \eqref{nlww} are given by
\begin{align*}
    E_{pot}(\phi(t))=\frac{1}{2}\int_{\mathbb{R}}\partial_{x}\phi(t,x)^{2}\,dx+\int_{\mathbb{R}}\phi(t,x)^{2}(1-\phi(t,x)^{2})^{2}\,dx,\\
    E_{kin}(\phi(t))=\frac{1}{2}\int_{\mathbb{R}}\partial_{t}\phi(t,x)^{2}\,dx,\\
    E_{total}(\phi(t),\partial_{t}\phi(t))=\frac{1}{2}\int_{\mathbb{R}}\left[\partial_{x}\phi(t,x)^{2}+\partial_{t}\phi(t,x)^{2}\right]\,dx+\int_{\mathbb{R}}\phi(t,x)^{2}(1-\phi(t,x)^{2})^{2}\,dx.
\end{align*}
We say that if a solution $\phi(t,x)$ of the integral equation associated to \eqref{nlww} has $E_{total}(\phi,\partial_{t}\phi)<+\infty,$ then it is in the energy space. The solutions of \eqref{nlww} in the energy space have constant total energy $E_{total}(\phi(t),\partial_{t}\phi(t)).$
By standard arguments, the Cauchy Problem associated \eqref{nlww} is locally well-posed in the energy space, moreover is globally well-posed since $U(\phi)=\phi^{2}(1-\phi^{2})^{2}$ satisfies $\lim_{|\phi|\to \infty} U(\phi)=+\infty.$
\par The stationary solutions of \eqref{nlww} are the critical points of the potential energy. The only non-constant stationary solutions in \eqref{nlww} are the topological solitons called kinks and anti-kinks, for more details see chapter $5$ of \cite{solitonss}. The kinks of \eqref{nlww} are given by  \begin{equation*}
H_{0,1}(x-a)=\frac{e^{\sqrt{2}(x-a)}}{(1+e^{2\sqrt{2}(x-a)})^{\frac{1}{2}}}, \, H_{-1,0}(x-a)=-H_{0,1}(-x+a)   
\end{equation*} 
for any real $a.$ The study of kink and multi kinks solitons solutions of nonlinear wave equations has applications in many domains of mathematical physics. More precisely, the model \eqref{nlww} that we study has applications in condensed matter physics \cite{condensed} and cosmology \cite{cosmic}, \cite{physicsa1}, \cite{physicsa2}.     
\par It is well known that the set of solutions in energy Space of \eqref{nlww} for any potential $U$ is invariant under space translation, time translation and space reflection. Also, for one of the stationary solutions $H$ of \eqref{nlww} and any $-1<v<1$, we have that the following solitary wave
\begin{equation*}
    H\Big(\frac{x-vt}{(1-v^{2})^{\frac{1}{2}}}\Big),
\end{equation*}
which is the Lorentz transform of $H$ is a solution of \eqref{nlww}.
\par The problem of stability of multi-kinks is of great interest in mathematical physics, see for example \cite{collision}, \cite{kinkcollision}. For the integrable model \textit{mKdV}, Muñoz proved in \cite{munozz} the $H^{1}$ stability and asymptotic stability of multi-kinks. However, for many non-integrable models as the $\phi^{6}$ nonlinear wave equation, the asymptotic and long time dynamics of multi-kinks after the instant where the collision or interaction happens are still unknown, even though there are numerical studies of kink-kink collision for the $\phi^{6}$ model, see \cite{kinkcollision}, which motivate our research on the topic of the description of long time behavior of a kink-kink pair. 
\par For nonlinear wave equation models in $1+1$ space dimension, results of stability for a single kink were obtained, for example for the $\phi^{4}$ model it was obtained asymptotic stability for odd pertubations  in \cite{munoz} and  \cite{kinkdelort}. Also, it was recently proved in \cite{asympt} by Martel, Muñoz, Kowalczyk and Van Den Bosch asymptotic stability of a single kink for a general class of nonlinear wave equations, including the model which we study here.  
\par The main purpose of our material is to describe the long time behavior of solutions $\phi(t,x)$ of \eqref{nlww} in the energy space such that
\begin{align*}
    \lim_{x\to +\infty}\phi(t,x)=1,\\
    \lim_{x\to -\infty} \phi(t,x)=-1
\end{align*}
with total energy equals to $2E_{pot}(H_{01})+\epsilon,$ for $0<\epsilon\ll 1.$ More precisely, we proved orbital stability for a sum of two moving kinks with total total energy $2E_{pot}(H_{0,1})+\epsilon$ and we verified that the remainder has a better estimate during a long time interval which goes to $\mathbb{R}$ as $\epsilon\to 0,$ indeed we proved that the estimate of the remainder during this long time interval is optimal. Also, we prove that the dynamics of the kinks movement is very close to two explicit functions $d_{j}:\mathbb{R}\to \mathbb{R}$ defined in Theorem \ref{trueTheo2} during a long time interval. This result is very important to understand the behavior of two kinks after the instant of collision, which happens when the kinetic energy is minimal, indeed, our main results Theorem \ref{T1}  and Theorem \ref{trueTheo2} describe the dynamics of the kinks before and after the collision instant for a long time interval. The numerical study of interaction and collision between kinks for the $\phi^{6}$ model was done in \cite{kinkcollision}, in which it was verified that the collision of kinks is close to an elastic collision when the speed of each kink is low and smaller than a critical speed $v_{c}$.
 \par For nonlinear wave equation models in dimension $2+1,$ there are similar results obtained in the dynamics of topological multi-solitons. For the Higgs Model, there are results in the description dynamics of multi-vortices in \cite{DynamicsHiggs} obtained by Stuart and in \cite{Vorticesdynamics} obtained by Gustafson and Sigal. Indeed, we took inspiration from the proof and statement of Theorem $2$ of \cite{Vorticesdynamics} to construct our main results. Also, in \cite{geodesicYangmills}, Stuart described the dynamics of monopole solutions for the Yang-Mills-Higgs equation. For more references, see also \cite{adiabaticc}, \cite{reducedwave}, \cite{coupled} and \cite{nonintegrable}.
\par In \cite{Dynamicsmultiple}, Bethuel, Orlandi and Smets described the asymptotic behavior of solutions of a parabolic Ginzburg-Landau equation closed to multi-vortices in the initial instant. For mores references, see also \cite{dynamicsgl} and \cite{gamma}. 
\par There are also results in the dynamics in multi-vortices for nonlinear Schrödinger equation, for example the description of the dynamics of multi-vortices for the Gross-Pitaevski equation was obtained in \cite{GL3} by Ovchinnikov and Sigal and results in the dynamics of vortices for the Ginzburg-Landau-Schrödinger equations were proved in \cite{VorticesGLS} by Colliander and Jerrard, see also \cite{refined} for more information about Gross-Pitaevski equation.
\subsection{Main results}
\par We recall that the objective of this paper is to show orbital stability for the solutions of the equation \eqref{nlww} which are close to a sum of two interacting kinks in an initial instant and estimate the size of the time interval where better stability properties hold. The main techniques of the proof are modulation techniques adapted from \cite{jkl}, \cite{gkdvmulti} and \cite{blowup} and a refined energy estimate method to control the size of the remainder term. 
\begin{notation}
For any $D\subset \mathbb{R},$ any real function $f:D\subset\mathbb{R}\to\mathbb{R},$ a real positive function $g$ with domain $D$ is in $O(f(x))$ if and only if there is a uniform constant $C>0$ such that $0<g(x)<C|f(x)|.$ We denote that two real non-negative functions $f,g:D\subset\mathbb{R}\to\mathbb{R}_{\geq 0}$ satisfy
\begin{equation*}
    f\lesssim g,
\end{equation*}
if there is a constant $C>0$ such that
\begin{equation*}
    \md{f(x)}\leq C\md{g(x)} \text{, for all $x \in D.$}
\end{equation*}
If $f\lesssim g$ and $g \lesssim f,$ we denote that $f\cong g.$
We use the notation $(x)_{+}\coloneqq\max(x,0)$.
If $g(t,x) \in C^{1}(\mathbb{R},L^{2}(\mathbb{R}))\cap C(\mathbb{R},H^{1}(\mathbb{R})),$
then we define $\overrightarrow{g(t)} \in H^{1}(\mathbb{R})\times L^{2}(\mathbb{R})$ by
\begin{equation*}
    \overrightarrow{g(t)}=(g(t),\partial_{t}g(t)),
\end{equation*}
and we also denote the energy norm of the remainder $\overrightarrow{g(t)}$ as  
\begin{equation*}
    \norm{\overrightarrow{g(t)}}=\norm{g(t)}_{H^{1}(\mathbb{R})}+\norm{\partial_{t}g(t)}_{L^{2}(\mathbb{R})}
\end{equation*}
to simplify our notation in the text.
\end{notation}
\begin{definition}
We define $S$ as the set $g \in L^{\infty}(\mathbb{R})$ such that
\begin{enumerate}
    \item $\frac{dg}{dx} \in L^{2}(\mathbb{R}),$
   \item $\int_{\mathbb{R}_{>0}}\md{g(x)-1}^{2}\,dx<\infty,$
   \item $\int_{\mathbb{R}_{<0}}\md{g(x)+1}^{2}\,dx<\infty.$
\end{enumerate}
\end{definition}
The partial differential equation \eqref{nlww} is locally well-posed in the affine space $S\times L^{2}(\mathbb{R}).$ 
Motivated by the proof and computations that we are going to present, we also consider:
\begin{definition}
We define for $x_{1},\,x_{2}\in \mathbb{R}$
\begin{center}
$H^{x_{2}}_{0,1}(x)\coloneqq H_{0,1}(x-x_{2})$ and $H^{x_{1}}_{-1,0}(x)\coloneqq H_{-1,0}(x-x_{1}),$
\end{center}
and we say that $x_{2}$ is the kink center of $H^{x_{2}}_{0,1}(x)$ and $x_{1}$ is the kink center of $H^{x_{1}}_{-1,0}(x).$
\end{definition}
\begin{remark}\label{R1}
Indeed, $S=\{g \in L^{\infty}(\mathbb{R})| \,g-H_{0,1}-H_{-1,0} \in  H^{1}(\mathbb{R})\}.$
\end{remark}
There are also non-stationary solutions $(\phi(t,x),\partial_{t}\phi(t,x))$ of \eqref{nlww} with finite total energy $E_{total}(\phi(t),\partial_{t}\phi(t))$ that satisfies for all $t\in\mathbb{R}$ 
\begin{equation}\label{66}
    \lim_{x\to+\infty}\phi(t,x)=1,\,\lim_{x\to -\infty}\phi(t,x)=0.
\end{equation}
But, for any $a\in \mathbb{R},$ the kinks $H_{0,1}(x-a)$ are the unique functions that minimize the Potential Energy in the set of functions $\phi(x)$ satisfying condition \eqref{66}, the proof of this fact follows from the Bogomolny identity, see \cite{solitonss} or section $2$ of \cite{jkl}. By a similar reasoning, we can verify that all functions $\phi(x)\in S$ have $E_{pot}(\phi)>2E_{pot}(H_{0,1}).$ \begin{definition}
We define the energy excess $\epsilon$ of a solution $(\phi(t),\partial_{t}\phi(t))\in S\times L^{2}(\mathbb{R})$ as the following value
\begin{equation*}
   \epsilon =E_{total}(\phi(t),\partial_{t}\phi(t))-2E_{pot}(H_{0,1}).
\end{equation*}
\end{definition}
Also, for $\phi(t)$ solution of \eqref{nlww}, we denote the Kinetic Energy of $\phi(t)$ by $E_{kin}(\phi(t))=E(\phi,\partial_{t}\phi)-E_{pot}(\phi(t))$.
We recall the notation $(x)_{+}\coloneqq\max(x,0)$. It's not difficult to verify the following inequalities
\begin{enumerate}
    \item [(D1)] $\md{H_{0,1}(x)}\leq e^{-\sqrt{2}(-x)_{+}},$
    \item [(D2)] $\md{H_{-1,0}(x)}\leq e^{-\sqrt{2}(x)_{+}},$
    \item [(D3)] $\md{\dot H_{0,1}(x)}\leq\sqrt{2}e^{-\sqrt{2}(-x)_{+}},$
    \item [(D4)] $\md{\dot H_{-1,0}(x)}\leq\sqrt{2}e^{-\sqrt{2}(x)_{+}}.$
\end{enumerate}
Moreover, since
\begin{equation}\label{kinkequation}
    \ddot H_{0,1}(x)=\dot U(H_{0,1}(x)),
\end{equation}
we can verify by induction the following estimate
\begin{equation}\label{kinkestimate}
   \md{\frac{d^{k}H_{0,1}(x)}{d x^{k}}}\lesssim_{k} \min\Big(e^{-2\sqrt{2}x},\,e^{\sqrt{2}x}\Big)
\end{equation}
for all $k\in\mathbb{N}\setminus \{0\}.$
The following result is crucial in the framework of this material:
\begin{lemma}[Modulation Lemma] 
$\exists \, C_{0}\,,\delta_{0}>0$, such that if $0<\delta\leq \delta_{0}$, $x_{2},\,x_{1}$ are real numbers with $x_{2}-x_{1}\geq \frac{1}{\delta}$ and $g \in H^{1}(\mathbb{R})$ satisfies $\norm{g}_{H^{1}}\leq \delta$, then for $\phi(x)=H_{-1,0}(x-x_{1})+H_{0,1}(x-x_{2})+g(x)$, $\exists!\, y_{1},\, y_{2}$ such that for
\begin{equation*}
    g_{1}(x)=\phi(x)-H_{-1,0}(x-y_{1})-H_{0,1}(x-y_{2}),
\end{equation*}
the four following statements are true
\begin{enumerate}
    \item [1] $\langle g_{1},\, \partial_{x}H_{-1,0}(x-y_{1})\rangle_{L^{2}}=0,$
    \item [2] $\langle g_{1},\, \partial_{x}H_{0,1}(x-y_{2})\rangle_{L^{2}}=0,$
    \item [3] $\norm{g_{1}}_{H^{1}(\mathbb{R})}\leq C_{0}\delta,$
    \item [4] $\md{y_{2}-x_{2}}+\md{y_{1}-x_{1}}\leq C_{0}\delta.$
\end{enumerate}
We will refer the first and second statements as the orthogonality conditions of the Modulation Lemma.
\end{lemma}
\begin{proof}
See the Appendix section \ref{auxil}.
\end{proof}
\par Now, our main results are the following:
\begin{theorem}\label{T1}
$\exists\,C, \delta_{0}>0$, such that if $\epsilon< \delta_{0}$ and \begin{center}$(\phi(0),\partial_{t}\phi(0)) \in S\times L^{2}(\mathbb{R})$\end{center} 
with $E_{total}(\phi(0),\partial_{t}\phi(0))=2E_{pot}(H_{0,1})+\epsilon$, then there are $x_{2},x_{1} \in C^{2}(\mathbb{R})$  functions such that the unique global time solution $\phi(t,x)$ of \eqref{nlww} is given by
\begin{equation}\label{formula}
    \phi(t)=H_{0,1}(x-x_{2}(t))+ H_{-1,0}(x-x_{1}(t))+g(t),
\end{equation}
with $g(t)$ satisfying orthogonality conditions of the Modulation Lemma and \begin{center}
    $e^{-\sqrt{2}(x_{2}(t)-x_{1}(t))}+\max_{j\in\{1,2\}}\md{\ddot x_{j}(t)}+\max_{j\in\{1,2\}}\dot x_{j}(t)^{2} +\norm{(g(t),\partial_{t}g(t))}_{H^{1}\times L^{2}}^{2} \lesssim \epsilon.$
\end{center} Furthermore, we have that 
\begin{equation}\label{evolutionf1}
\norm{(g(t),\partial_{t}g(t))}_{H^{1}\times L^{2}}^{2}\leq C\Big[\norm{(g(0),\partial_{t}g(0))}_{H^{1}\times L^{2}}^{2}+\epsilon^{2}\Big]\exp\Big(\frac{C\epsilon^{\frac{1}{2}}|t|}{\ln{(\frac{1}{\epsilon})}}\Big) \text{ for all $t\in\mathbb{R}.$}
\end{equation}
\end{theorem}
\begin{remark}\label{hypot1}
In notation of the statement of Theorem \ref{T1}, for any $\delta>0,$ there is $0<K(\delta)<1$ such that if $0<\epsilon<K(\delta),$ $E_{total}(\phi(0),\partial_{t}\phi(0))=2E_{pot}(H_{0,1})+\epsilon,$ then we have that $\norm{(g(0),\partial_{t}g(0))}_{H^{1}\times L^{2}}<\delta$ and $x_{2}(0)-x_{1}(0)>\frac{1}{\delta},$ for the proof see Lemma \ref{Lgta} and Corollary \ref{remarkestimate} in the Appendix section \ref{auxil}.    
\end{remark}
\begin{remark}[Optimal decay.]\label{rll}
The result of Theorem \ref{T1} is optimal in the sense that for any function 
$r:\mathbb{R}_{+}\to \mathbb{R}_{+}$
with $\lim_{h\to 0}r(h)=0,$ there is a positive value $\delta(r)$ such that if $0<\epsilon<\delta(r)$ and $\norm{\overrightarrow{g(0)}}\leq r(\epsilon)\epsilon,$ then 
$\epsilon\lesssim\norm{\overrightarrow{g(t)}}$ for some $0<t=O\Big(\frac{\ln{(\frac{1}{\epsilon})}}{\epsilon^{\frac{1}{2}}}\Big).$ 
The proof of this fact is in the Appendix section \ref{opt}.
\end{remark}
\begin{remark}
From Remark \ref{rll}, we obtain that there is an $0<\delta_{0}$ such that if $0<\epsilon<\delta_{0},$ then for any $(\phi(0,x),\partial_{t}\phi(0,x))\in S\times L^{2}(\mathbb{R})$ with $E_{total}(\phi(0),\partial_{t}\phi(0))$ equals to $2E_{pot}(H_{0,1})+\epsilon,\,g(t,x)$ defined in identity \eqref{formula} satisfies
$\epsilon\lesssim\limsup\limits_{t\to +\infty}\norm{\overrightarrow{g(t)}},$ similarly we have that $\epsilon\lesssim\limsup\limits_{t\to -\infty}\norm{\overrightarrow{g(t)}}.$
The proof of this fact is in the Appendix section \ref{opt}.
\end{remark}
\begin{theorem}\label{trueTheo2}
 $\exists C,\delta_{0}>0,$ such that if $0<\epsilon<\delta_{0},\, (\phi(0),\partial_{t}\phi(0)) \in S\times L^{2}(\mathbb{R}),$ and $E_{total}(\phi(0),\partial_{t}\phi(0))=2E_{pot}(H_{0,1})+\epsilon,$
then there are $v_{1},\,v_{2} \in \mathbb{R}$ such that \begin{center}$ 
\begin{pmatrix}
    \phi(0) \\
    \partial_{t}\phi(0)
\end{pmatrix}= \begin{pmatrix}
    H_{0,1}(x-x_{2}(0))+H_{-1,0}(x-x_{1}(0))+g_{0}(x)\\
    v_{2}\partial_{x}H_{0,1}(x-x_{2}(0))+v_{1}\partial_{x}H_{-1,0}(x-x_{1}(0))+g_{1}(x)
\end{pmatrix}$\end{center}
with $g_{0}$ satisfying the orthogonality conditions of Modulation Lemma
\begin{align*}
    \left\langle \dot H_{0,1}(x-x_{2}(0)),\,g_{1}(x)\right\rangle_{L^{2}(\mathbb{R})}=-v_{2}\left\langle \ddot H_{0,1}(x-x_{2}(0)),\,g_{0}(x)\right\rangle_{L^{2}(\mathbb{R})},\\ \left\langle \dot H_{-1,0}(x-x_{1}(0)),\,g_{1}(x)\right\rangle_{L^{2}(\mathbb{R})}=-v_{1}\left\langle \ddot H_{-1,0}(x-x_{1}(0)),\,g_{0}(x)\right\rangle_{L^{2}(\mathbb{R})}
\end{align*}
and $\epsilon$ the energy excess of the solution $(\phi(t,x),\partial_{t}\phi(t,x))$ of \eqref{nlww}. Indeed, let the smooth functions $d_{1}(t),\,d_{2}(t)$ be defined by
\begin{align}\label{d1}
   d_{1}(t)=a+b t-\frac{1}{2\sqrt{2}}\ln{\Big(\frac{8}{v^{2}}\cosh{\big(\sqrt{2}vt+c\big)}^{2}\Big)},\\ \label{d2}
   d_{2}(t)=a+b t+\frac{1}{2\sqrt{2}}\ln{\Big(\frac{8}{v^{2}}\cosh{\big(\sqrt{2}vt+c\big)}^{2}\Big)},
\end{align}
such that $d_{j}(0)=x_{j}(0),\,\dot d_{j}(0)=-v_{j}$ for $j\in\{1,\,2\}.$ Let $d(t)=d_{2}(t)-d_{1}(t),$ then, for all $ t\in \mathbb{R}$
\begin{equation*}
    \md{z(t)-d(t)}\lesssim \min(\epsilon^{\frac{1}{2}}\md{t},\epsilon t^{2}),\, \md{\dot z(t)-\dot d(t)}\lesssim \epsilon \md{t},
\end{equation*}
moreover, we have the following estimates
\begin{align}\label{oded1}
\epsilon \max_{j\in \{1,\,2\}}\md{d_{j}(t)-x_{j}(t)}=O\left(\max\Big(\norm{\overrightarrow{g(0)}},\epsilon \Big)^{2}\ln{\Big(\frac{1}{\epsilon}\Big)}^{11}\exp\Big(\frac{C\epsilon^{\frac{1}{2}}\md{t}}{\ln{(\frac{1}{\epsilon})}}\Big)\right),\\\label{oded2}
  \epsilon^{\frac{1}{2}}\max_{j\in \{1,\,2\}}\md{\dot d_{j}(t)- \dot x_{j}(t)}=O\left(\max\Big(\norm{\overrightarrow{g(0)}},\epsilon\Big)^{2}\ln{\Big(\frac{1}{\epsilon}\Big)}^{11}\exp\Big(\frac{C\epsilon^{\frac{1}{2}}\md{t}}{\ln{(\frac{1}{\epsilon})}}\Big)\right).
\end{align}
\end{theorem}
\begin{remark}\label{obvious}
The proof of Theorem \ref{T1} and Theorem \ref{trueTheo2} for $t\leq 0$ is analogous to the proof for $t\geq 0,$ so we will only prove them for $t\geq 0.$
\end{remark}
Theorem \ref{T1} will be obtained as a consequence of Theorem \ref{trueTheo2}.
Clearly, from Theorem \ref{trueTheo2}, we can deduce the following corollary. 
\begin{corollary}\label{colo2}
With the same hypotheses as in Theorem \ref{trueTheo2}, we have that
\begin{equation*}
    \max_{j \in\{1,\,2\}}|\ddot d_{j}(t)-\ddot x_{j}(t)|=O\left(\max\Big(\norm{\overrightarrow{g(0)}},\epsilon\Big)\epsilon^{\frac{1}{2}}\exp\Big(\frac{C\epsilon^{\frac{1}{2}}\md{t}}{\ln{(\frac{1}{\epsilon})}}\Big)
    +\max\Big(\norm{\overrightarrow{g(0)}},\epsilon\Big)^{2}\ln{\Big(\frac{1}{\epsilon}\Big)}^{11}\exp\Big(\frac{C\epsilon^{\frac{1}{2}}\md{t}}{\ln{(\frac{1}{\epsilon})}}\Big)\right).
\end{equation*}
\end{corollary}
\begin{proof}[Proof of Corollary \ref{colo2}]
It follows directly from Theorem \ref{trueTheo2} and from Lemma \ref{LemmaLL} presented in the Appendix Section \ref{auxil}.
\end{proof}
\subsection{Resume of the proof}
In this subsection, we present how the article is organized and explain briefly the content of each section. \\
\textbf{Section 2.} In this section, we prove orbital stability of a perturbation of a sum of two kinks. Moreover, we prove that if the initial data $(\phi(0,x),\partial_{t}\phi(0,x))$ satisfies the hypothesis of Theorem \ref{T1}, then there are real functions $x_{1},\,x_{2}$ of class $C^{2}$ such that for all $t\geq 0$
\begin{align*}
     \norm{\phi(t,x)-H^{x_{2}(t)}_{0,1}-H^{x_{1}(t)}_{-1,0}}_{H^{1}(\mathbb{R})}\lesssim \epsilon^{\frac{1}{2}},\\
     \norm{\partial_{t}\left(\phi(t,x)-H^{x_{2}(t)}_{0,1}-H^{x_{1}(t)}_{-1,0}\right)}_{L^{2}(\mathbb{R})}\lesssim \epsilon^{\frac{1}{2}}.
\end{align*}
The proof of the orbital stability follows from studying the expression
\begin{equation*}
    E_{pot}(H^{x_{2}(t)}_{0,1}+H^{x_{1}(t)}_{-1,0}+g)-E_{pot}(H^{x_{2}(t)}_{0,1}+H^{x_{1}(t)}_{-1,0}),
\end{equation*}
which is bigger than $\norm{\overrightarrow{g(t)}}^{2}$ less some remaining terms from Taylor's Expansion Theorem and the fact that the kinks are critical points of $E_{pot}.$ But, from the modulation lemma, we will introduce the functions $x_{2},\,x_{1}$ that will guarantee the following coercitivity property
\begin{equation*}
     \norm{\overrightarrow{g(t)}}^{2}\lesssim E_{pot}(H^{x_{2}(t)}_{0,1}+H^{x_{1}(t)}_{-1,0}+g)-E_{pot}(H^{x_{2}(t)}_{0,1}+H^{x_{1}(t)}_{-1,0}).
\end{equation*}
From the orthogonality conditions of the modulation lemma and standard ordinary differential equation techniques, we also obtain uniform bounds for $\norm{\dot x_{j}(t)}_{L^{\infty}(\mathbb{R})},\,\norm{\ddot x_{j}(t)}_{L^{\infty}(\mathbb{R})}$ for $j\in\{1,\,2\}.$ The main techniques of this section are an adaption of section 2 and 3 of \cite{jkl}.\\
\textbf{Section 3.}
In this section, we study the long time behavior of $\dot x_{j}(t),\,x_{j}(t)$ for $j\in\{1,\,2\}.$ More precisely, we elaborate a Lemma similar to the Lemma $3.5$ of \cite{jkl}, but our estimates are more precise, more precisely the errors of our estimate are written in function of $z(t),\,\dot x_{j}(t),\,\dot x_{j}(t)$ and $\norm{\overrightarrow{g(t)}}.$\\ 
\textbf{Section 4.}
In Section $4,$ we introduce a functional $F(t)$ with the objective of controlling $\norm{\overrightarrow{g(t)}}$ for a long time interval. More precisely, we show that the function $F(t)$ satisfies for a constant $K>0$ the global estimate $\norm{\overrightarrow{g(t)}}^{2}\lesssim F(t)+K\epsilon^{2}$ and we show that $|\dot F(t)|$ is small enough for a long time interval. We start the functional from the quadratic part of the total energy of $\phi(t),$ more precisely with
\begin{equation*}
    D(t)=\norm{\partial_{t}g(t,x)}_{L^{2}(\mathbb{R})}^{2}+\norm{\partial_{x}g(t,x)}_{L^{2}(\mathbb{R})}^{2}+\int_{\mathbb{R}}\ddot U(H^{x_{2}(t)}_{0,1}(x)+H^{x_{1}(t)}_{-1,0}(x))g(t,x)^{2}\,dx.
\end{equation*}
However, we obtain that the terms of worst decay that appear in the computation of $\dot D(t)$ are expressions similar to
\begin{equation*}
    \int_{\mathbb{R}}\partial_{t}g(t,x)F(x_{1},x_{2},\dot x_{1},\dot x_{2},x)\,dx.
\end{equation*}
But, we can cancel these bad terms after we add to the functional $D(t)$ correction terms similar to
\begin{equation*}
    -\int_{\mathbb{R}}g(t,x)F(x_{1},x_{2},\dot x_{1},\dot x_{2},x)\,dx,
\end{equation*}
and now in the time derivative of $D(t)$ plus the correction terms, we obtain an expression with size smaller or equivalent to
\begin{equation*}
    \norm{\overrightarrow{g(t)}}\norm{\partial_{t}(F(x_{1},x_{2},\dot x_{1},\dot x_{2},x))}_{L^{2}_{x}(\mathbb{R})}\max_{j\in{1,2}}\md{\dot x_{j}(t)}.
\end{equation*}
Finally, based on the correction term described in the proof of Lemma 4.2 of \cite {jkl}, we aggregate another kind of correction term such that its time derivative cancels with 
\begin{equation*}
    -\int_{\mathbb{R}} U^{(3)}(H^{x_{2}(t)}_{0,1}(x)+H^{x_{1}(t)}_{-1,0}(x))(\dot x_{2}(t)\partial_{x}H^{x_{2}(t)}_{0,1}+\dot x_{1}(t)\partial_{x}H^{x_{1}(t)}_{-1,0})g(t,x)^{2},
\end{equation*}
and then we evaluate the time derivative of the functional obtained from this sum D(t) with all the corrections terms.
\\ \textbf{Remaining Sections.} In the remaining part of this paper, we prove our main results, Theorem \ref{T1} is a consequence of the energy estimate obtained in Section 4 and the estimates with higher precision of the modulations parameters $x_{1}(t),\,x_{2}(t)$ which are obtained in Section $5.$ In Section 5, we prove the result of Theorem \ref{trueTheo2}, where we study the evolution of the precision of the modulation parameters estimates by comparing it with a solution of a system of ordinary differential equations. Complementary information are given in Appendices \ref{auxil} and \ref{opt}.  
\section{Global Stability of two moving kinks}
\par Before the presentation of the proof of the main theorem, we define a functional to study the potential energy of a sum of two kinks.  
\begin{definition}
The function $A:\mathbb{R}_{+}\to \mathbb{R}$ is defined by
\begin{equation}
    A(z)\coloneqq E_{pot}(H_{0,1}^{z}(x)+H_{-1,0}(x)).
\end{equation}
\end{definition}
The study of the function $A$ is essential to obtain global in time control of the norm of the remainder $g$ and the lower bound of $x_{2}(t)-x_{1}(t)$ in Theorem \ref{T1}.
\begin{remark}
It's easy to verify that
$
    E_{pot}(H_{0,1}(x-x_{2})+H_{-1,0}(x-x_{1}))=E_{pot}(H_{0,1}(x-(x_{2}-x_{1}))+H_{-1,0}(x)).
$
\end{remark}
We will use several times the following elementary estimate from the Lemma 2.5 of \cite{jkl} given by:
\begin{lemma}\label{interact}
For any real numbers $x_{2},x_{1}$, such that $x_{2}-x_{1}>0$ and $\alpha,\,\beta>0$ with $\alpha\neq \beta$ the following bound holds:
\begin{equation*}
    \int_{\mathbb{R}} e^{-\alpha(x-x_{1})_{+}}e^{-\beta(x_{2}-x)_{+}}\lesssim_{\alpha,\beta} e^{-\min(\alpha,\beta)(x_{2}-x_{1})},
\end{equation*}
For any $\alpha>0$, the following bound holds
\begin{equation*}
    \int_{\mathbb{R}} e^{-\alpha(x-x_{1})_{+}}e^{-\alpha(x_{2}-x)_{+}}\lesssim_{\alpha}(1+(x_{2}-x_{1})) e^{-\alpha(x_{2}-x_{1})}.
\end{equation*}
\end{lemma}
The main result of this section is the following 
\begin{lemma}\label{LL.1}
 The function $A$ is of class $C^{2}$ and there is a constant $C>0$, such that
\begin{enumerate}
    \item  $\md{\ddot A(z)-4\sqrt{2} e^{-\sqrt{2}z}}\leq Cze^{-2\sqrt{2}z},$
    \item $\md{\dot A(z)+4e^{-\sqrt{2}z}}\leq  Cze^{-2\sqrt{2}z},$
    \item  $\md{A(z)-2E_{pot}(H_{0,1})-2\sqrt{2}e^{-\sqrt{2}z}}\leq  Cze^{-2\sqrt{2}z}.$
\end{enumerate}
\end{lemma}
\begin{proof}
By definition of $A$, it's clear that
\begin{align*}
    A(z)&=\frac{1}{2}\int_{\mathbb{R}} \Big(\partial_{x}\big[H^{z}_{0,1}(x)+H_{-1,0}(x)\big]\Big)^{2}\,dx+\int_{\mathbb{R}} U(H^{z}_{0,1}(x)+H_{-1,0}(x))\,dx\\
    &=\norm{\partial_{x}H_{0,1}}_{L^{2}(\mathbb{R})}^{2}+\int_{\mathbb{R}} \partial_{x}H^{z}_{0,1}(x)\partial_{x}H_{-1,0}(x)\,dx+\int_{\mathbb{R}}U(H^{z}_{0,1}(x)+H_{-1,0}(x))\,dx.
\end{align*}
Since the functions $U$ and $H_{0,1}$ are smooth and $\partial_{x}H_{0,1}(x)$ has exponential decay when $|x|\to +\infty$, it's possible to differentiate $A(z)$ in $z.$ More precisely, we obtain 
\begin{align}
    \dot A(z)&=-\int_{\mathbb{R}}\partial^{2}_{x}H^{z}_{0,1}(x)\partial_{x}H_{-1,0}(x)\,dx-\int_{\mathbb{R}}\dot U(H^{z}_{0,1}(x)+H_{-1,0}(x))\partial_{x}H^{z}_{0,1}(x)\,dx\\\label{L.1}
    &=\int_{\mathbb{R}}\partial_{x}H^{z}_{0,1}(x)\big[\dot U(H_{-1,0})(x)-\dot U(H_{-1,0}(x)+H^{z}_{0,1}(x))\big]\,dx.
\end{align}
By similar reasons, it is always possible to differentiate $A(z)$ twice, precisely, we obtain
\begin{equation*}
 \ddot A(z)=\int_{\mathbb{R}}\partial_{x}H^{z}_{0,1}(x)^{2}\ddot U(H_{-1,0}(x)+H_{0,1}^{z}(x))
 -\partial^{2}_{x}H^{z}_{0,1}(x)\big[\dot U(H_{-1,0}(x))-\dot U\big(H_{-1,0}(x)+H^{z}_{0,1}(x))\big]\,dx.
\end{equation*}
Then, integrating by parts, we obtain 
\begin{equation}\label{IMP}
    \ddot A(z)=\int_{\mathbb{R}}\partial_{x}H^{z}_{0,1}(x)\partial_{x}H_{-1,0}(x)\big[\ddot U(H_{-1,0}(x))-\ddot U(H_{-1,0}(x)+H^{z}_{0,1}(x))\big]\,dx.
\end{equation}
Now, consider the function
\begin{equation}\label{M1}
    B(z)=\int_{\mathbb{R}}\partial_{x}H_{0,1}(x)\partial_{x}H_{-1,0}(x+z)\big[\ddot U(0)-\ddot U(H_{0,1}(x))\big]\,dx.
\end{equation}
Then, we have
\begin{equation}\label{c2}
    \md{\ddot A(z)-B(z)}=\md{\int_{\mathbb{R}}\partial_{x}H^{z}_{0,1}(x)\partial_{x}H_{-1,0}(x)\big[\big(\ddot U(H_{-1,0}(x))-\ddot U(H_{-1,0}(x)+H^{z}_{0,1}(x))\big)
    -\big(\ddot U(0)-\ddot U(H_{0,1}^{z}(x))\big)\big]\,dx}.
\end{equation}
Also, it's not difficult to verify the following identity
\begin{multline}\label{c3}
    \big[\ddot U(H_{-1,0}(x))-\ddot U(H_{-1,0}(x)+H^{z}_{0,1}(x))\big]-\big[\ddot U(0)-\ddot U(H^{z}_{0,1}(x))\big]=-\int_{0}^{H_{-1,0}(x)}\int_{0}^{H_{0,1}^{z}(x)}U^{(4)}(\omega_{1}+\omega_{2})\,d\omega_{1}\,d\omega_{2}.
\end{multline}
So, the identities \eqref{c3} and \eqref{c2} imply the following inequality
\begin{equation}
    \md{\ddot A(z)-B(z)}\leq\int_{\mathbb{R}} \md{\partial_{x}H^{z}_{0,1}(x)\partial_{x}H_{-1,0}(x)} \md{\int_{0}^{H_{-1,0}(x)}\int_{0}^{H^{z}_{0,1}(x)}|U^{(4)}(\omega_{1}+\omega_{2})\,d\omega_{1}\,d\omega_{2}}\,dx.
\end{equation}
Since $U$ is smooth and $\norm{H_{0,1}}_{L^{\infty}}=1$, we have that there is a constant $C>0$ such that
\begin{equation}\label{c4}
    \md{\ddot A(z)-B(z)}\leq C\int_{\mathbb{R}} \md{\partial_{x}H^{z}_{0,1}(x)\partial_{x}H_{-1,0}(x)H_{-1,0}(x)H^{z}_{0,1}(x)}\,dx.
\end{equation}
Now, Using the inequalities from (D1) to (D4) and Lemma \ref{interact} to the above inequality \eqref{c4}, we obtain that exist a constant $C_{1}$ non dependent of $z$ such that
\begin{equation}\label{c5}
    \md{\ddot A(z)-B(z)}\leq C_{1}z e^{-2\sqrt{2}z}.
\end{equation}
Also, it's not difficult to verify that the estimate
\begin{equation}
    \md{\partial_{x}H_{-1,0}(x)-\sqrt{2}e^{-\sqrt{2}x}}\leq C \min (e^{-3\sqrt{2}x},e^{-\sqrt{2}x}).
\end{equation}
and the identity \eqref{M1} imply the inequality
\begin{multline}
    \md{B(z)-\sqrt{2}e^{-\sqrt{2}z}\int_{\mathbb{R}}e^{-\sqrt{2}x}\partial_{x}H_{0,1}(x)(\ddot U(0)-\ddot U(H_{0,1}(x)))\,dx}\lesssim
    \int_{\mathbb{R}} \md{\partial_{x}H_{0,1}(x)}\min \big(e^{-3\sqrt{2}(x+z)},e^{-\sqrt{2}(x+z)}\big)\,dx\\
    \lesssim \int_{\mathbb{R}}e^{-2\sqrt{2}(-x)_{+}}\min \big(e^{-3\sqrt{2}(x+z)},e^{-\sqrt{2}(x+z)}\big)\,dx
    \lesssim \int_{-\infty}^{0}e^{-2\sqrt{2}(z-x)_{+}}e^{-\sqrt{2}x}\,dx+\int_{0}^{+\infty}e^{-2\sqrt{2}(z-x)_{+}}e^{-3\sqrt{2}(x)_{+}}\,dx.
\end{multline}
Since, we have the following identity and an estimate from Lemma \ref{interact}
\begin{align}
    \int_{-\infty}^{0}e^{-2\sqrt{2}(z-x)}e^{-\sqrt{2}x}\,dx=\frac{e^{-2\sqrt{2}z}}{\sqrt{2}},\\
    \int_{0}^{+\infty}e^{-2\sqrt{2}(z-x)_{+}}e^{-3\sqrt{2}(x)_{+}}\lesssim e^{-2\sqrt{2}z},
\end{align}
we obtain, then:
\begin{equation}
    \md{B(z)-\sqrt{2}e^{-\sqrt{2}z}\int_{\mathbb{R}}e^{-\sqrt{2}x}\partial_{x}H_{0,1}(x)\big[\ddot U(0)-\ddot U(H_{0,1}(x))\big]\,dx}\lesssim e^{-2\sqrt{2}z},
\end{equation}
which clearly implies with \eqref{c5} the inequality
\begin{equation}\label{central}
   \md{\ddot A(z)-\sqrt{2}e^{-\sqrt{2}z}\int_{\mathbb{R}}e^{-\sqrt{2}x}\partial_{x}H_{0,1}(x)\big [\ddot U(0)-\ddot U(H_{0,1}(x))\big]\,dx}\lesssim ze^{-2\sqrt{2}z}.
\end{equation}
Also we have the identity 
\begin{equation}\label{Kenergy}
    \int_{\mathbb{R}}\big (8(H_{0,1}(x))^{3}-6(H_{0,1}(x))^{5}\big )e^{-\sqrt{2}x}\,dx=2\sqrt{2},
\end{equation}
for the proof consult the Appendix \ref{auxil}.
Also, since  we have the identity $\ddot U(0)-\ddot U(\phi)=24 \phi^{2}-30 \phi^{4}$, by integration by parts, we obtain 
\begin{equation}
    \int_{\mathbb{R}} \frac{e^{-\sqrt{2}x}}{\sqrt{2}}\partial_{x}H_{0,1}(x)\big [\ddot U(0)-\ddot U(H_{0,1}(x))\big ]\,dx=\int_{\mathbb{R}}\big (8(H_{0,1}(x))^{3}-6(H_{0,1}(x))^{5}\big )e^{-\sqrt{2}x}\,dx.
\end{equation}
In conclusion, inequality \eqref{central} is equivalent to $\Big|\ddot A(z)-4\sqrt{2}e^{-\sqrt{2}z}\Big|\lesssim z e^{-2\sqrt{2}z}.$ The identities
\begin{align*}
    \dot U(\phi)+\dot U(\theta)-\dot U(\phi+\theta)=24\phi\theta(\phi+\theta)-6\Big(\sum_{j=1}^{4}\begin{pmatrix}
        5\\
        j
    \end{pmatrix}\phi^{j}\omega^{5-j}\Big),\\
    \dot A(z)=-\int_{\mathbb{R}}\partial_{x}H^{z}_{0,1}(x)\big [\dot U(H^{z}_{0,1}(x)+H_{-1,0}(x))+\dot U(H_{-1,0}(x))-\dot U(H^{z}_{0,1}(x))\big]\,dx
\end{align*}
 and Lemma \ref{interact} imply the following estimate for $z>0$
 \begin{equation*}
     \md{\dot A(z)}\lesssim e^{-\sqrt{2}z},
 \end{equation*}
so $\lim_{|z|\to +\infty} \md{\dot A(z)}=0$. In conclusion, integrating inequality $\md{\ddot A(z)-4\sqrt{2}e^{-\sqrt{2}z}}\lesssim z e^{-2\sqrt{2}z}$ from $z$ to $+\infty$ we obtain the second result of the lemma
\begin{equation}\label{MM}
   \md{\dot A(z)+4e^{-\sqrt{2}z}}\lesssim z e^{-2\sqrt{2}z}.
\end{equation}
\par Finally, from the fact that $\lim_{z\to +\infty} E_{pot}(H_{-1,0}+H^{z}_{0,1}(x))=2E_{pot}(H_{0,1})$, we obtain the last estimate integrating inequality \eqref{MM} from $z$ to $+\infty$, which is
\begin{equation*}
    \md{2E_{pot}(H_{0,1})+2\sqrt{2}e^{-\sqrt{2}z}-A(z)}\lesssim z e^{-2\sqrt{2}z}.
\end{equation*}
\end{proof}
It is not difficult to verify that the Fréchet derivative of $E_{pot}$ as a linear functional from $H^{1}(\mathbb{R})$ to $\mathbb{R}$ is given by
\begin{equation}\label{firstd}
    (DE_{pot}(\phi))(v)\coloneqq \int_{\mathbb{R}}\partial_{x}\phi(x)\partial_{x}v(x)+\dot U(\phi(x))v(x)\,dx.
\end{equation}
Also, it is not difficult to verify that for any $v,\, w \in H^{1}(\mathbb{R})$, we have
\begin{equation}
    \left\langle D^{2}E_{pot}(\phi)v,\,w\right\rangle_{L^{2}(\mathbb{R})}=\int_{\mathbb{R}}\partial_{x}v(x)\partial_{x}w(x)\,dx+\int_{\mathbb{R}}\ddot U(\phi(x))v(x)w(x)\,dx.
\end{equation}
\begin{lemma}[Coercitivity Lemma] $\exists\, C, c,\,\delta>0,$ such that if $x_{2}-x_{1}\geq\frac{1}{\delta}$, then for any $g \in H^{1}(\mathbb{R})$ we have 
\begin{equation}\label{coerc}
    \left\langle D^{2}E_{pot}(H^{x_{2}}_{0,1}+H^{x_{1}}_{-1,0})g,\,g\right\rangle_{L^{2}(\mathbb{R})}\geq c\norm{g}_{H^{1}(\mathbb{R})}^{2}-C\left[ \langle g,\,\partial_{x}H^{x_{1}}_{-1,0}\rangle^{2}+\langle g,\, \partial_{x}H^{x_{2}}_{0,1} \rangle^{2}\right].
\end{equation}
\end{lemma}
\par The proof of this lemma is based in the proof of Lemma 2.4 from \cite{jkl}. To prove the Coercitivity Lemma, we need the following result about the spectrum and kernel of the operators $D^{2}E_{pot}(H^{x_{1}}_{-1,0}),\,  D^{2}E_{pot}(H^{x_{2}}_{0,1}).$
\begin{lemma}\label{L}
The operators $ D^{2}E_{pot}(H^{x_{1}}_{-1,0}),\,  D^{2}E_{pot}(H^{x_{2}}_{0,1})$ satisfy the following properties
\begin{enumerate}
    \item [1] $\ker D^{2}E_{pot}(H^{x_{1}}_{-1,0}) = \left \{c\partial_{x}H^{x_{1}}_{-1,0},\,c\in \mathbb{C} \right \},\,\ker D^{2}E_{pot}(H^{x_{2}}_{0,1}) = \left \{c\partial_{x}H^{x_{2}}_{0,1},\,c\in \mathbb{C} \right \}$
    \item [2] $\sigma(D^{2}E_{pot}(H^{x_{2}}_{0,1}))=\sigma(D^{2}E_{pot}(H^{x_{1}}_{-1,0}))\subset \{0\}\cup[\lambda_{1},+\infty)$,  with $\lambda_{1}>0.$
\end{enumerate}
\end{lemma}
\begin{proof}
\par Since the operators $D^{2}E_{pot}(H^{x_{1}}_{-1,0})$ and $D^{2}E_{pot}(H^{x_{1}}_{-1,0})$ are equivalent by reflection and translation, they have the same spectrum that $D^{2}E_{pot}(H_{-1,0}).$ So, we'll just analyse the spectrum of the operator
\begin{equation}\label{Hess1}
    D^{2}E_{pot}(H^{x_{1}}_{-1,0})=-\partial^{2}_{x}+\ddot U(H^{x_{1}}_{-1,0}).
\end{equation}
 Also, we will only study the kernel of $D^{2}E_{pot}(H^{x_{1}}_{-1,0}),$ since the kernel of the other operator can be found by similar reasoning.
\par If we derive the Bogomolny equation satisfied by $H_{-1,0}$
\begin{equation}
    \partial^{2}_{x}H^{x_{1}}_{-1,0}(x)=\dot U(H^{x_{1}}_{-1,0}(x))
\end{equation}
with respect to $x,$ we obtain the identity
\begin{equation}
    \partial^{2}_{x}(\partial_{x}H^{x_{1}}_{-1,0}(x))=\ddot U(H^{x_{1}}_{-1,0}(x))\partial_{x}H^{x_{1}}_{-1,0}(x),
\end{equation}
which implies that $\partial_{x}H^{x_{1}}_{-1,0} \in \ker D^{2}E_{pot}(H^{x_{1}}_{-1,0})$. Also, $\partial_{x}H_{-1,0}(x)\neq 0$ for all $x\in\mathbb{R}$, so the Sturm-Liouville Oscillation Theory implies indeed that $0$ is the minimum element of the discrete spectrum of $D^{2}E_{pot}(H^{x_{1}}_{-1,0})$ and
\begin{equation}
    \ker D^{2}E_{pot}(H^{x_{1}}_{-1,0})=\left \{c\partial_{x}H^{x_{1}}_{-1,0},\,c\in\mathbb{C}\right \}.
\end{equation}
In conclusion, we have obtained that for some constant $\lambda_{1}>0$
\begin{equation}
    \sigma_{d}(D^{2}E_{pot}(H^{x_{1}}_{-1,0}))\subset \{0\}\cup [\lambda_{1},+\infty).
\end{equation}
By similar reasoning, we have  
\begin{align}
    \sigma_{d}(D^{2}E_{pot}(H^{x_{2}}_{0,1}))\subset \{0\}\cup [\lambda_{1},+\infty),\\
    \ker(D^{2}E_{pot}(H^{x_{2}}_{0,1}))=\left \{c\partial_{x}H^{x_{2}}_{0,1},\,c\in\mathbb{C}\right \}.
\end{align}
\par Now, it remains to estimate the lower bound of the essential spectrum of both operators. The main tool used to estimate the essential spectrum is a theorem of Spectral Theory written in the book \cite{spectral}.
\begin{theorem}\label{RC}
Suppose $A$ and $B$ are self-adjoint operators on a Hilbert Space $\mathcal{H}$. If $\exists z \in \mathbb{C}$, such that $(A-z)^{-1}-(B-z)^{-1}$ is compact, then $\sigma_{ess}(A)=\sigma_{ess}(B)$.
\end{theorem}
Since $D^{2}E_{pot}(H^{x_{1}}_{-1,0})=-\partial^{2}_{x}+(2-24(H^{x_{1}}_{-1,0})^{2}+30(H^{x_{1}}_{-1,0})^{4})$, we can rewrite this operator as
\begin{equation*}
   D^{2}E_{pot}(H^{x_{1}}_{-1,0})=\Big(-\partial^{2}_{x}+\big[2-24(H^{x_{1}}_{-1,0})^{2}+30(H^{x_{1}}_{-1,0})^{4}-2\chi_{[0,+\infty)}(x)-8\chi_{(-\infty,0)}(x)\big]\\
   +\big[2\chi_{[0,+\infty)}(x)+8\chi_{(-\infty,0)}(x)\big]\Big).
\end{equation*}
\par Now, we consider
\begin{equation*}
    T_{1}=-\partial^{2}_{x}+\big[2\chi_{[0,+\infty)}(x)+8\chi_{(-\infty,0)}(x)\big].
\end{equation*}
The next step is to check that for the self-adjoint operators $A=D^{2}E_{pot}(H^{x_{1}}_{-1,0})$, $B=T_{1}$ and for $z=-i$ all the hypothesis of Theorem \ref{RC} are fulfilled, which would imply that $\sigma_{ess}(D^{2}E_{pot}(H^{x_{1}}_{-1,0}))=\sigma_{ess}(T_{1}).$ 
\par Since we have the identity
\begin{multline}
    (D^{2}E_{pot}(H^{x_{1}}_{-1,0})+i)^{-1}-(T_{1}+i)^{-1}=
    -(D^{2}E_{pot}(H^{x_{1}}_{-1,0})+i)^{-1}\circ \Big(2-24(H^{x_{1}}_{-1,0})^{2}\\+30(H^{x_{1}}_{-1,0})^{4}-2\chi_{[0,+\infty)}(x)-8\chi_{(-\infty,0)}(x)\Big)\circ(T_{1}+i)^{-1},
\end{multline}
to prove that $D^{2}E_{pot}(H^{x_{1}}_{-1,0})$ and $T_{1}$ have same essential spectrum, we only need to verify that \begin{equation}
    T_{2}=\Big(2-24(H^{x_{1}}_{-1,0})^{2}+30(H^{x_{1}}_{-1,0})^{4}-2\chi_{[0,+\infty)}(x)-8\chi_{(-\infty,0)}(x)\Big)\circ(T_{1}+i)^{-1}
\end{equation}
is a compact operator on $L^{2}(\mathbb{R})$. By asymptotic properties of $H_{-1,0}$, it is not difficult to verify that \begin{equation*}
     Y=\Big(2-24(H^{x_{1}}_{-1,0})^{2}+30(H^{x_{1}}_{-1,0})^{4}-2\chi_{[0,+\infty)}(x)-8\chi_{(-\infty,0)}(x)\Big)
\end{equation*}
decays exponentially when $\md{x}$ goes to $+\infty$. Also, it is not difficult to verify that $(T_{1}+i)^{-1}$ is a bounded map from $L^{2}(\mathbb{R})$ to $H^{1}(\mathbb{R})\subset L^{\infty}(\mathbb{R})$. The last information and the Banach-Alaoglu Theorem imply that for any bounded sequence $(v_{n})\subset L^{2}(\mathbb{R})$, $\exists\, w \in H^{1}(\mathbb{R})$ and a subsequence that for simplicity we'll still denote by $(v_{n})$ such that
\begin{equation}
    (T_{1}+i)^{-1}(v_{n})\underset{H^{1}}{\rightharpoonup}w.
\end{equation}
Also, from the fact that $(T_{1}+i)^{-1}(v_{n})$ is uniformly bounded in $H^{1}(\mathbb{R})$, it can be verified that for any compact interval $K\subset \mathbb{R}$ that 
\begin{equation*}
(T_{1}+i)^{-1}(v_{n})\underset{L^{\infty}(K)}{\rightarrow}w,
\end{equation*}
and this fact with the exponential decay of $Y$ and $H^{1}(\mathbb{R})\subset L^{\infty}(\mathbb{R})$ implies directly the following convergence
\begin{equation}
    T_{2}(v_{n})\underset{L^{2}}{\rightarrow}Y(w),
\end{equation}
which implies that $T_{1}$ and $D^{2}E_{pot}(H^{x_{1}}_{-1,0})$ have the same essential spectrum, more precisely
\begin{equation}
    \sigma_{ess}\left(D^{2}E_{pot}(H^{x_{1}}_{-1,0})\right)\subset [2,+\infty),
\end{equation}
and so,
\begin{equation}
    \sigma_{ess}\left(D^{2}E_{pot}(H^{x_{2}}_{0,1})\right)\subset [2,+\infty).
\end{equation}
This finishes the proof of Lemma \ref{L}.
\end{proof}
Before starting the demonstration of the Coercitivity Lemma, let's consider from now on the function $0\leq\zeta\leq 1$ to be a smooth function satisfying:
\begin{equation}\label{C}
\zeta(x)=\begin{cases}
 1, \text{ if $x\leq \frac{3}{4}$,}\\
 0, \text{ if $x\geq \frac{4}{5}$.}
\end{cases}
\end{equation}
\begin{proof}[Proof of Coercitivity Lemma]
Our proof follows the scheme of proof of Lemma $2.4$ of \cite{jkl}. Here we denote $\langle,\,\rangle$ to be the scalar product on $L^{2}(\mathbb{R})$. First, because of Lemma \ref{L}, there is a $\lambda>0$ such that for any $v \in H^{1}(\mathbb{R})$
\begin{equation}\label{R13}
    \left \langle D^{2}E_{pot}(H^{x_{1}}_{-1,0})v,\,v\right\rangle \geq \lambda (\norm{v}_{L^{2}(\mathbb{R})} ^{2}-\norm{\partial_{x}H_{-1,0}}_{L^{2}}^{-2}\langle v,\,\partial_{x}H_{-1,0}^{x_{1}}\rangle^{2}).
\end{equation}
Also, because of the identity \eqref{Hess1}, we have
\begin{equation}\label{R14}
   \left \langle D^{2}E_{pot}(H^{x_{1}}_{-1,0})v,\,v\right \rangle= \norm{\partial_{x}v}_{L^{2}(\mathbb{R})}^{2}+\int_{\mathbb{R}}\ddot U(H^{x_{1}}_{-1,0}(x))(v(x))^{2}\,dx.
\end{equation}
Then, inequalities \eqref{R13} and \eqref{R14} imply for any $0<\theta<1$ that
\begin{equation}\label{R15}
       \left\langle D^{2}E_{pot}(H^{x_{1}}_{-1,0})v,\,v\right\rangle\geq  \theta\lambda \big[\norm{v}_{L^{2}(\mathbb{R})}^{2}-\norm{\partial_{x}H_{-1,0}}_{L^{2}}^{-2}\langle v,\,\partial_{x}H_{-1,0}^{x_{1}}\rangle^{2}\big]+(1-\theta)\big[\norm{\partial_{x}v}_{L^{2}(\mathbb{R})}^{2}\\+\int_{\mathbb{R}}\ddot U(H^{x_{2}}_{0,1}(x))(v(x))^{2}\,dx\big].
\end{equation}
Since $\norm{\ddot U(H_{-1,0}(x))}_{L^{\infty}}<\infty$, we can choose $\theta$ close enough to 1, and obtain from \eqref{R15} the following inequality for a positive constant $c>0$ such that
\begin{equation}\label{put1}
   \left\langle D^{2}E_{pot}(H^{x_{1}}_{-1,0})v,\,v\right\rangle \geq c \big[\norm{v}_{H^{1}(\mathbb{R})} ^{2}-\norm{\partial_{x}H_{-1,0}}_{L^{2}}^{-2}\langle v,\,\partial_{x}H_{-1,0}^{x_{1}}\rangle^{2}\big].
\end{equation}
By similar reasoning, we also have 
\begin{equation}\label{put2}
    \langle D^{2}E_{pot}(H^{x_{2}}_{0,1})v,\,v\rangle \geq c \big[\norm{v}_{H^{1}(\mathbb{R})} ^{2}-\norm{\partial_{x}H_{-1,0}}_{L^{2}}^{-2}\langle v,\,\partial_{x}H_{0,1}^{x_{2}}\rangle^{2}\big].
\end{equation}
 Now to study the operator $D^{2}E_{pot}(H^{x_{2}}_{0,1}+H^{x_{1}}_{-1,0})$, consider the function $\zeta_{1}(x)=\zeta(\frac{x-x_{1}}{x_{2}-x_{1}})$ and also 
\begin{align}
    V\coloneqq(\ddot U(H^{x_{1}}_{-1,0}+H^{x_{2}}_{0,1})-\ddot U(H^{x_{2}}_{0,1})),\\ D^{2}E_{pot}(H^{x_{2}}_{0,1}+H^{x_{1}}_{-1,0})=-\partial^{2}_{x}+\ddot U(H^{x_{2}}_{0,1})+V.
\end{align}
It can be verified that the support of $(1-\zeta_{1}(x))$ is included in $\{x \in \mathbb{R}|\,\frac{x-x_{1}}{x_{2}-x_{1}}\geq \frac{3}{4}\}$, and so
\begin{gather}
    (1-\zeta_{1}(x))\md{H^{x_{1}}_{-1,0}(x)}\leq e^{-\frac{3\sqrt{2}(x_{2}-x_{1})}{4}},\\\label{RR}
   \md{\langle V(x)(1-\zeta_{1}(x))v(x),\,(1-\zeta_{1}(x))v(x)\rangle}\lesssim e^{-\frac{3\sqrt{2}(x_{2}-x_{1})}{4}}\norm{v}_{L^{2}}^{2}.
\end{gather}
Therefore, if $\delta>0$ is small enough, from $x_{2}-x_{1}\geq \frac{1}{\delta}$, we obtain from the inequality \eqref{RR} that
\begin{equation}\label{put3}
   \md{\langle V(x)(1-\zeta_{1}(x))v(x),\,(1-\zeta_{1}(x))v(x)\rangle}\leq O(\delta)\norm{v}_{L^{2}}^{2},
\end{equation}
so the inequalities \eqref{put3} and \eqref{put2} imply for a $c>0$
\begin{multline}\label{fe1}
    \left  \langle D^{2}E_{pot}(H^{x_{1}}_{-1,0}+H^{x_{2}}_{0,1} )((1-\zeta_{1}(x))v),\,(1-\zeta_{1}(x))v\right\rangle \geq c \big[\norm{(1-\zeta_{1})v}_{H^{1}(\mathbb{R})} ^{2}-\norm{\partial_{x}H_{0,1}}_{L^{2}}^{-2}\langle (1-\zeta_{1})v,\,\partial_{x}H_{0,1}^{x_{2}}\rangle^{2}\big]\\-O(\delta)\norm{v}_{L^{2}(\mathbb{R})}^{2}.
\end{multline}
Also the support of $\zeta_{1}(x)$ is included $\{x|\,\frac{x-x_{2}}{x_{2}-x_{1}}\leq{-\frac{1}{5}}\}.$ So, by similar arguments, we can verify the analogous inequality
\begin{equation}\label{fe2}
    \left\langle D^{2}E_{pot}(H^{x_{1}}_{-1,0}+H^{x_{2}}_{0,1} )(\zeta_{1}(x)v),\,\zeta_{1}(x)v\right\rangle \geq c \big[\norm{\zeta_{1}v}_{H^{1}(\mathbb{R})} ^{2}-\norm{\partial_{x}H_{-1,0}}_{L^{2}}^{-2}\langle \zeta_{1}v,\,\partial_{x}H_{-1,0}^{x_{1}}\rangle^{2}\big]-O(\delta)\norm{v}_{L^{2}(\mathbb{R})}^{2}.
\end{equation}
Also, we obtain that there is a uniform constant $C>0$ such that if $\delta>0$ is small enough, then we obtain the following estimate for all $v \in H^{1}(\mathbb{R})$ 
\begin{equation}\label{IT1}
    \left\langle \partial_{x}(\zeta_{1}(x)v(x)),\, \partial_{x}\left[(1-\zeta_{1}(x))v(x)\right]\right\rangle\geq - C\delta\norm{v}_{H^{1}(\mathbb{R})}^{2}.
\end{equation}
Also, if $\delta>0$ is small enough, we have that $\ddot U(H_{-1,0}(x))>1$ for $\frac{3(x_{2}-x_{1})}{4}\leq x-x_{1}\leq \frac{4(x_{2}-x_{1})}{5}$. In conclusion, since the support of $(1-\zeta_{1}(x))\zeta_{1}(x)$ is included in $\{x-x_{1} \in [\frac{3(x_{2}-x_{1})}{4},\,\frac{4(x_{2}-x_{1})}{5}]\}$, we have the following inequality
\begin{align}\label{IT2}
    \int_{\mathbb{R}}\ddot U(H^{x_{1}}_{-1,0}(x))\zeta_{1}(x)(1-\zeta_{1}(x))(v(x))^{2}\,dx\geq 0.
\end{align}
Finally, from the mean value theorem, the knowledge of the support of $\zeta_{1}$ and the exponential decay of $H_{0,1}(x)$, we have that 
\begin{align}\label{IT3}
   \md{\int_{\mathbb{R}}\big[\ddot U(H^{x_{1}}_{-1,0}(x)+H^{x_{2}}_{0,1}(x))-\ddot U(H^{x_{1}}_{-1,0}(x))\big]\zeta_{1}(x)(1-\zeta_{1}(x))v(x)^{2}\,dx}
    \leq o(1)\norm{v}_{L^{2}}^{2}.
\end{align}
 Therefore, the inequalities \eqref{IT1}, \eqref{IT2} and \eqref{IT3} imply for a uniform constant $C>0$ that
\begin{equation}\label{IT4}
    \langle D^{2}E_{pot}(H^{x_{1}}_{-1,0}+H^{x_{2}}_{0,1})(\zeta_{1}v),\, (1-\zeta_{1})v\rangle\geq -C\delta\norm{v}_{L^{2}(\mathbb{R})}^{2}.
\end{equation}
Since we know that support of $\zeta_{1}$ is included in $\{x|\,\frac{x-x_{2}}{x_{2}-x_{1}}\leq{-\frac{1}{5}}\}$, we can deduce the estimate
\begin{equation}\label{II4}
    \md{\langle \partial_{x}H^{x_{2}}_{0,1},\,(1-\zeta_{1})v\rangle^{2}-\langle \partial_{x}H^{x_{2}}_{0,1},\,v\rangle^{2}}=\md{\langle \partial_{x} H^{x_{2}}_{0,1},\,\zeta_{1}v\rangle\langle \partial_{x}H^{x_{2}}_{0,1},\,(2-\zeta_{1})v\rangle}
    \leq O(\delta)\norm{v}_{L^{2}}^{2},
\end{equation}
and similarly, 
\begin{equation}\label{II5}
    \md{\langle \partial_{x}H^{x_{1}}_{-1,0},\,\zeta_{1}v\rangle^{2}-\langle \partial_{x}H^{x_{1}}_{-1,0},\,v\rangle^{2}}\leq O(\delta)\norm{v}_{L^{2}}^{2}.
\end{equation}
Therefore, we have that \eqref{fe1}, \eqref{fe2}, \eqref{IT4}, \eqref{II4} and \eqref{II5} imply the inequality \eqref{coerc} of the statement.
\end{proof}
\begin{lemma}\label{DEl2}
There is a constant $C_{2}$, such that if $x_{2}-x_{1}>0$, then
\begin{equation}\label{DE}
\norm{DE_{pot}(H^{x_{2}}_{0,1}+H^{x_{1}}_{-1,0})}_{L^{2}(\mathbb{R})}\leq C_{2}e^{-\sqrt{2}(x_{2}-x_{1})}.
\end{equation}
\begin{proof}
By the definition of the potential energy, the equation \eqref{kinkequation} and the exponential decay of the two kinks functions, we have that
\begin{equation*}
    DE_{pot}(H^{x_{2}}_{0,1}+H^{x_{1}}_{-1,0})=\dot U(H^{x_{2}}_{0,1}+H^{x_{1}}_{-1,0})-\dot U(H^{x_{2}}_{0,1})-\dot U(H^{x_{1}}_{-1,0})
\end{equation*}
as a bounded linear operator from $L^{2}(\mathbb{R})$ to $\mathbb{C}$. So, we have that
\begin{equation*}
    DE_{pot}(H^{x_{2}}_{0,1}+H^{x_{1}}_{-1,0})=-24H^{x_{2}}_{0,1}H^{x_{1}}_{-1,0}\big[H^{x_{2}}_{0,1}+H^{x_{1}}_{-1,0}\big]
    +6\Big[\sum_{j=1}^{4} \begin{pmatrix}
        5\\
        j
    \end{pmatrix}(H^{x_{1}}_{-1,0})^{j} (H^{x_{2}}_{0,1})^{5-j}\Big],
\end{equation*}
and, then, the conclusion follows directly from Lemma \ref{interact}, $(D1)$ and $(D2).$ 
\end{proof}
\end{lemma}
\begin{theorem}[Orbital Stability of a sum of two moving kinks]\label{Stab}
$\exists \delta_{0}>0$ such that if the solution $\phi$ of \eqref{nlww} satisfies $(\phi(0),\partial_{t}\phi(0))\in S\times L^{2}(\mathbb{R})$ and the energy excess $\epsilon=E(\phi)-2E_{pot}(H_{0,1})$ is smaller than $\delta_{0}$, then $\exists x_{1},\, x_{2}:\mathbb{R}\to\mathbb{R}$ functions of class $C^{2}$, such that for all $t \in \mathbb{R}$ denoting $g(t)=\phi(t)-H_{0,1}(x-x_{2}(t))+H_{-1,0}(x-x_{1}(t))$ and $z(t)=x_{2}(t)-x_{1}(t)$, we have:
\begin{enumerate}
    \item $\norm{g(t)}_{H^{1}(\mathbb{R})}=O(\epsilon^{\frac{1}{2}}),$\\
    \item $z(t)\geq \frac{1}{\sqrt{2}}[\ln{(\frac{1}{\epsilon})}+\ln{2}],$\\
    \item $\norm{\partial_{t}\phi(t)}_{L^{2}(\mathbb{R})}^{2}\leq 2\epsilon,$
    \item $\max_{j \in \{1,2\}}|\dot x_{j}(t)|^{2}+\max_{j \in \{1,2\}}|\ddot x_{j}(t)|=O(\epsilon).$
    \end{enumerate}
\end{theorem}
\begin{proof}
First, from the fact that $E_{total}(\phi(x))>2E_{pot}(H_{0,1}),$ we deduce, from the conservation of total energy, the estimate
\begin{equation}\label{edd1}
    \norm{\partial_{t}\phi(t)}_{L^{2}}^{2}\leq 2\epsilon.
\end{equation}
From Remark \ref{hypot1}, we can assume if $\epsilon\ll 1$ that there are $p_{1},\,p_{2} \in \mathbb{R}$ such that 
\begin{equation*}
    \phi(0,x)=H_{0,1}(x-p_{2})+H_{-1,0}(x-p_{1})+g_{1}(x),
\end{equation*}
such that
\begin{equation*}
    \norm{g_{1}}_{H^{1}(\mathbb{R})}<\delta,\,p_{2}-p_{1}>\frac{1}{\delta},
\end{equation*}
for a small constant $\delta>0$.
Since the equation \ref{nlww} is locally well-posed in the space $S\times L^{2}(\mathbb{R}),$ we conclude that there is a $\delta_{1}>0$ depending only on $\delta$ and $\epsilon$ such that if $-\delta_{1} \leq t \leq \delta_{1},$ then
\begin{equation}\label{boot1}
    \norm{\phi(t,x)-H_{0,1}(x-p_{2})-H_{-1,0}(x-p_{1})}_{H^{1}(\mathbb{R})}\leq 2\delta.
\end{equation}
If $\delta,\epsilon>0$ are small enough, then, from the inequality \eqref{boot1} and the Modulation Lemma, we obtain in the time interval $[-\delta_{1},\delta_{1}]$ the existence of modulations parameters $x_{1}(t),\,x_{2}(t)$ such that for
\begin{equation*}
    g(t)=\phi(t)-H_{0,1}(x-x_{2}(t))-H_{-1,0}(x-x_{1}(t)),
\end{equation*}
we have
\begin{align}\label{occ1}
\left\langle g(t),\,\partial_{x}H_{0,1}(x-x_{2}(t))\right\rangle_{L^{2}}=
\left\langle g(t),\,\partial_{x}H_{-1,0}(x-x_{1}(t))\right\rangle_{L^{2}}=0,\\\label{occ2}
\frac{1}{\md{x_{2}(t)-x_{1}(t)}}+\norm{g(t)}_{H^{1}}\lesssim \delta.
\end{align}
From now on, we denote $z(t)=x_{2}(t)-x_{1}(t).$ From the Energy Conservation Law, we have for $-\delta_{1} \leq t \leq \delta_{1}$ that
\begin{multline*}
    E(\phi(t))=2E_{pot}(H_{0,1})+\epsilon=\frac{\norm{\partial_{t}\phi(t)}_{L^{2}(\mathbb{R})}^{2}}{2}+E_{pot}\big(H^{x_{2}(t)}_{0,1}+H^{x_{1}(t)}_{-1,0}\big)+\big\langle DE_{pot}\big(H^{x_{2}(t)}_{0,1}+H^{x_{1}(t)}_{-1,0}\big),\,g(t)\big\rangle_{L^{2}(\mathbb{R})} \\+\frac{\big\langle D^{2}E_{pot}\big(H^{x_{2}(t)}_{0,1}+H^{x_{1}(t)}_{-1,0}\big)g(t),\,g(t)\big\rangle_{L^{2}(\mathbb{R})}}{2}+ O(\norm{g(t)}_{H^{1}}^{3}).
\end{multline*}
 From Lemma \ref{LL.1} and \eqref{occ2}, the above identity implies that
\begin{multline}\label{MMM1}
    \epsilon= \frac{\norm{\partial_{t}\phi(t)}_{L^{2}(\mathbb{R})}^{2}}{2}+2\sqrt{2}e^{-\sqrt{2}z(t)}+\big\langle DE_{pot}\big(H^{x_{2}(t)}_{0,1}+H^{x_{1}(t)}_{-1,0}\big),\,g(t)\big\rangle_{L^{2}(\mathbb{R})}
    +\frac{\big\langle D^{2}E_{pot}\big(H^{x_{2}(t)}_{0,1}+H^{x_{1}(t)}_{-1,0}\big)g(t),\,g(t)\big\rangle_{L^{2}(\mathbb{R})}}{2}\\+ O\Big(\norm{g(t)}_{H^{1}}^{3}+z(t)e^{-2\sqrt{2}z(t)}\Big)
\end{multline}
for $-\delta_{1}\leq t\leq \delta_{1}.$
From \eqref{DE}, it is not difficult to verify that
$\md{\langle DE_{pot}(H^{x_{2}(t)}_{0,1}+H^{x_{1}(t)}_{-1,0}),\,g(t)\rangle}\leq C_{2}e^{-\sqrt{2}z(t)}\norm{g(t)}_{H^{1}(\mathbb{R})}.$
So, the equation \eqref{MMM1} and the Coercitivity Lemma imply, while $-\delta_{1}\leq t\leq \delta_{1}$, the following inequality
\begin{equation}\label{RRR1}
    \epsilon +C_{2}e^{-\sqrt{2}z(t)}\norm{g(t)}_{H^{1}(\mathbb{R})}\geq \frac{\norm{\partial_{t}\phi(t)}_{L^{2}}^{2}}{2}+2\sqrt{2}e^{-\sqrt{2}z(t)}+\frac{c\norm{g(t)}_{H^{1}(\mathbb{R})}^{2}}{2}+O\big(\norm{g(t)}_{H^{1}}^{3}+z(t)e^{-2\sqrt{2}z(t)}\big).
\end{equation}
Finally, applying the Young Inequality in the term $C_{2}e^{-\sqrt{2}z(t)}\norm{g(t)}_{H^{1}(\mathbb{R})}$, we obtain that the inequality \eqref{RRR1} can be rewritten in the form
\begin{equation}\label{dess}
    \epsilon\geq \frac{\norm{\partial_{t}\phi(t)}_{L^{2}}^{2}}{2}+2\sqrt{2}e^{-\sqrt{2}z(t)}+\frac{c\norm{g(t)}_{H^{1}(\mathbb{R})}^{2}}{4}+O\big(\norm{g(t)}_{H^{1}}^{3}+z(t)e^{-2\sqrt{2}z(t)}+e^{-2\sqrt{2}z(t)}\big).
\end{equation}
Then, the estimates \eqref{dess}, \eqref{occ2} imply for $\delta>0$ small enough the following inequality
\begin{equation}\label{dess2}
   \epsilon\geq \frac{\norm{\partial_{t}\phi(t)}_{L^{2}}^{2}}{2}+2e^{-\sqrt{2}z(t)}+\frac{c\norm{g(t)}_{H^{1}(\mathbb{R})}^{2}}{8}.
\end{equation}
So, the inequality \eqref{dess2} implies the estimates
\begin{align}\label{edd2}
    e^{-\sqrt{2}z(t)}<\frac{\epsilon}{2},\\\label{edd3}
    \norm{g(t)}_{H^{1}(\mathbb{R})}^{2}\lesssim \epsilon,
\end{align}
for $t \in [-\delta_{1},\delta_{1}]$. In conclusion, if $\frac{1}{\delta}\lesssim \ln{(\frac{1}{\epsilon})}^{\frac{1}{2}},$ we can conclude by a bootstrap argument that
the inequalities \eqref{edd1}, \eqref{edd2}, \eqref{edd3} are true for all $t \in \mathbb{R}.$
More precisely, we study the set
\begin{equation*}
    C=\left\{b \in \mathbb{R}_{>0}|\,\epsilon\geq \frac{\norm{\partial_{t}\phi(t)}_{L^{2}}^{2}}{2}+2e^{-\sqrt{2}z(t)}+\frac{c\norm{g(t)}_{H^{1}(\mathbb{R})}^{2}}{8}\text{, if $\md{t}\leq b.$}\right\}
\end{equation*}
and prove that $M=\sup_{b \in C}b=+\infty.$ We already have checked that $C$ is not empty, also $C$ is closed by its definition. Now from the previous argument, we can verify that the set where inequality \eqref{dess2}  holds is open. So, by connectivity, we obtain that $C=\mathbb{R}_{>0}.$ \par In conclusion, it remains to prove that the modulation parameters $x_{1}(t),\,x_{2}(t)$ are of class $C^{2}$ and that the fourth item of the statement of Theorem \ref{Stab} is true. \\
\textbf{(Proof of the $C^{2}$ regularity of $x_{1},\, x_{2},$ and of the fourth item.)}
\par For $\delta_{0}>0$ small enough, we denote $(y_{1}(t),\,  y_{2}(t))$ to be the solution of the following system of ordinary differential equations, with the function $g_{1}(t)=\phi(t,x)-H^{y_{2}(t)}_{0,1}(x)-H^{y_{1}(t)}_{-1,0}(x)$,
\begin{equation}\label{eqq1}
\Big(\norm{\partial_{x}H_{0,1}}_{L^{2}}^{2}-\left\langle g_{1}(t),\,\partial^{2}_{x}H^{y_{1}(t)}_{-1,0}\right\rangle\Big)\dot y_{1}(t)+
\Big(\left\langle \partial_{x}H^{y_{2}(t)}_{0,1},\,\partial_{x}H^{y_{1}(t)}_{-1,0}\right\rangle\Big)\dot y_{2}(t)= -\left\langle \partial_{t}\phi(t),\,\partial_{x}H^{y_{1}(t)}_{-1,0}(x)\right\rangle,
\end{equation}
\begin{equation}\label{eqq2}
\Big(\left\langle \partial_{x}H^{y_{2}(t)}_{0,1},\,\partial_{x}H^{y_{1}(t)}_{-1,0}\right\rangle\Big)\dot y_{1}(t)+\Big(\norm{\partial_{x}H_{0,1}(t)}_{L^{2}}^{2}-\left\langle g_{1}(t),\,\partial^{2}_{x}H^{y_{2}}_{0,1}\right\rangle\Big)\dot y_{2}(t)=-\left\langle \partial_{t}\phi(t),\,\partial_{x}H^{y_{2}(t)}_{0,1}(x)\right\rangle,
\end{equation}
with initial condition $(y_{2}(0),y_{1}(0))=(x_{2}(0),x_{1}(0))$. 
This ordinary differential equation system is motivated from the time derivative of the orthogonality conditions of the Modulation Lemma.
\par Since we have the estimate $\ln{(\frac{1}{\epsilon})}\lesssim x_{2}(0)-x_{1}(0)$ and $g_{1}(0)=g(0),$ Lemma \ref{interact} and the inequality \eqref{edd3} imply that the matrix \begin{equation}\label{Matrix}
\begin{bmatrix}
 \norm{\partial_{x}H_{0,1}}_{L^{2}}^{2}-\left\langle g_{1}(0),\,\partial^{2}_{x}H^{y_{1}(0)}_{-1,0}\right\rangle & \left\langle \partial_{x}H^{y_{2}(0)}_{0,1},\,\partial_{x}H^{y_{1}(0)}_{-1,0}\right\rangle\\
 \left\langle \partial_{x}H^{y_{2}(0)}_{0,1},\,\partial_{x}H^{y_{1}(0)}_{-1,0}\right\rangle & \norm{\partial_{x}H_{0,1}}_{L^{2}}^{2}-\left\langle g_{1}(0),\,\partial^{2}_{x}H^{y_{2}}_{0,1}\right\rangle
\end{bmatrix}
\end{equation}
is positive, so we have from Picard-Lindelöf Theorem that $(y_{2}(t),y_{1}(t))$ are of class $C^{1}$ for some interval $[-\delta,\delta],$ with $\delta>0$ depending on $\md{x_{2}(0)-x_{1}(0)}$ and $\epsilon.$ From the fact that $(y_{2}(0),y_{1}(0))=(x_{2}(0),x_{1}(0)),$ we obtain, from the equations \eqref{eqq1} and \eqref{eqq2}, that $(y_{2}(t),y_{1}(t))$ also satisfies the orthogonality conditions of Modulation Lemma for $t \in [-\delta,\delta].$ In conclusion, the uniqueness of Modulation Lemma implies that $(y_{2}(t),y_{1}(t))=(x_{2}(t),x_{1}(t))$ for $t \in [-\delta,\delta].$ From this argument, we also have for $t \in [-\delta,\delta]$ that
$e^{-\sqrt{2}(y_{2}(t)-y_{1}(t))}\leq \frac{\epsilon}{2\sqrt{2}}.$
By bootstrap, we can show, repeating the argument above, that 
\begin{equation}
    \sup\left\{C>0|\,(y_{2}(t),y_{1}(t))=(x_{2}(t),x_{1}(t)), \text{for $t \in [-C,C]$}\right\}=+\infty.
\end{equation}
Also, the argument above implies that if $(y_{1}(t),y_{2}(t))=(x_{1}(t),x_{2}(t))$ in an instant $t,$ then $y_{1},\,y_{2}$ are of class $C^{1}$ in a neighborhood of $t.$ In conclusion, $x_{1}, \, x_{2}$ are functions in $C^{1}(\mathbb{R}).$
Finally, since $\norm{g(t)}_{H^{1}}=O(\epsilon^{\frac{1}{2}})$ and $e^{-\sqrt{2}z(t)}=O(\epsilon),$ the following matrix  
\begin{equation}\label{Matrix1}
M(t)\coloneqq
\begin{bmatrix}
\norm{\partial_{x}H_{0,1}}_{L^{2}}^{2}-\left\langle g(t),\,\partial^{2}_{x}H^{x_{1}(t)}_{-1,0}\right\rangle &  \left\langle \partial_{x}H^{x_{2}(t)}_{0,1},\,\partial_{x}H^{x_{1}(t)}_{-1,0}\right\rangle\\
 \left\langle \partial_{x}H^{x_{2}(t)}_{0,1},\,\partial_{x}H^{x_{1}(t)}_{-1,0}\right\rangle & \norm{\partial_{x}H_{0,1}}_{L^{2}}^{2}-\left\langle g(t),\,\partial^{2}_{x}H^{x_{2}(t)}_{0,1}\right\rangle
\end{bmatrix}    
\end{equation}
is uniformly positive for all $t \in \mathbb{R}.$ So, from the estimate $\norm{\partial_{t}\phi(t)}_{L^{2}(\mathbb{R})}=O(\epsilon^{\frac{1}{2}}),$ the identities $x_{j}(t)=y_{j}(t)$ for $j=1,2$ and the equations \eqref{eqq1} and \eqref{eqq2}, we obtain
\begin{equation}\label{derivada1}
    \max_{j \in \{1,2\}}\md{\dot x_{j}(t)}=O(\epsilon^{\frac{1}{2}}).
\end{equation}
\par Since the matrix $M(t)$ is invertible for any $t\in\mathbb{R},$ we can obtain from the equations \eqref{eqq1}, \eqref{eqq2} that the functions $\dot x_{1}(t),\dot x_{2}(t)$ are given by
\begin{equation}\label{c2eq}
    \begin{bmatrix}
    \dot x_{1}(t)\\
    \dot x_{2}(t)
    \end{bmatrix}
    =
    M(t)^{-1}
    \begin{bmatrix}
    -\left\langle \partial_{t}\phi(t),\,\partial_{x}H^{x_{1}(t)}_{-1,0}(x)\right\rangle\\
    -\left\langle \partial_{t}\phi(t),\,\partial_{x}H^{x_{2}(t)}_{0,1}(x)\right\rangle
    \end{bmatrix}
.\end{equation}
Now, since we have that $(\phi(t),\partial_{t}\phi(t))\in C(\mathbb{R},S\times L^{2}(\mathbb{R}))$ and $x_{1}(t),\,x_{2}(t)$ are of class $C^{1},$ we can deduce that $(g(t),\partial_{t}g(t))\in C(\mathbb{R},H^{1}(\mathbb{R})\times L^{2}(\mathbb{R})).$ So, by definition, we can verify that $M(t) \in C^{1}(\mathbb{R},\mathbb{R}^{4}).$
\par Also, since $\phi(t,x)$  is the solution in distributional sense of \eqref{nlww}, we have that for any $y_{1},\,y_{2} \in \mathbb{R}$ the following identities hold
\begin{align*}
    \big\langle \partial_{x}H^{y_{2}}_{0,1},\,\partial^{2}_{t}\phi(t)\big\rangle &=\big\langle \partial_{x}H^{y_{2}}_{0,1},\,\partial^{2}_{x}\phi(t)-\dot U(\phi(t))\big\rangle
    =-\big\langle \partial^{2}_{x}H^{y_{2}}_{0,1},\,\partial_{x}\phi(t)\big\rangle-\big\langle \partial_{x}H^{y_{2}(t)}_{0,1}, \, \dot U(\phi(t))\big\rangle, \\
    \big\langle \partial_{x}H^{y_{1}}_{-1,0},\,\partial^{2}_{t}\phi(t)\big\rangle&=\big\langle \partial_{x}H^{y_{1}}_{-1,0},\,\partial^{2}_{x}\phi(t)-\dot U(\phi(t))\big\rangle=-\big\langle \partial^{2}_{x}H^{y_{1}}_{-1,0},\,\partial_{x}\phi(t)\big\rangle-\big\langle \partial_{x}H^{y_{1}}_{-1,0}, \, \dot U(\phi(t))\big\rangle.
\end{align*}
Since \eqref{nlww} is locally well-posed in $S\times L^{2}(\mathbb{R}),$ we obtain from the identities above that the following functions $h(t,y)\coloneqq\big\langle \partial_{x}H^{y}_{0,1},\,\partial^{2}_{t}\phi(t)\big\rangle$ and $l(t,y)\coloneqq\big\langle \partial_{x}H^{y}_{-1,0},\,\partial^{2}_{t}\phi(t)\big\rangle$ are continuous in the domain $\mathbb{R}\times\mathbb{R}.$
\par So, from the continuity of the functions $h(t,y),\,l(t,y)$ and from the fact that $x_{1},\,x_{2}\in C^{1}(\mathbb{R}),$ we obtain that the functions
\begin{equation*}
    h_{1}(t)\coloneqq -\langle \partial_{t}\phi(t),\,\partial_{x}H^{x_{1}(t)}_{-1,0}(x)\rangle,\,
    h_{2}(t)\coloneqq -\langle \partial_{t}\phi(t),\,\partial_{x}H^{x_{2}(t)}_{0,1}(x)\rangle
\end{equation*}
are of class $C^{1}.$ In conclusion, from the equation \eqref{c2eq}, by chain rule and product rule, 
we verify that $x_{1},\,x_{2}$ are in $C^{2}(\mathbb{R}).$
\par Now, since $x_{1},\,x_{2} \in C^{2}(\mathbb{R})$ and $\dot x_{1},\,\dot x_{2}$ satisfy \eqref{c2eq}, we deduce after derive at time the function
\begin{equation*}
    M(t)
    \begin{bmatrix}
    \dot x_{1}(t)\\
    \dot x_{2}(t)
    \end{bmatrix},
\end{equation*}
the following equations
\begin{multline}\label{II1}
    \ddot x_{1}(t)\Big(\norm{\partial_{x}H_{0,1}}_{L^{2}}^{2}+\left\langle \partial_{x}g(t),\, \partial_{x}H^{x_{1}(t)}_{-1,0}\right\rangle\Big)+\ddot x_{2}(t)\Big(\left\langle \partial_{x}H^{x_{1}(t)}_{-1,0},\, \partial_{x}H^{x_{2}(t)}_{0,1}\right\rangle\Big)=
    \dot x_{1}(t)^{2}\Big(\left\langle \partial^{2}_{x}H^{x_{1}(t)}_{-1,0},\,\partial_{x}g(t)\right\rangle\Big)\\+\dot x_{1}(t)\left\langle \partial^{2}_{x}H^{x_{1}(t)}_{-1,0},\,\partial_{t}g(t)\right\rangle
    +\dot x_{1}(t)\dot x_{2}(t)\left\langle \partial^{2}_{x}H^{x_{1}(t)}_{-1,0},\,\partial_{x}H^{x_{2}(t)}_{0,1}\right\rangle+\dot x_{2}(t)^{2}\left\langle \partial_{x}H^{x_{1}(t)}_{-1,0},\,\partial^{2}_{x}H^{x_{2}(t)}_{0,1}\right\rangle 
    +\dot x_{1}(t)\left\langle \partial^{2}_{x}H^{x_{1}(t)}_{-1,0},\,\partial_{t}\phi(t)\right\rangle\\-\left\langle \partial_{x}H^{x_{1}(t)}_{-1,0},\,\partial^{2}_{t}\phi(t)\right\rangle,
\end{multline}
\begin{multline}\label{II2}
    \ddot x_{2}(t)\Big(\norm{\partial_{x}H_{0,1}}_{L^{2}}^{2}+\left\langle \partial_{x}g(t),\, \partial_{x}H^{x_{2}(t)}_{0,1}\right\rangle\Big)+\ddot x_{1}(t)\Big(\big\langle \partial_{x}H^{x_{1}(t)}_{-1,0},\, \partial_{x}H^{x_{2}(t)}_{0,1}\big\rangle\Big)=
    \dot x_{2}(t)^{2}\Big(\left\langle \partial^{2}_{x}H^{x_{2}(t)}_{0,1},\,\partial_{x}g(t)\right\rangle\Big)\\+\dot x_{2}(t)\Big(\left\langle \partial^{2}_{x}H^{x_{2}(t)}_{0,1},\,\partial_{t}g(t)\right\rangle\Big)
    +\dot x_{1}(t)\dot x_{2}(t)\left\langle \partial_{x}H^{x_{1}(t)}_{-1,0},\,\partial^{2}_{x}H^{x_{2}(t)}_{0,1}\right\rangle+(\dot x_{1}(t))^{2}\left\langle \partial_{x}H^{x_{2}(t)}_{0,1},\,\partial^{2}_{x}H^{x_{1}(t)}_{-1,0}\right\rangle
    +\dot x_{2}(t)\left\langle \partial^{2}_{x}H^{x_{2}(t)}_{0,1},\,\partial_{t}\phi(t)\right\rangle\\ -\left\langle \partial_{x}H^{x_{2}(t)}_{0,1},\,\partial^{2}_{t}\phi(t)\right\rangle
.\end{multline}
 
\par Also, from the identity $g(t)=\phi(t)-H^{x_{1}(t)}_{-1,0}-H^{x_{2}(t)}_{0,1},$ we obtain that
$\partial_{t}g(t)=\partial_{t}\phi(t,x)+\dot x_{1}(t)\partial_{x}H^{x_{1}(t)}_{-1,0}+\dot x_{2}(t)\partial_{x}H^{x_{2}(t)}_{0,1},$
so, from the estimates \eqref{edd1} and \eqref{derivada1}, we obtain that
\begin{equation}\label{gt}
    \norm{\partial_{t}g(t)}_{L^{2}}=O(\epsilon^{\frac{1}{2}}).
\end{equation}
\par Now, since $\phi(t)$ is a distributional solution of \eqref{nlww}, we also have, from the global equality $\phi(t)=H^{x_{1}(t)}_{-1,0}+H^{x_{2}(t)}_{0,1}+g(t),$ the following identity
\begin{multline}
    \left\langle \partial_{x}H^{x_{1}(t)}_{-1,0},\,\partial^{2}_{t}\phi(t)\right\rangle=
    \left\langle \partial_{x}H^{x_{1}(t)}_{-1,0},\,\partial^{2}_{x}g(t)-\ddot U\left(H^{x_{1}(t)}_{-1,0}\right)g(t)\right\rangle  -\left\langle \partial_{x}H^{x_{1}(t)}_{-1,0},\left[\ddot U\left(H^{x_{1}(t)}_{-1,0}+H^{x_{2}(t)}_{0,1}\right)-\ddot U\left(H^{x_{1}(t)}_{-1,0}\right)\right]g(t)\right\rangle\\
    +\left\langle \partial_{x}H^{x_{1}(t)}_{-1,0},\,\dot U\left(H^{x_{1}(t)}_{-1,0}\right)+\dot U\left(H^{x_{2}(t)}_{0,1}\right)-\dot U\left(H^{x_{1}(t)}_{-1,0}+H^{x_{2}(t)}_{0,1}\right)\right\rangle\\
    -\left\langle \partial_{x}H^{x_{1}(t)}_{-1,0}, \dot U\left(H^{x_{1}(t)}_{-1,0}+H^{x_{2}(t)}_{0,1}+g(t)\right)-\dot U\left(H^{x_{1}(t)}_{-1,0}+H^{x_{2}(t)}_{0,1}\right)-\ddot U\left(H^{x_{1}(t)}_{-1,0}+H^{x_{2}(t)}_{0,1}\right)g(t)\right\rangle
\end{multline}
Since $\partial_{x}H^{x_{1}(t)}_{-1,0} \in ker D^{2}E_{pot}\left(H^{x_{1}(t)}_{-1,0}\right)$, we have by integration by parts that
$\left \langle \partial_{x}H^{x_{1}(t)}_{-1,0},\,\partial^{2}_{x}g(t)-\ddot U\left(H^{x_{1}(t)}_{-1,0}\right)g(t)\right \rangle=0.$
Since, we have
\begin{equation}\label{intterm}
    \dot U\left(H^{x_{1}(t)}_{-1,0}\right)+\dot U\left(H^{x_{2}(t)}_{0,1}\right)-\dot U\left(H^{x_{1}(t)}_{-1,0}+H^{x_{2}(t)}_{0,1}\right)=24H^{x_{1}(t)}_{-1,0}H^{x_{2}(t)}_{0,1}(H^{x_{1}(t)}_{-1,0}+H^{x_{2}(t)}_{0,1})-6\sum_{j=1}^{4}\begin{pmatrix}
        5\\
        j
    \end{pmatrix}\left(H^{x_{1}(t)}_{-1,0}\right)^{j}\left(H^{x_{2}(t)}_{0,1}\right)^{5-j},
\end{equation}
Lemma \ref{interact} implies that
$\left\langle \partial_{x}H^{x_{1}(t)}_{-1,0},\,\dot U\left(H^{x_{1}(t)}_{-1,0}\right)+\dot U\left(H^{x_{2}(t)}_{0,1}\right)-\dot U\left(H^{x_{1}(t)}_{-1,0}+H^{x_{2}(t)}_{0,1}\right)\right\rangle=O\Big(e^{-\sqrt{2}(z(t))}\Big).$
Also, we have from Taylor's Expansion Theorem
the estimate
\begin{equation*}
    \left\langle \partial_{x}H^{x_{1}(t)}_{-1,0}, \dot U\left(H^{x_{1}(t)}_{-1,0}+H^{x_{2}(t)}_{0,1}+g(t)\right)-\dot U\left(H^{x_{1}(t)}_{-1,0}+H^{x_{2}(t)}_{0,1}\right)-\ddot U\left(H^{x_{1}(t)}_{-1,0}+H^{x_{2}(t)}_{0,1}\right)g(t)\right\rangle=O(\norm{g(t)}_{H^{1}}^{2}).
\end{equation*}
From Lemma \ref{interact}, the fact that $U$ is a smooth function and $H_{0,1}\in L^{\infty}(\mathbb{R})$, we can obtain
\begin{equation*}
    \left\langle \partial_{x}H^{x_{1}(t)}_{-1,0},\left[\ddot U\left(H^{x_{1}(t)}_{-1,0}+H^{x_{2}(t)}_{0,1}\right)-\ddot U\left(H^{x_{1}(t)}_{-1,0}\right)\right]g(t)\right\rangle =O\Big(\int_{\mathbb{R}}\partial_{x}H^{x_{1}(t)}_{-1,0}H^{x_{2}(t)}_{0,1}\md{g(t)}\,dx\Big)
    =O\Big(e^{-\sqrt{2}z(t)}\norm{g(t)}_{H^{1}}z(t)^{\frac{1}{2}}\Big).
\end{equation*}    
In conclusion, we have
\begin{equation}\label{EEE1}
    \left\langle \partial_{x}H^{x_{1}(t)}_{-1,0},\,\partial^{2}_{t}\phi(t)\right\rangle=O\Big(\norm{g(t)}_{H^{1}}^{2}+e^{-\sqrt{2}z(t)}\Big),
\end{equation}
and by similar arguments, we have
\begin{equation}\label{EEE}
    \left\langle \partial_{x}H^{x_{2}(t)}_{0,1},\,\partial^{2}_{t}\phi(t)\right\rangle=O\Big(\norm{g(t)}_{H^{1}}^{2}+e^{-\sqrt{2}z(t)}\Big).
\end{equation}
\par Also, the equations \eqref{II1} and \eqref{II2} form a linear system with $\ddot x_{1}(t),\,\ddot x_{2}(t).$ Recalling that the Matrix $M(t)$ is uniformly positive, we obtain from the estimates \eqref{edd3}, \eqref{derivada1}, \eqref{gt}, \eqref{EEE1} and \eqref{EEE} that
\begin{equation}
    \max_{j \in \{1,2\}}\md{\ddot x_{j}(t)}=O(\epsilon).
\end{equation}
\end{proof}
The Theorem \ref{Stab} can also be improved when the kinetic energy of the solution is included in the computation and additional conditions are added, more precisely:
\begin{theorem}\label{T2}
$\exists \delta_{0}>0,$ such that if $0<\epsilon\leq \delta_{0},\,(\phi(0,x),\partial_{t}\phi(0,x))\in S\times L^{2}(\mathbb{R})$ and $E_{total}((\phi(0,x),\partial_{t}\phi(0,x)))=2E_{pot}(H_{0,1})+\epsilon,$ 
then there are $x_{2},\,x_{1}\in C^{2}(\mathbb{R})$ such that $g(t,x)=\phi(t,x)-H^{x_{2}(t)}_{0,1}(x)-H^{x_{1}(t)}_{-1,0}$ satisfies
\begin{equation*}
    \left\langle g(t,x),\partial_{x}H^{x_{2}(t)}_{0,1}(x)\right\rangle=0,\,\left\langle g(t,x),\partial_{x}H^{x_{1}(t)}_{-1,0}(x)\right\rangle=0,
\end{equation*}
and 
\begin{equation}
    \epsilon \cong e^{-\sqrt{2}(x_{2}(t)-x_{1}(t))}+\norm{(g(t),\partial_{t}g(t))}_{H^{1},L^{2}}^{2}+\md{\dot x_{1}(t)}^{2}+\md{\dot x_{2}(t)}^{2},
\end{equation}
for all $t\in \mathbb{R},$ which means the existence of positive constants $C,c$ independent on $\epsilon$, such that for all $t\in \mathbb{R}$
\begin{equation}
    c\epsilon\leq e^{-\sqrt{2}(x_{2}(t)-x_{1}(t))}+\norm{(g(t),\partial_{t}g(t))}_{H^{1},L^{2}}^{2}+|\dot x_{1}(t)|^{2}+|\dot x_{2}(t)|^{2}\leq C\epsilon.
\end{equation}
\end{theorem}
\begin{proof}
In this proof, $L^{2},\, H^{1}$ mean, respectively, $L^{2}(\mathbb{R}),\, H^{1}(\mathbb{R})$.
From Modulation Lemma and Theorem \ref{Stab}, we can rewrite the solution $\phi(t)$ in the form
\begin{equation*}
    \phi(t,x)=H^{x_{1}(t)}_{-1,0}(x)+H^{x_{2}(t)}_{0,1}(x)+g(t,x)
\end{equation*}
with $x_{1}(t),\,x_{2}(t),\,g(t)$ satisfying the conclusion of Theorem \ref{Stab}. First we denote 
\begin{equation}
    \phi_{\sigma}(t)=\left(H^{x_{1}(t)}_{-1,0}(x)+H^{x_{2}(t)}_{0,1}(x), -\dot x_{1}(t)\partial_{x}H^{x_{1}(t)}_{-1,0}-\dot x_{2}(t)\partial_{x}H^{x_{2}(t)}_{0,1}\right)\in S\times L^{2}(\mathbb{R}),
\end{equation}
then we apply Taylor's Expansion Theorem in $E(\phi(t))$ around $\phi_{\sigma}(t)$, more precisely for $R_{\sigma}(t)$ the residue of second order of Energy's Taylor Expansion of $E(\phi(t),\partial_{t}\phi(t))$ around $\phi_{\sigma}(t)$, we have:
\begin{multline}
    2E_{pot}(H_{0,1})+\epsilon=E(\phi_{\sigma}(t))+\left\langle DE(\phi_{\sigma}(t)),\,\left (g(t),\partial_{t}g(t)\right)\right\rangle_{L^{2}\times L^{2}}
    +\frac{\left\langle D^{2}E(\phi_{\sigma}(t))\big (g(t),\partial_{t}g(t)\big ),\,\big (g(t),\partial_{t}g(t)\big )\right\rangle_{L^{2}\times L^{2}}}{2}\\+R_{\sigma}(t),
\end{multline}
such that for $(w_{1},w_{2}) \in S\times L^{2}(\mathbb{R})$ and $(v_{1},v_{2})\in H^{1}(\mathbb{R})\times L^{2}(\mathbb{R})$, we have the identities
\begin{equation*}
    E(w_{1},w_{2})=\frac{\norm{\partial_{x}w_{1}}_{L^{2}}^{2}+\norm{w_{2}}_{L^{2}}^{2}}{2}+\int_{\mathbb{R}}U(w_{1}(x))\,dx,
\end{equation*}
\begin{align}\label{DEE}
    \left\langle DE(w_{1},w_{2}),(v_{1},v_{2})\right\rangle_{L^{2}\times L^{2}}=\int_{\mathbb{R}}\partial_{x}w_{1}(x)\partial_{x}v_{1}(x)+\dot U(w_{1})v_{1}+ w_{2}(x)v_{2}(x)\,dx,
    \\\label{DE2}
    D^{2}E(w_{1},w_{2})=
    \begin{bmatrix}
    -\partial^{2}_{x}+ \ddot U(w_{1}) & 0\\
    0 & \mathbb{I}
    \end{bmatrix}
\end{align}
with $D^{2}E(w_{1},w_{2})$ defined as a bilinear operator from $H^{1}\times L^{2}$ to $\mathbb{C}$. So, from identities \eqref{DEE} and \eqref{DE2}, it is not difficult to verify that 
\begin{multline*}
    R_{\sigma}(t)=\int_{\mathbb{R}}U\left(H^{x_{1}(t)}_{-1,0}(x)+H^{x_{2}(t)}_{0,1}(x)+g(t,x)\right)-U\left(H^{x_{1}(t)}_{-1,0}(x)+H^{x_{2}(t)}_{0,1}(x)\right)-\dot U\left(H^{x_{1}(t)}_{-1,0}(x)+H^{x_{2}(t)}_{0,1}(x)\right)g(t,x)\\-\frac{\ddot U\left(H^{x_{1}(t)}_{-1,0}(x)+H^{x_{2}(t)}_{0,1}(x)\right)g(t,x)^{2}}{2}\,dx,
\end{multline*}
and, so,
\begin{equation}\label{Id0}
    \md{R_{\sigma}(t)}=O(\norm{g(t)}_{H^{1}}^{3}).
\end{equation}
Also, we have
\begin{equation}\label{eident}
    \left\langle DE(\phi_{\sigma}(t)),\,(g(t),\partial_{t}g(t))\right\rangle_{L^{2}\times L^{2}}=\left\langle DE_{pot}\left(H^{x_{1}(t)}_{-1,0}+H^{x_{2}(t)}_{0,1}\right),\,g(t)\right\rangle+\left\langle -\dot x_{1}(t)\partial_{x}H^{x_{1}(t)}_{-1,0}-\dot x_{2}(t)\partial_{x}H^{x_{2}(t)}_{0,1}, \, \partial_{t}g(t)\right\rangle.
\end{equation}
The orthogonality conditions satisfied by $g(t)$ also imply for all $t\in\mathbb{R}$ that
\begin{align}\label{ot}
    \left\langle \partial_{t}g(t),\,\partial_{x}H^{x_{1}(t)}_{-1,0}\right\rangle_{L^{2}}=\dot x_{1}(t)\langle g(t),\,\partial^{2}_{x}H^{x_{1}(t)}_{-1,0}\rangle_{L^{2}},\\
    \label{ot2}
    \left\langle \partial_{t}g(t),\,\partial_{x}H^{x_{2}(t)}_{0,1}\right\rangle_{L^{2}}=\dot x_{2}(t)\langle g(t),\,\partial^{2}_{x}H^{x_{2}(t)}_{0,1}\rangle_{L^{2}}.
\end{align}
So, the inequality \eqref{DE} and the identities \eqref{eident}, \eqref{ot}, \eqref{ot2} imply that
\begin{equation}\label{Id1}
    \md{\langle DE(\phi_{\sigma}(t)),\,(g(t),\partial_{t}g(t))\rangle_{L^{2}\times L^{2}}}=O\Big(\sup_{j\in\{1,2\}}\md{\dot x_{j}(t)}^{2}\norm{g(t)}_{H^{1}}+e^{-\sqrt{2}z(t)}\norm{g(t)}_{H^{1}}\Big).
\end{equation}
From the Coercitivity Lemma and the definition of $D^{2}E(\phi_{\sigma}(t))$, we have that
\begin{equation}\label{Id2}
    \left\langle D^{2}E(\phi_{\sigma}(t))(g(t),\partial_{t}g(t)),\,(g(t),\partial_{t}g(t))\right\rangle_{L^{2}\times L^{2}} \cong \norm{(g(t),\partial_{t}g(t))}_{H^{1}\times L^{2}}^{2}.
\end{equation}
Finally, there is the identity
\begin{multline}\label{Id3}
    \norm{\dot x_{1}(t)\partial_{x}H_{-1,0}^{x_{1}(t)}(x)+ \dot x_{2}(t)\partial_{x}H_{0,1}^{x_{2}(t)}(x)}_{L^{2}}^{2}=2\dot x_{1}(t)\dot x_{2}(t)\left\langle \partial_{x}H^{z(t)}_{0,1},\,\partial_{x}H_{-1,0}\right\rangle_{L^{2}}\\+\dot x_{1}(t)^{2}\norm{\partial_{x}H_{0,1}}_{L^{2}}^{2}+\dot x_{2}(t)^{2}\norm{\partial_{x}H_{0,1}}_{L^{2}}^{2}.
\end{multline}
From Lemma \ref{interact}, we have that $\md{\langle \partial_{x}H^{z}_{0,1},\,\partial_{x}H_{-1,0}\rangle_{L^{2}}}=O\big(ze^{-\sqrt{2}z}\big)$ for $z$ big enough.
Then, it is not difficult to verify that Lemma \ref{LL.1}, \eqref{Id0}, \eqref{Id1}, \eqref{Id2} and \eqref{Id3} imply directly the statement of the Theorem \ref{T2} which finishes the proof.
\end{proof}
\begin{remark}
Theorem \ref{T2} implies that it is possible to have a solution $\phi$ of the equation \eqref{nlww} with energy excess $\epsilon>0$ small enough satisfying all the hypotheses of Theorem \ref{T1}. More precisely, in notation of Theorem \ref{T1}, if $\norm{(g(0,x),\partial_{t}g(0,x))}_{H^{1}\times L^{2}}\ll \epsilon^{\frac{1}{2}},$ then, from Theorem \ref{T2}, we have that
\begin{equation*}
    e^{-\sqrt{2}z(0)}+v_{1}^{2}+v_{2}^{2}\cong \epsilon.
\end{equation*}
In conclusion, we obtain that $E(\phi(0),\partial_{t}\phi(0))-2E_{pot}(H_{0,1})\cong \epsilon.$
\end{remark}
\section{Long Time Behavior of Modulation Parameters}
Even though Theorem \ref{Stab} proves the orbital stability of a sum of two kinks with low energy excess, this theorem doesn't explain the movement of the kinks' centers $x_{2}(t),\,x_{1}(t)$ and their speed for long time. More precisely, we still don't know if there is a explicit smooth real function $d(t),$ such that $(z(t),\dot z(t))$ is close to $(d(t),\dot d(t))$ in a large time interval. 
\par But, the global estimates on the modulus of the first and second derivatives of $x_{1}(t),\,x_{2}(t)$ obtained in Theorem \ref{Stab} will be very useful to estimate with high precision the functions $x_{1}(t),\,x_{2}(t)$ during a very large time interval. Moreover, we first have the following auxiliary lemma. 
\begin{lemma} \label{modueq}
Let $0<\theta,\,\gamma<1.$ We recall the function \begin{center}$A(z)=E_{pot}(H^{z}_{0,1}+H_{-1,0})$\end{center} for any $z >0.$ If the same hypothesis of Theorem \ref{Stab} are true and
let $\chi(x)$ be a smooth function such that
\begin{equation}\label{generall}
    \chi(x)=\begin{cases}
        1, \text{ if $x\leq \theta(1-\gamma)$,}\\
        0, \text{ if $x\geq \theta$.}
\end{cases}
\end{equation}
and $0\leq\chi(x)\leq 1$ for all $x\in\mathbb{R}.$ In notation of Theorem \ref{Stab}, we denote
\begin{equation*}
    \chi_{0}(t,x)=\chi\Big(\frac{x-x_{1}(t)}{z(t)}\Big),\,
    \overrightarrow{g(t)}=(g(t),\partial_{t}g(t)) \in H^{1}(\mathbb{R})\times L^{2}(\mathbb{R})
\end{equation*}
 and $\norm{\overrightarrow{g(t)}}=\norm{(g(t),\partial_{t}g(t))}_{H^{1}(\mathbb{R})\times L^{2}(\mathbb{R})},$ 
\begin{multline}\label{alpha}
    \alpha(t)=\max_{j \in \{1,2\}}\dot x_{j}(t)^{2}z(t)e^{-\sqrt{2}z(t)}+\frac{\max_{j\in\{1,\,2\}}\dot x_{j}(t)^{2}}{z(t)\gamma}\Big(e^{-2\sqrt{2}z(t)(\frac{1-\gamma}{2-\gamma})}\Big)
    \\+\norm{\overrightarrow{g(t)}}\max_{j\in\{1,\,2\}}\md{\dot x_{j}(t)}\Big[1+\frac{1}{z(t)\gamma}+\frac{1}{z(t)^{2}\gamma^{2}}\max_{j\in\{1,\,2\}}\md{\dot x_{j}(t)}\Big]\Big(e^{-\sqrt{2}z(t)(\frac{1-\gamma}{2-\gamma})}\Big)
    +\norm{\overrightarrow{g(t)}}^{2}\Big[\frac{1}{\gamma^{2}z(t)^{2}}+\frac{1}{\gamma z(t)}+\Big(e^{-\sqrt{2}z(t)(\frac{1-\gamma}{2-\gamma})}\Big)\Big].
\end{multline}
Then, for $\theta=\frac{1-\gamma}{2-\gamma}$ and the correction terms
\begin{align*}
    p_{1}(t)=-\frac{\big\langle\partial_{t}\phi(t),\,\partial_{x}H^{x_{1}(t)}_{-1,0}(x)+\partial_{x}(\chi_{0}(t,x)g(t))\big\rangle}{\norm{\partial_{x}H_{0,1}}_{L^{2}}^{2}},\\
    p_{2}(t)=-\frac{\big\langle\partial_{t}\phi(t),\,\partial_{x}H^{x_{2}(t)}_{0,1}(x)+\partial_{x}([1-\chi_{0}(t,x)]g(t))\big\rangle}{\norm{\partial_{x}H_{0,1}}_{L^{2}}^{2}},
\end{align*}
we have the estimates, for $j \in \{1,2\},$
\begin{equation}\label{Modul1}
    \md{\dot x_{j}(t)-p_{j}(t)}\lesssim \Big[1+\frac{\norm{\dot \chi}_{L^{\infty}}}{z(t)}\Big]\Big(\max_{j \in \{1,2\}}\md{\dot x_{j}(t)}\norm{\overrightarrow{g(t)}}+\norm{\overrightarrow{g(t)}}^{2}\Big)+\max_{j \in \{1,2\}}\md{\dot x_{j}(t)}z(t)e^{-\sqrt{2}z(t)},
    \end{equation}
\begin{equation}\label{Modul2}
    \md{\dot p_{j}(t)+(-1)^{j}\frac{\dot A(z(t))}{\norm{\partial_{x}H_{0,1}}_{L^{2}}^{2}}}\lesssim\alpha(t).
\end{equation}
\end{lemma}
\begin{remark}
We will take $\gamma=\frac{\ln{\ln{(\frac{1}{\epsilon})}}}{\ln{(\frac{1}{\epsilon})}}.$ With this value of $\gamma$ and the estimates of Theorem \ref{Stab}, we will see in Lemma \ref{aux22}  that $\exists C>0$ such that \begin{equation*}\alpha(t)\lesssim \frac{\Big(\norm{(g_{0},g_{1})}_{H^{1}\times L^{2}}+\epsilon\ln{(\frac{1}{\epsilon})}\Big)^{2}}{\ln{\ln{(\frac{1}{\epsilon})}}}\exp\Big(\frac{2C\md{t}\epsilon^{\frac{1}{2}}}{\ln{(\frac{1}{\epsilon})}}\Big).\end{equation*}
 \end{remark}
\begin{proof}
For $\gamma\ll 1$ enough and from the definition of $\chi(x),$ it is not difficult to verify that
\begin{equation}\label{linfty}
    \norm{\dot \chi}_{L^{\infty}(\mathbb{R})}\lesssim \frac{1}{\gamma},\, \norm{\ddot \chi}_{L^{\infty}(\mathbb{R})}\lesssim \frac{1}{\gamma^{2}}.
\end{equation}
We will only do the proof of the estimates \eqref{Modul1} and \eqref{Modul2} for $j=1,$ the proof for the case $j=2$ is completely analogous. From the proof of Theorem \ref{Stab}, we know that $\dot x_{1}(t),\,\dot x_{2}(t)$ solve the linear system
\begin{equation*}
    M(t)
    \begin{bmatrix}
    \dot x_{1}(t)\\
    \dot x_{2}(t)
    \end{bmatrix}=
    \begin{bmatrix}
    -\langle \partial_{t}\phi(t),\,\partial_{x}H^{x_{1}(t)}_{-1,0}\rangle\\
    -\langle \partial_{t}\phi(t),\,\partial_{x}H^{x_{2}(t)}_{0,1}\rangle
    \end{bmatrix},
\end{equation*}
where $M(t)$ is the matrix defined by\eqref{Matrix1}. Then, from Cramer's rule, we obtain that
\begin{equation}\label{identx1}
    \dot x_{1}(t)=\frac{-\left\langle\partial_{t}\phi(t),\,\partial_{x}H^{x_{1}(t)}_{-1,0}\right\rangle\Big(\left\langle \partial_{x}H^{x_{2}(t)}_{0,1},\,\partial_{x}g(t)\right\rangle +\norm{\partial_{x}H_{0,1}}_{L^{2}}^{2}
\Big)}{\det(M(t))}
    +\frac{\left\langle\partial_{t}\phi(t),\,\partial_{x}H^{x_{2}(t)}_{0,1}\right\rangle\left\langle\partial_{x}H^{x_{2}(t)}_{0,1},\,\partial_{x}H^{x_{1}(t)}_{-1,0}\right\rangle}{\det(M(t))}.
\end{equation}
Using the definition \eqref{Matrix1} of the matrix $M(t)$, $\norm{\overrightarrow{g(t)}}=O(\epsilon^{\frac{1}{2}})$ and Lemma \ref{interact} which implies the following estimate
\begin{equation}\label{int2}
    \left\langle\partial_{x}H^{x_{2}(t)}_{0,1},\,\partial_{x}H^{x_{1}(t)}_{-1,0}\right\rangle=O(z(t)e^{-\sqrt{2}z(t)}), 
\end{equation}
we obtain that
\begin{equation}\label{FirstEst}
    \md{\det(M(t))-\norm{\partial_{x}H_{0,1}}_{L^{2}}^{4}}=O\Big(\norm{\overrightarrow{g(t)}}+z(t)^{2}e^{-2\sqrt{2}z(t)}\Big)=O(\epsilon^{\frac{1}{2}}).
\end{equation}
So, from the estimate \eqref{FirstEst} and the identity \eqref{identx1}, we obtain that
\begin{multline}\label{FirstEst1}
\md{\dot x_{1}(t)+\frac{1}{\norm{\partial_{x}H_{0,1}}_{L^{2}(\mathbb{R})}^{2}}\left\langle \partial_{t}\phi(t),\,\partial_{x}H^{x_{1}(t)}_{-1,0}\right\rangle}= O\Big(\md{\left\langle \partial_{x}H^{x_{1}(t)}_{-1,0},\,\partial_{x}g(t)\right\rangle \left\langle\partial_{t}\phi(t),\,\partial_{x}H^{x_{1}(t)}_{-1,0}\right\rangle}\Big)\\    
+O\Big(\md{\left\langle \partial_{x}H^{x_{1}(t)}_{-1,0},\,\partial_{x}H^{x_{2}(t)}_{0,1}\right\rangle \left\langle\partial_{t}\phi(t),\,\partial_{x}H^{x_{2}(t)}_{0,1}\right\rangle}\Big)
+O\Big(\md{\left\langle\partial_{t}\phi(t),\,\partial_{x}H^{x_{1}(t)}_{-1,0}(x)\right\rangle}\Big[\norm{\overrightarrow{g(t)}}+z(t)^{2}e^{-2\sqrt{2}z(t)}\Big]).
\end{multline}
Finally, from the definition of $g(t,x)$ in Theorem \ref{Stab} we know that
\begin{equation*}
    \partial_{t}\phi(t,x)=-\dot x_{1}(t)\partial_{x}H^{x_{1}(t)}_{-1,0}(x)-\dot x_{2}(t)\partial_{x}H^{x_{2}(t)}_{0,1}(x)+\partial_{t}g(t,x),
\end{equation*}
from the Modulation Lemma we also have verified that
\begin{equation*}
    \left\langle \partial_{t}g(t),\,\partial_{x}H^{x_{1}(t)}_{-1,0}\right\rangle=O\Big(\norm{\overrightarrow{g(t)}}\md{\dot x_{1}(t)}\Big),\,
    \left\langle \partial_{t}g(t),\,\partial_{x}H^{x_{2}(t)}_{0,1}\right\rangle=O\Big(\norm{\overrightarrow{g(t)}}\md{\dot x_{2}(t)}\Big)
\end{equation*}
 and from Theorem \ref{Stab} we have that $\norm{\overrightarrow{g(t)}}+\max_{j \in \{1,2\}}|\dot x_{j}(t)|\ll 1.$ In conclusion, we can rewrite the estimate \eqref{FirstEst1} as
\begin{equation}\label{parameter1}
    \md{\dot x_{1}(t)+\frac{1}{\norm{\partial_{x}H_{0,1}}_{L^{2}(\mathbb{R})}^{2}}\left\langle \partial_{t}\phi(t),\,\partial_{x}H^{x_{1}(t)}_{-1,0}\right\rangle}=O\Big(\max_{j \in \{1,2\}}\md{\dot x_{j}(t)}\norm{\overrightarrow{g(t)}}+\norm{\overrightarrow{g(t)}}^{2}+z(t)e^{-\sqrt{2}z(t)}\max_{j \in \{1,2\}}\md{\dot x_{j}(t)}\Big).
\end{equation}
By a similar reasoning, we can also deduce that
\begin{equation}\label{parameter2}
    \md{\dot x_{2}(t)+\frac{1}{\norm{\partial_{x}H_{0,1}}_{L^{2}(\mathbb{R})}^{2}}\left\langle \partial_{t}\phi(t),\,\partial_{x}H^{x_{2}(t)}_{0,1}\right\rangle}=O\Big(\max_{j \in \{1,2\}}\md{\dot x_{j}(t)}\norm{\overrightarrow{g(t)}}+\norm{\overrightarrow{g(t)}}^{2}+z(t)e^{-\sqrt{2}z(t)}\max_{j \in \{1,2\}}\md{\dot x_{j}(t)}\Big).
\end{equation}
\par Following the reasoning of Lemma 3.5 of \cite{jkl}, we will use the terms $p_{1}(t),\,p_{2}(t)$ with the objective of obtaining the estimates \eqref{Modul2}, which have high precision and will be useful later to approximate $x_{j}(t),\,\dot x_{j}(t)$ by explicit smooth functions during a long time interval. 
\par First, it is not difficult to verify that
\begin{equation*}
    \left\langle \partial_{t}\phi(t),\,\partial_{x}(\chi_{0}(t)g(t)) \right\rangle=O\Big(\Big[1+\frac{\norm{\dot \chi}_{L^{\infty}}}{z(t)}\Big]\norm{\overrightarrow{g(t)}}^{2}+\max_{j \in \{1,2\}} \md{\dot x_{j}(t)}\norm{\overrightarrow{g(t)}}\Big),
\end{equation*}
which clearly implies with estimate \eqref{parameter1} the inequality \eqref{Modul1} for $j=1.$ The proof of inequality \eqref{Modul1} for $j=2$ is completely analog.
\par Now, the demonstration of the inequality \eqref{Modul2} is similar to the proof of the second inequality of Lemma 3.5 of \cite{jkl}. First, we have
\begin{multline}\label{principalid}
    \dot p_{1}(t)=-\frac{\left\langle \partial_{t}\phi(t),\,\partial_{t}\big(\partial_{x}H^{x_{1}(t)}_{-1,0}(x)\big)\right\rangle}{\norm{\partial_{x}H_{0,1}}_{L^{2}}^{2}}-\frac{\left\langle \partial_{t}\phi(t),\,\partial_{x}\big(\partial_{t}\chi_{0}(t)g(t)\big)\right\rangle}{\norm{\partial_{x}H_{0,1}}_{L^{2}}^{2}}
    -\frac{\left\langle\partial_{x}\big(\chi_{0}(t)\partial_{t}g(t)\big),\,\partial_{t}\phi(t)\right\rangle}{\norm{\partial_{x}H_{0,1}}_{L^{2}}^{2}}-\frac{\left\langle \partial_{x}H^{x_{1}(t)}_{-1,0},\,\partial^{2}_{t}\phi(t)\right\rangle}{\norm{\partial_{x}H_{0,1}}_{L^{2}}^{2}}\\
    -\frac{\left\langle \partial_{x}\chi_{0}(t)g(t),\,\partial^{2}_{t}\phi(t)\right\rangle}{\norm{\partial_{x}H_{0,1}}_{L^{2}}^{2}}-\frac{\left\langle \chi_{0}(t)\partial_{x}g(t),\,\partial^{2}_{t}\phi(t)\right\rangle}{\norm{\partial_{x}H_{0,1}}_{L^{2}}^{2}}=I+II+III+IV+V+VI,
\end{multline}
and we will estimate each term one by one. More precisely, from now on, we will work with a general cut function $\chi(x),$ that is a smooth function $0\leq\chi\leq 1$ satisfying
\begin{equation}\label{S}
    \chi(x)=\begin{cases}
        1, \text{ if $x\leq \theta(1-\gamma)$,}\\
        0, \text{ if $x\geq \theta$.}
\end{cases}
\end{equation}
with $0<\theta,\,\gamma<1$ and 
\begin{equation}\label{S2}
    \chi_{0}(t,x)=\chi\Big(\frac{x-x_{1}(t)}{z(t)}\Big).
\end{equation}
The reason for this notation is to improve the precision of the estimate of $\dot p_{1}(t)$ by the searching of the $\gamma,\,\theta$ which minimize $\alpha(t)$. \\
\textbf{Step 1.}(Estimate of $I$)
We will only use the identity
$ I=\dot x_{1}(t)\frac{\big\langle\partial_{t}\phi(t),\,\partial^{2}_{x}H^{x_{1}(t)}_{-1,0}\big\rangle}{\norm{\partial_{x}H_{0,1}}_{L^{2}}^{2}}.$\\
\textbf{Step 2.}(Estimate of $II.$)
We have, by chain rule and definition of $\chi_{0},$ that
\begin{align*}
 II &=-\frac{\big\langle \partial_{t}\phi(t),\,\partial_{x}\big(\partial_{t}\chi_{0}(t)g(t)\big)\big\rangle}{\norm{\partial_{x}H_{0,1}}_{L^{2}}^{2}}
 =-\frac{\Big\langle\partial_{t}\phi(t),\,\partial_{x}\Big[\frac{d}{dt}\Big[\chi\Big(\frac{x-x_{1}(t)}{z(t)}\Big)\Big]g(t,x)\Big]\Big\rangle}{\norm{\partial_{x}H_{0,1}}_{L^{2}}^{2}}\\
 &=-\frac{\Big\langle\partial_{t}\phi(t),\,\partial_{x}\Big(\dot\chi\Big(\frac{x-x_{1}(t)}{z(t)}\Big)\frac{d}{dt}\Big[\frac{x-x_{1}(t)}{z(t)}\Big]g(t)\Big)\Big\rangle}{\norm{\partial_{x}H_{0,1}}_{L^{2}}^{2}}
 =\frac{\Big\langle\partial_{t}\phi(t),\,\partial_{x}\Big(\dot \chi\Big(\frac{x-x_{1}(t)}{z(t)}\Big)\Big[\frac{\dot x_{1}(t)z(t)+(x-x_{1}(t))\dot z(t)}{z(t)^{2}}\Big]g(t)\Big)\Big\rangle}{\norm{\partial_{x}H_{0,1}}_{L^{2}}^{2}}.
\end{align*}
So, we obtain that
\begin{multline}\label{II.1}
    II=\frac{\Big\langle \partial_{t}\phi(t),\,\ddot \chi\Big(\frac{x-x_{1}(t)}{z(t)}\Big)\Big[\frac{\dot x_{1}(t)}{z(t)}+\frac{(x-x_{1}(t))\dot z(t)}{z(t)^{2}}\Big]g(t)\Big\rangle}{z(t)\norm{\partial_{x}H_{0,1}}_{L^{2}}^{2}}
    +\frac{\Big\langle\partial_{t}\phi(t),\,\dot\chi\Big(\frac{x-x_{1}(t)}{z(t)}\Big)\frac{\dot z(t)}{z(t)^{2}}g(t)\Big\rangle}{\norm{\partial_{x}H_{0,1}}_{L^{2}}^{2}}\\
    +\frac{\Big\langle \partial_{t}\phi(t),\,\dot\chi\Big(\frac{x-x_{1}(t)}{z(t)}\Big)\Big[\frac{\dot x_{1}(t)}{z(t)}+\frac{(x-x_{1}(t))\dot z(t)}{z(t)^{2}}\Big]\partial_{x}g(t)\Big\rangle}{\norm{\partial_{x}H_{0,1}}_{L^{2}}^{2}}.
\end{multline}
First, note that since the support of $\dot \chi$ is contained in $[\theta(1-\gamma),\,\theta],$ from the estimates (D3) and (D4) we obtain that 
\begin{align}\label{cutest1}
    \norm{\partial_{x}H^{x_{1}(t)}_{-1,0}}_{L^{2}_{x}\Big(\supp \partial_{x}\chi_{0}(t,x)\Big)}^{2}=O\Big(e^{-2\sqrt{2}\theta(1-\gamma)z(t)}\Big),\\ \label{cutest2}
    \norm{\partial_{x}H^{x_{2}(t)}_{0,1}}_{L^{2}_{x}\Big(\supp \partial_{x}\chi_{0}(t,x)\Big)}^{2}=O\Big(e^{-2\sqrt{2}(1-\theta)z(t)}\Big),
\end{align}
Now, we recall the identity 
$ \partial_{t}\phi(t,x)=-\dot x_{1}(t)\partial_{x}H^{x_{1}(t)}_{-1,0}-\dot x_{2}(t)\partial_{x}H^{x_{2}(t)}_{0,1}+\partial_{t}g(t),$
by using the estimates \eqref{cutest1}, \eqref{cutest2} in the identity \eqref{II.1}, we deduce that 
\begin{multline}\label{m22}
    II=O\Bigg(\norm{\dot \chi}_{L^{\infty}(\mathbb{R})}\frac{ \max_{j \in \{1,\,2\}}\md{\dot x_{j}(t)}}{z(t)}\norm{\overrightarrow{g(t)}}^{2}+\norm{\ddot\chi}_{L^{\infty}(\mathbb{R})}\norm{\overrightarrow{g(t)}}^{2}\frac{\max_{j\in\{1,\,2\}}\md{\dot x_{j}(t)}}{z(t)^{2}}
    \\+(e^{-\sqrt{2}\theta(1-\gamma)z(t)}+e^{-\sqrt{z(t)}(1-\theta)})\norm{\ddot\chi}_{L^{\infty}(\mathbb{R})}\frac{\max_{j \in \{1,\,2\} }\dot x_{j}(t)^{2}}{z(t)^{2}}\norm{\overrightarrow{g(t)}}
    \\+\norm{\overrightarrow{g(t)}}(e^{-\sqrt{2}z(t)(1-\theta)}+e^{-\sqrt{2}\theta(1-\gamma)z(t)})\left[\frac{\norm{\ddot\chi}_{L^{\infty}(\mathbb{R})}}{z(t)^{2}}+\frac{\norm{\dot \chi}_{L^{\infty}(\mathbb{R})}}{z(t)}\right]\max_{j \in \{1,\,2\}}\dot x_{j}(t)^{2}\Bigg).
\end{multline}
Since $\frac{1-\gamma}{2-\gamma}\leq \max((1-\theta),\theta(1-\gamma))$ for $0<\gamma,\theta<1,$ we have that the estimate \eqref{m22} is minimal when $\theta=\frac{1-\gamma}{2-\gamma}.$ So, from now on, we consider
\begin{equation}\label{ttheta}
\theta=\frac{1-\gamma}{2-\gamma},
\end{equation}
which with \eqref{linfty} and \eqref{m22} imply that 
$II=O(\alpha(t)).$\\
\textbf{Step 3.}(Estimate of $III.$)
We deduce from the identity
\begin{equation*}
    III=-\frac{\langle\partial_{x}(\chi_{0}(t)\partial_{t}g(t)),\,\partial_{t}\phi(t)\rangle}{\norm{\partial_{x}H_{0,1}}_{L^{2}}^{2}}
\end{equation*}
that
\begin{multline}\label{IdIII}
III=-\frac{\Big\langle\dot\chi\Big(\frac{x-x_{1}(t)}{z(t)}\Big)\partial_{t}g(t),\,-\dot x_{1}(t)\partial_{x}H^{x_{1}(t)}_{-1,0}-\dot x_{2}(t)\partial_{x}H^{x_{2}(t)}_{0,1}+\partial_{t}g(t,x)\Big\rangle}{z(t)\norm{\partial_{x}H_{0,1}}_{L^{2}}^{2}}\\
-\frac{\Big\langle\chi_{0}(t,x)\partial^{2}_{t,x}g(t,x),\,-\dot x_{1}(t)\partial_{x}H^{x_{1}(t)}_{-1,0}-\dot x_{2}(t)\partial_{x}H^{x_{2}(t)}_{0,1}+\partial_{t}g(t,x)\Big\rangle}{\norm{\partial_{x}H_{0,1}}_{L^{2}}^{2}}
=III.1+III.2.
\end{multline}
The identity \eqref{ttheta} and the estimates \eqref{linfty}, \eqref{cutest1} and \eqref{cutest2} imply by Cauchy-Schwarz inequality that
\begin{equation}\label{III.1est}
    III.1=O\Bigg(\frac{\max_{j\in\{1,\,2\}}\md{\dot x_{j}(t)}e^{-\sqrt{2}z(t)(\frac{1-\gamma}{2-\gamma})}}{\gamma z(t)}\norm{\overrightarrow{g(t)}}
    +\frac{1}{z(t)\gamma}\norm{\overrightarrow{g(t)}}^{2}\Bigg).
\end{equation}
In conclusion, we have the estimate that
$III.1=O(\alpha(t)).$
Also, from condition \eqref{S} and the estimate \eqref{kinkestimate}, we can deduce that
\begin{equation}\label{supp1}
    \norm{(1-\chi_{0}(t))\partial^{2}_{x}H^{x_{1}(t)}_{-1,0}}_{L^{2}_{x}(\mathbb{R})}+
    \norm{\chi_{0}(t)\partial^{2}_{x}H^{x_{2}(t)}_{0,1}}_{L^{2}_{x}(\mathbb{R})}=O\Big(e^{-\sqrt{2}z(t)(\frac{1-\gamma}{2-\gamma})}\Big).
\end{equation}
Also, we have that 
\begin{equation}\label{estIII.2}
    III.2=-\frac{\Big\langle\chi_{0}(t,x)\Big[\partial^{2}_{t,x}\phi(t)+\dot x_{1}(t)\partial^{2}_{x}H^{x_{1}(t)}_{-1,0}+\dot x_{2}(t)\partial^{2}_{x}H^{x_{2}(t)}_{0,1}\Big],\,\partial_{t}\phi(t) \Big\rangle}{\norm{\partial_{x}H_{0,1}}_{L^{2}}^{2}}.
\end{equation}
By integration by parts, we have that
\begin{equation*}
   \md{\Big\langle\chi\Big(\frac{x-x_{1}(t)}{z(t)}\Big)\partial^{2}_{t,x}\phi(t,x),\,\partial_{t}\phi(t,x)\Big\rangle}=O\Big(\frac{1}{\gamma z(t)}\norm{\partial_{t}\phi(t)}_{L^{2}_{x}(\supp\partial_{x}\chi_{0}(t))}^{2}\Big).
\end{equation*}
In conclusion, from the estimates \eqref{linfty}, \eqref{cutest1}, \eqref{cutest2} and identity \eqref{ttheta}, we obtain that
\begin{equation}\label{quadratic}
    \md{\Big\langle\chi\Big(\frac{x-x_{1}(t)}{z(t)}\Big)\partial^{2}_{t,x}\phi(t,x),\,\partial_{t}\phi(t,x)\Big\rangle}=O\Big(\frac{1}{\gamma z(t)}\norm{\overrightarrow{g(t)}}^{2}+\max_{j \in \{1,\,2\}}\dot x_{j}(t)^{2}\frac{1}{\gamma z(t)}\Big[e^{-2\sqrt{2}z(t)(\frac{1-\gamma}{2-\gamma})}\Big]\Big).
\end{equation}
Also, from Lemma \eqref{interact}, the estimate \eqref{kinkestimate} and the fact of $0\leq\chi_{0}\leq 1,$ we deduce that
\begin{align}\label{innt1}
     \md{\Big\langle\chi_{0}(t,x)\partial^{2}_{x}H^{x_{2}(t)}_{0,1},\,\partial_{x}H^{x_{1}(t)}_{-1,0}\Big\rangle}=O\Big(z(t)e^{-\sqrt{2}z(t)}\Big),\\ \label{innt2}
     \md{\Big\langle(1-\chi_{0}(t,x))\partial^{2}_{x}H^{x_{1}(t)}_{-1,0},\,\partial_{x}H^{x_{2}(t)}_{0,1}\Big\rangle}=O\Big(z(t)e^{-\sqrt{2}z(t)}\Big).
\end{align}
From the estimates \eqref{cutest1}, \eqref{cutest2} and identity \eqref{ttheta}, we can verify by integration by parts the following estimates
\begin{equation}\label{nn1}
    \Big \langle (1-\chi_{0}(t))\dot x_{1}(t)\partial^{2}_{x}H^{x_{1}(t)}_{-1,0},\,\dot x_{1}(t)\partial_{x}H^{x_{1}(t)}_{-1,0}\Big\rangle=O\Big(\frac{\dot x_{1}(t)^{2}}{\gamma z(t)}e^{-2\sqrt{2}z(t)(\frac{1-\gamma}{2-\gamma})}\Big),
\end{equation}
\begin{equation}\label{nn2}
     \Big \langle \chi_{0}(t)\dot x_{2}(t)\partial^{2}_{x}H^{x_{2}(t)}_{0,1},\,\dot x_{2}(t)\partial_{x}H^{x_{2}(t)}_{0,1}\Big\rangle=
     O\Big(\frac{\dot x_{2}(t)^{2}}{\gamma z(t)}e^{-2\sqrt{2}z(t)(\frac{1-\gamma}{2-\gamma})}\Big).
\end{equation}
Finally, from Cauchy-Schwarz inequality and the estimate \eqref{supp1} we obtain that
\begin{align}\label{conc1}
    \Big\langle(1-\chi_{0}(t))\dot x_{1}(t)\partial^{2}_{x}H^{x_{1}(t)}_{-1,0},\,\partial_{t}g(t)\Big\rangle=O\Big(\md{\dot x_{1}(t)}\norm{\overrightarrow{g(t)}}e^{-\sqrt{2}z(t)(\frac{1-\gamma}{2-\gamma})}\Big),\\\label{conc2}
    \Big\langle \chi_{0}(t)\dot x_{1}(t)\partial^{2}_{x}H^{x_{2}(t)}_{0,1},\,\partial_{t}g(t)\Big\rangle=O\Big(\md{\dot x_{2}(t)}\norm{\overrightarrow{g(t)}}e^{-\sqrt{2}z(t)(\frac{1-\gamma}{2-\gamma})}\Big).
\end{align}
In conclusion, we obtain from the estimates \eqref{innt1}, \eqref{innt2}, \eqref{nn1}, \eqref{nn2} \eqref{conc1} and \eqref{conc2} that
\begin{equation}
    III.2=-\dot x_{1}(t)\frac{\Big\langle\partial^{2}_{x}H^{x_{1}(t)}_{-1,0},\,\partial_{t}\phi(t)\Big\rangle}{\norm{\partial_{x}H_{0,1}}^{2}}+O(\alpha(t)).
\end{equation}
This estimate of $III.2$ and the estimate \eqref{III.1est} of $III.1$ imply
\begin{equation}\label{IIIestimate}
    III=-\dot x_{1}(t)\frac{\Big\langle\partial^{2}_{x}H^{x_{1}(t)}_{-1,0},\,\partial_{t}\phi(t)\Big\rangle}{\norm{\partial_{x}H_{0,1}}^{2}}+O(\alpha(t)).
\end{equation}
In conclusion, from the estimates $II=O(\alpha(t)),$ \eqref{IIIestimate} and the definition of $I,$ we have that
$I+II+III=O(\alpha(t)).$ \\
 \textbf{Step 4.}(Estimate of $V.$)
 We recall that
 $V=-\frac{\langle\partial_{x}\chi_{0}(t)g(t),\,\partial^{2}_{t}\phi(t) \rangle}{\norm{\partial_{x}H_{0,1}}_{L^{2}}^{2}},$
and that
\begin{equation}\label{stst}
    \partial^{2}_{t}\phi(t)=\partial^{2}_{x}g(t)+\Big[\dot U\left(H^{x_{1}(t)}_{-1,0}\right)+\dot U\left(H^{x_{2}(t)}_{0,1}\right)-\dot U\left(H^{x_{1}(t)}_{-1,0}+H^{x_{2}(t)}_{0,1}\right)\Big]+\Big[\dot U\left(H^{x_{1}(t)}_{-1,0}+H^{x_{2}(t)}_{0,1}\right)-\dot U\left(H^{x_{1}(t)}_{-1,0}+H^{x_{2}(t)}_{0,1}+g(t)\right)\Big].
\end{equation}
First, by integration by parts, using estimate \eqref{linfty}, we have the following estimate
\begin{equation}\label{V.1}
   -\frac{1}{\norm{\partial_{x}H_{0,1}}_{L^{2}}^{2}}\langle\partial_{x}\chi_{0}(t) \partial^{2}_{x}g(t),\,g(t)\rangle=O\Big(\Big[\frac{1}{\gamma z(t)}+\frac{1}{\gamma^{2}z(t)^{2}}\Big]\norm{\overrightarrow{g(t)}}^{2}\Big)=O(\alpha(t)).
\end{equation}
Second, since $U$ is smooth and $\norm{g(t)}_{L^{\infty}}=O\big(\epsilon^{\frac{1}{2}}\big)$ for all $t\in\mathbb{R},$ we deduce that 
\begin{equation}\label{V2}
   \left \langle \dot U\left(H^{x_{1}(t)}_{-1,0}+H^{x_{2}(t)}_{0,1}\right)-\dot U\left(H^{x_{1}(t)}_{-1,0}+H^{x_{2}(t)}_{0,1}+g(t)\right),\partial_{x}\chi_{0}(t)g(t)\right\rangle=O\Big(\frac{1}{z(t)\gamma}\norm{\overrightarrow{g(t)}}^{2}\Big)=O(\alpha(t)). 
\end{equation}
Next, from equation \eqref{intterm} and Lemma \ref{interact}, we have that
\begin{equation}\label{estimatecc}
    \norm{\dot U\left(H^{x_{1}(t)}_{-1,0}\right)+\dot U\left(H^{x_{2}(t)}_{0,1}\right)-\dot U\left(H^{x_{1}(t)}_{-1,0}+H^{x_{2}(t)}_{0,1}\right) }_{L^{2}(\mathbb{R})}=O(e^{-\sqrt{2}z(t)}),
\end{equation}
then, by Hölder inequality we have that
\begin{equation}\label{V3}
    \left \langle \dot U\left(H^{x_{1}(t)}_{-1,0}\right)+\dot U\left(H^{x_{2}(t)}_{0,1}\right)-\dot U\left(H^{x_{1}(t)}_{-1,0}+H^{x_{2}(t)}_{0,1}\right),\partial_{x}\chi_{0}(t)\partial_{x}g(t)\right \rangle=O\Big(\frac{1}{\gamma z(t)}\norm{\overrightarrow{g(t)}}e^{-\sqrt{2}z(t)}\Big)
    =O(\alpha(t)).
\end{equation}
Clearly, the estimates \eqref{V.1}, \eqref{V2} and \eqref{V3} imply that $V=O(\alpha(t)).$\\
\textbf{Step 5.}(Estimate of $VI.$)
We know that
\begin{equation*}
    VI=-\frac{\big\langle \partial_{x}g(t)\chi_{0}(t),\,\partial^{2}_{t}\phi(t)\big\rangle}{\norm{\partial_{x}H_{0,1}}_{L^{2}}^{2}}.
\end{equation*}
We recall the equation \eqref{stst} which implies that
\begin{multline*}
   \norm{\partial_{x}H_{0,1}}_{L^{2}}^{2} VI=-\left\langle \partial_{x}g(t)\chi_{0}(t),\,\partial^{2}_{x}g(t)\right\rangle
   +\left\langle \partial_{x}g(t)\chi_{0}(t),\,\dot U\left(H^{x_{1}(t)}_{-1,0}+H^{x_{2}(t)}_{0,1}+g(t)\right)-\dot U\left(H^{x_{1}(t)}_{-1,0}+H^{x_{2}(t)}_{0,1}\right)\right\rangle\\
   +\left \langle \partial_{x}g(t)\chi_{0}(t),\, \dot U\left(H^{x_{1}(t)}_{-1,0}+H^{x_{2}(t)}_{0,1}\right)-\dot U\left(H^{x_{1}(t)}_{-1,0}\right)-\dot U\left(H^{x_{2}(t)}_{0,1}\right)\right\rangle.
\end{multline*}
By integration by parts, we have from estimate \eqref{linfty} that
\begin{equation}\label{VI.1}
    \langle \partial_{x}g(t,x)\chi_{0}(t,x),\,\partial^{2}_{x}g(t,x)\rangle=O\Big(\frac{1}{\gamma z(t)}\norm{\overrightarrow{g(t)}}^{2}\Big).
\end{equation}
From the estimate \eqref{estimatecc} and Cauchy-Schwarz inequality, we can obtain the following estimate
\begin{equation}\label{VI.2}
    \left\langle \partial_{x}g(t)\chi_{0}(t),\, \dot U\left(H^{x_{1}(t)}_{-1,0}+H^{x_{2}(t)}_{0,1}\right)-\dot U\left(H^{x_{1}(t)}_{-1,0}\right)-\dot U\left(H^{x_{2}(t)}_{0,1}\right)\right\rangle=O\Big(e^{-\sqrt{2}z(t)}\norm{\overrightarrow{g(t)}}\Big).
\end{equation}
Then, to conclude the estimate of $VI$ we just need to study the following term
$C(t)\coloneqq\big\langle \partial_{x}g(t)\chi_{0}(t),\,\dot U(H^{x_{1}(t)}_{-1,0}+H^{x_{2}(t)}_{0,1}+g(t))-\dot U(H^{x_{1}(t)}_{-1,0}+H^{x_{2}(t)}_{0,1})\big\rangle.$
Since we have from the Taylor's theorem that
\begin{equation*}
   \dot U\left(H^{x_{1}(t)}_{-1,0}+H^{x_{2}(t)}_{0,1}+g(t)\right)-\dot U\left(H^{x_{1}(t)}_{-1,0}+H^{x_{2}(t)}_{0,1}\right)=\sum_{k=2}^{6} U^{(k)}\left(H^{x_{1}(t)}_{-1,0}+H^{x_{2}(t)}_{0,1}\right)\frac{g(t)^{k-1}}{(k-1)!},
\end{equation*}
from estimate \eqref{linfty}, we can deduce by integration by parts that
\begin{equation*}
    C(t)=-\Big\langle\chi_{0}(t)\partial_{x}\left(H^{x_{1}(t)}_{-1,0}+H^{x_{2}(t)}_{0,1}\right),\,\sum_{k=3}^{6} U^{(k)}\left(H^{x_{1}(t)}_{-1,0}+H^{x_{2}(t)}_{0,1}\right)\frac{g(t)^{k-1}}{(k-1)!} \Big\rangle+O\left(\frac{1}{\gamma z(t)}\norm{\overrightarrow{g(t)}}^{2}\right).
\end{equation*}
Since
\begin{equation*}
    \norm{\chi_{0}(t)\partial_{x}H^{x_{2}(t)}_{0,1}}_{L^{\infty}}+
    \norm{(1-\chi_{0}(t))\partial_{x}H^{x_{1}(t)}_{-1,0}}_{L^{\infty}}=O\Big(e^{-\sqrt{2}z(t)(\frac{1-\gamma}{2-\gamma})}\Big),
\end{equation*}
we obtain that
\begin{multline*}
    C(t)=-\Big\langle\partial_{x}H^{x_{1}(t)}_{-1,0},\,\sum_{k=3}^{6} U^{(k)}\left(H^{x_{1}(t)}_{-1,0}+H^{x_{2}(t)}_{0,1}\right)\frac{g(t)^{k-1}}{(k-1)!} \Big\rangle+O\left(\frac{1}{\gamma z(t)}\norm{\overrightarrow{g(t)}}^{2}+e^{-\sqrt{2}z(t)(\frac{1-\gamma}{2-\gamma})}\norm{\overrightarrow{g(t)}}^{2}\right).
\end{multline*}
Also, from Lemma \ref{interact} and the fact that $\norm{g(t)}_{L^{\infty}}\lesssim\norm{\overrightarrow{g(t)}}$, we deduce that
\begin{equation}
    \Big\langle\partial_{x}H^{x_{1}}_{-1,0} ,\Big[\ddot U\left(H^{x_{1}(t)}_{-1,0}\right)-\ddot U\left(H^{x_{1}(t)}_{-1,0}+H^{x_{2}(t)}_{0,1}\right)\Big]g(t)\Big\rangle=O\Big(e^{-\sqrt{2}z(t)}\norm{\overrightarrow{g(t)}}\Big).
\end{equation}
In conclusion, we obtain that
\begin{multline}
    C(t)=-\int_{\mathbb{R}}\partial_{x}H^{x_{1}(t)}_{-1,0}\Big(\dot U\left(H^{x_{1}(t)}_{-1,0}+H^{x_{2}(t)}_{0,1}+g(t)\right)-\dot U\left(H^{x_{1}(t)}_{-1,0}+H^{x_{2}(t)}_{0,1}\right)\Big)\,dx
    +\int_{\mathbb{R}}\partial_{x}H^{x_{1}(t)}_{-1,0}\ddot U\left(H^{x_{1}(t)}_{-1,0}\right)g(t,x)\,dx\\
    +O(\alpha(t)).
\end{multline}
So
\begin{multline}\label{VI.est}
    VI=\frac{-\int_{\mathbb{R}}\partial_{x}H^{x_{1}(t)}_{-1,0}\Big(\dot U\left(H^{x_{1}(t)}_{-1,0}+H^{x_{2}(t)}_{0,1}+g(t)\right)-\dot U(H^{x_{1}(t)}_{-1,0}+H^{x_{2}(t)}_{0,1})\Big)\,dx}{\norm{\partial_{x}H_{0,1}}_{L^{2}}^{2}}\\
    +\frac{\int_{\mathbb{R}}\partial_{x}H^{x_{1}(t)}_{-1,0}\ddot U\left(H^{x_{1}(t)}_{-1,0}\right)g(t,x)\,dx}{\norm{\partial_{x}H_{0,1}}_{L^{2}}^{2}}+O(\alpha(t)).
\end{multline}
\textbf{Step 6.}(Sum of $IV,\,VI.$)
From the identities \eqref{stst} and
\begin{equation*}
IV= -\frac{\big\langle \partial_{x}H^{x_{1}(t)}_{-1,0},\,\partial^{2}_{t}\phi(t)\big\rangle}{\norm{\partial_{x}H_{0,1}}_{L^{2}}^{2}},
\end{equation*}
we obtain that
\begin{multline}\label{Iv.est}
IV=-\frac{\Big\langle \partial^{2}_{x}g(t)-\Big(\dot U\left(H^{x_{1}(t)}_{-1,0}+H^{x_{2}(t)}_{0,1}+g(t)\right)-\dot U\left(H^{x_{1}(t)}_{-1,0}+H^{x_{2}(t)}_{0,1}\right)\Big),\,\partial_{x}H^{x_{1}(t)}_{-1,0}\Big\rangle}{\norm{\partial_{x}H_{0,1}}_{L^{2}}^{2}}\\
-\frac{\Big\langle\dot U\left(H^{x_{1}(t)}_{-1,0}\right)+\dot U\left(H^{x_{2}(t)}_{0,1}\right)-U\left(H^{x_{1}(t)}_{-1,0}+H^{x_{2}(t)}_{0,1}\right),\,\partial_{x}H^{x_{1}(t)}_{-1,0}\Big\rangle}{\norm{\partial_{x}H_{0,1}}_{L^{2}}^{2}}.
\end{multline}
In conclusion, from the identity
\begin{equation*}
   \Big[ \partial_{x}^{2}-\ddot U\left(H^{x_{1}(t)}_{-1,0}\right)\Big]\partial_{x}H^{x_{1}(t)}_{-1,0}=0
\end{equation*}
and by integration by parts we have that
\begin{equation}
    IV+VI=-\frac{\Big\langle\dot U\left(H^{x_{1}(t)}_{-1,0}\right)+\dot U\left(H^{x_{2}(t)}_{0,1}\right)-U\left(H^{x_{1}(t)}_{-1,0}+H^{x_{2}(t)}_{0,1}\right),\,\partial_{x}H^{x_{1}(t)}_{-1,0}\Big\rangle}{\norm{\partial_{x}H_{0,1}}_{L^{2}}^{2}}+O(\alpha(t)).
\end{equation}
From our previous results,
we conclude that
\begin{equation}\label{fini}
    I+II+III+IV+V+VI=\\
    -\frac{\Big\langle\dot U\left(H^{x_{1}(t)}_{-1,0}\right)+\dot U\left(H^{x_{2}(t)}_{0,1}\right)-\dot U\left(H^{x_{1}(t)}_{-1,0}+H^{x_{2}(t)}_{0,1}\right),\,\partial_{x}H^{x_{1}(t)}_{-1,0}\Big\rangle}{\norm{\partial_{x}H_{0,1}}_{L^{2}}^{2}}+O(\alpha(t)).
\end{equation}
The conclusion of the lemma follows from estimate \eqref{fini} with identity 
\begin{equation*}
    \dot A(z(t))=-\Big\langle\dot U\left(H_{-1,0}\right)+\dot U\left(H^{z(t)}_{0,1}\right)-\dot U\left(H_{-1,0}+H^{z(t)}_{0,1}\right),\,\partial_{x}H_{-1,0}\Big\rangle,
\end{equation*}
which can be obtained from \eqref{L.1} by integration by parts with the fact that 
\begin{equation*}
\left\langle \dot U\left(H_{-1,0}+H^{z(t)}_{0,1}\right),\,\partial_{x} H_{-1,0}+\partial_{x}H^{z(t)}_{0,1}\right\rangle=0.
\end{equation*}
\end{proof}
\begin{remark}\label{odez}
Since, we know from Lemma \ref{interact} that
\begin{equation*}
   \md{\dot A(z(t))+4e^{-\sqrt{2}z(t)}}\lesssim z(t)e^{-2\sqrt{2}z(t)},
\end{equation*}
and, by elementary calculus with change of variables, that
$\norm{\partial_{x}H_{0,1}}_{L^{2}}^{2}=\frac{1}{2\sqrt{2}},$
then the estimates \eqref{Modul1} and \eqref{Modul2} obtained in Lemma \ref{modueq} motivate us to study the following ordinary differential equation
\begin{equation}\label{theode}
    \ddot z(t)=16\sqrt{2}e^{-\sqrt{2}z(t)}.
\end{equation}
Clearly, the solution of \eqref{theode} satisfies the equation
\begin{equation}\label{odeident}
    \frac{d}{dt}\Big[\frac{\dot z(t)^{2}}{4}+8e^{-\sqrt{2}z(t)}\Big]=0.
\end{equation}
As a consequence, it can be verified that if $z(t_{0})>0$ for some $t_{0}\in\mathbb{R},$ then there are real constants $v> 0,\,c$ such that
\begin{equation}\label{smooth}
    z(t)=\frac{1}{\sqrt{2}}\ln{\Big(\frac{8}{v^{2}}\cosh{\big(\sqrt{2}vt+c\big)}^{2}\Big)} \text{ for all $t\in\mathbb{R}.$}
\end{equation}
In conclusion, the solution of the equations
\begin{align*}
    \ddot d_{1}(t)=-8\sqrt{2}e^{-\sqrt{2}z(t)},\\
    \ddot d_{2}(t)=8\sqrt{2}e^{-\sqrt{2}z(t)},\\
    d_{2}(t)-d_{1}(t)=z(t)>0,
\end{align*}
are given by
\begin{align}\label{sol1}
    d_{2}(t)=a+b t+\frac{1}{2\sqrt{2}}\ln{\Big(\frac{8}{v^{2}}\cosh{\big(\sqrt{2}vt+c\big)}^{2}\Big)},\\\label{sol2}
    d_{1}(t)=a+b t-\frac{1}{2\sqrt{2}}\ln{\Big(\frac{8}{v^{2}}\cosh{\big(\sqrt{2}vt+c\big)}^{2}\Big)},
\end{align}
for $a,\,b$ real constants. So, we now are motivated to study how close the modulations parameters $x_{1},\,x_{2}$ of Theorem \ref{Stab} can be to functions $d_{1},\,d_{2}$ satisfying, respectively the identities \eqref{sol2} and \eqref{sol1} for constants $v\neq 0,\, a,\,b$.
\end{remark}
At first view, the statement of the Lemma \ref{modueq} seems too complex and unnecessary for use and that a simplified version should be more useful for our objectives. However, we will show later that for a suitable choice of $\gamma$ depending on the energy excess of the solution $\phi(t),$ we can get a high precision in the approximation of the modulation parameters $x_{1},\,x_{2}$ by smooth functions $d_{1},\,d_{2}$ satisfying \eqref{sol2} and \eqref{sol1} for a large time interval.
\section{Energy Estimate Method}
\par Before applying Lemma \ref{modueq}, we  need to construct a functional $F(t)$ to get lower estimate on the value of $\norm{(g(t),\partial_{t}g(t))}_{H^{1}\times L^{2}}$ than that obtained in Theorem \ref{Stab}.  
\par From now on, we consider $\phi(t)=H_{0,1}(x-x_{2}(t))+H_{-1,0}(x-x_{1}(t))+g(t,x)$, with $x_{1}(t),\,x_{2}(t)$ satisfying the orthogonality conditions of the Modulation Lemma and $x_{1},\,x_{2}$,  $(g(t),\partial_{t}g(t))$ and $\epsilon>0$ satisfying all the properties of Theorem \ref{Stab}. Before the enunciation of the main theorem of this section, to simplify the notation in computations, we denote:  
\begin{equation*}
    D^{2}E(H^{x_{2}(t)}_{0,1}+H^{x_{1}(t)}_{-1,0})=
    \begin{bmatrix}
    -\partial^{2}_{x}+ \ddot U(H^{x_{2}(t)}_{0,1}+H^{x_{1}(t)}_{-1,0}) & 0\\
    0 & \mathbb{I}
    \end{bmatrix}
\end{equation*}
as a bilinear operator from $H^{1}(\mathbb{R})\times L^{2}(\mathbb{R})$ to $\mathbb{C}.$
We also denote $\omega_{1}(t,x)=\omega\big(\frac{x-x_{1}(t)}{x_{2}(t)-x_{1}(t)}\big)$ for $\omega$ a smooth cut function with image contained in the interval $[0,1]$, satisfying the following condition 
\begin{equation*}\label{C}
\omega(x)=\begin{cases}
 1, \text{ if $x\leq \frac{3}{4}$,}\\
 0, \text{ if $x\geq \frac{4}{5}$.}
\end{cases}
\end{equation*} We consider now the following functional
\begin{multline}\label{F(t)}
    F(t)=\big\langle D^{2}E(H^{x_{2}(t)}_{0,1}+H^{x_{1}(t)}_{-1,0})\overrightarrow{g(t)},\,\overrightarrow{g(t)}\big\rangle_{L^{2}\times L^{2}}+2\int_{\mathbb{R}}\partial_{t}g(t,x)\partial_{x}g(t,x)\Big[\dot x_{1}(t)\omega_{1}(t,x)+\dot x_{2}(t)(1-\omega_{1}(t,x))\Big]\,dx\\
    -2\int_{\mathbb{R}}g(t,x)\Big(\dot U(H^{x_{1}(t)}_{-1,0}(x))+\dot U(H^{x_{2}(t)}_{0,1}(x))-\dot U(H^{x_{2}(t)}_{0,1}(x)+H^{x_{1}(t)}_{-1,0}(x))\Big)\,dx\\
    +2\int_{\mathbb{R}}g(t,x)\Big[(\dot x_{1}(t))^{2}\partial^{2}_{x}H^{x_{1}(t)}_{-1,0}(x)+(\dot x_{2}(t))^{2}\partial^{2}_{x}H^{x_{2}(t)}_{0,1}(x)\Big]\,dx
    +\frac{1}{3}\int_{\mathbb{R}}U^{(3)}(H^{x_{2}(t)}_{0,1}(x)+H^{x_{1}(t)}_{-1,0}(x))g(t,x)^{3}\,dx.
\end{multline}
Since $x_{1},\,x_{2}$ are functions of class $C^{2}$, is not difficult to verify that $(g(t),\partial_{t}g(t))$ solves the integral equation associated to the following partial differential equation 
\begin{multline}\label{NLW2}\tag{II}
    \partial^{2}_{t}g(t,x)-\partial^{2}_{x}g(t,x)+\ddot U(H^{x_{2}(t)}_{0,1}(x)+H^{x_{1}(t)}_{-1,0}(x))g(t,x)=\\-\left[
\dot U(H^{x_{2}(t)}_{0,1}(x)+H^{x_{1}(t)}_{-1,0}(x)+g(t,x))-\dot U(H^{x_{2}(t)}_{0,1}(x)+H^{x_{1}(t)}_{-1,0}(x))-\ddot U(H^{x_{2}(t)}_{0,1}(x)+H^{x_{1}(t)}_{-1,0}(x))g(t,x)\right]\\+\dot U(H^{x_{1}(t)}_{-1,0}(x))+\dot U(H^{x_{2}(t)}_{0,1}(x))-\dot U(H^{x_{2}(t)}_{0,1}(x)+H^{x_{1}(t)}_{-1,0}(x))\\
    -\dot x_{1}(t)^{2}\partial^{2}_{x}H^{x_{1}(t)}_{-1,0}(x)-\dot x_{2}(t)^{2}\partial^{2}_{x}H^{x_{2}(t)}_{0,1}(x)
    +\ddot x_{1}(t)\partial_{x}H^{x_{1}(t)}_{-1,0}(x)+\ddot x_{2}(t)\partial_{x}H^{x_{2}(t)}_{0,1}(x)
\end{multline}
in the space $H^{1}(\mathbb{R})\times L^{2}(\mathbb{R}).$
\begin{theorem}\label{EnergyE}
Assuming the hypotheses of Theorem \ref{Stab} and recalling its notation, let $\delta(t)$ be the following quantity
\begin{multline*}
    \delta(t)=\norm{\overrightarrow{g(t)}}\Big(e^{-\sqrt{2}z(t)}\max_{j \in \{1,2\}}\md{\dot x_{j}(t)}+\max_{j \in \{1,2\}}\md{\dot x_{j}(t)}^{3}e^{-\frac{\sqrt{2}z(t)}{5}}+\max_{j \in \{1,2\}}\md{\dot x_{j}(t) \ddot x_{j}(t)}\Big)\\+\norm{\overrightarrow{g(t)}}^{2}\Big(\frac{\max_{j\in\{1,\,2\}}\md{\dot x_{j}(t)}}{z(t)}+\max_{j\in\{1,\,2\}}\dot x_{j}(t)^{2}+\max_{j\in\{1,\,2\}}\md{\ddot x_{j}(t)}\Big)+\norm{\overrightarrow{g(t)}}^{4}.
\end{multline*}
Then, $\exists$ positive constants $A_{1},\,A_{2},\,A_{3}$ such that the functional $F(t)$ satisfies the inequalities
    \begin{equation*}
        F(t)+A_{1}\epsilon^{2}\geq A_{2}\norm{\overrightarrow{g(t)}}^{2},
         \,\md{\dot F(t)}\leq A_{3}\delta(t).
    \end{equation*}
\end{theorem}
\begin{remark}\label{ttt1}
Theorem \ref{Stab} and Theorem \ref{EnergyE} imply
\begin{equation*}
    \md{\dot F(t)}\lesssim \frac{\epsilon^{\frac{1}{2}}}{\ln{(\frac{1}{\epsilon})}}\norm{\overrightarrow{g(t)}}^{2}+\norm{\overrightarrow{g(t)}}\epsilon^{\frac{3}{2}}.
\end{equation*}
\end{remark}
\begin{proof}
Since the formula defining function $F(t)$ is very large, we decompose the function in a sum of five terms $F_{1},\, F_{2},\, F_{3},\,F_{4}$ and $F_{5}$. More specifically:
\begin{align*}
    F_{1}(t)=\int_{\mathbb{R}}\partial_{t}g(t,x)^{2}+\partial_{x}g(t,x)^{2}+\ddot U(H^{x_{1}(t)}_{-1,0}(x)+H^{x_{2}(t)}_{0,1}(x))g(t,x)^{2}\,dx,\\
    F_{2}(t)=-2\int_{\mathbb{R}}g(t,x)\big[\dot U(H^{x_{1}(t)}_{-1,0}(x))+\dot U(H^{x_{2}(t)}_{0,1}(x))-\dot U(H^{x_{2}(t)}_{0,1}(x)+H^{x_{1}(t)}_{-1,0}(x))\big]\,dx,\\
    F_{3}(t)=2\int_{\mathbb{R}}g(t,x)\big[\dot x_{1}(t)^{2}\partial^{2}_{x}H^{x_{1}(t)}_{-1,0}(x)+\dot x_{2}(t)^{2}\partial^{2}_{x}H^{x_{2}(t)}_{0,1}(x)\big]\,dx,\\
    F_{4}(t)=2\int_{\mathbb{R}}\partial_{t}g(t,x)\partial_{x}g(t,x)(\dot x_{1}(t)\omega_{1}(t,x)+\dot x_{2}(t)(1-\omega_{1}(t,x)))\,dx,\\
    F_{5}(t)=\frac{1}{3}\int_{\mathbb{R}}U^{(3)}(H^{x_{2}(t)}_{0,1}(x)+H^{x_{1}(t)}_{-1,0}(x))g(t,x)^{3}\,dx.
\end{align*}
 First, We prove that $\md{\dot F(t)}\lesssim \delta(t).$ 
The main idea of the proof of this item is to estimate each derivative $\frac{dF_{j}(t)}{dt},$ for $1\leq j\leq 5,$ with an error of size $O(\delta(t)),$ then we will check that the sum of these estimates are going to be a value of order $O(\delta(t)),$ which means that the estimates of these derivatives cancel.\\
\textbf{Step 1.}(The derivative of $F_{1}(t).$) By definition of $F_{1}(t),$ we have that
    \begin{multline}
        \frac{dF_{1}(t)}{dt}=2\int_{\mathbb{R}}\left(\partial^{2}_{t}g(t,x)-\partial^{2}_{x}g(t,x)+\ddot U(H^{x_{2}(t)}_{0,1}(x)+H^{x_{1}(t)}_{-1,0}(x))g(t,x)\right)\partial_{t}g(t,x)\,dx\\
        -\int_{\mathbb{R}}\left(\dot x_{1}(t)\partial_{x}H^{x_{1}(t)}_{-1,0}(x)+\dot x_{2}(t)\partial_{x}H^{x_{2}(t)}_{0,1}(x)\right)U^{(3)}\left(H^{x_{2}(t)}_{0,1}(x)+H^{x_{1}(t)}_{-1,0}(x)\right)g(t,x)^{2}\,dx.
    \end{multline}
Moreover, from the identity \eqref{NLW2} satisfied by $g(t,x)$, we can rewrite the value of $\frac{dF_{1}(t)}{dt}$ as
\begin{multline*}
    \frac{dF_{1}(t)}{dt}=2\int_{\mathbb{R}}\left[\dot U\left(H^{x_{1}(t)}_{-1,0}(x)\right)+\dot U\left(H^{x_{2}(t)}_{0,1}(x)\right)-\dot U\left(H^{x_{1}(t)}_{-1,0}(x)+H^{x_{2}(t)}_{0,1}(x)\right)\right]\partial_{t}g(t,x)\,dx\\
    -2\int_{\mathbb{R}}\left[U\left(H^{x_{2}(t)}_{0,1}(x)+H^{x_{1}(t)}_{-1,0}(x)+g(t,x)\right)-\dot U\left(H^{x_{1}(t)}_{-1,0}(x)+H^{x_{2}(t)}_{0,1}(x)\right)-\ddot U\left(H^{x_{2}(t)}_{0,1}(x)+H^{x_{1}(t)}_{-1,0}(x)\right)g(t,x)\right]\partial_{t}g(t,x)\,dx
    \\-2\int_{\mathbb{R}} \Big[\dot x_{1}(t)^{2}\partial^{2}_{x}H^{x_{1}(t)}_{-1,0}(x)+\dot x_{2}(t)^{2}\partial^{2}_{x}H^{x_{2}(t)}_{0,1}(x)\Big]\partial_{t}g(t,x)\,dx
    +2\int_{\mathbb{R}}\Big[\ddot x_{1}(t)\partial_{x}H^{x_{1}(t)}_{-1,0}(x)+\ddot x_{2}(t)\partial_{x}H^{x_{2}(t)}_{0,1}(x)\Big]\partial_{t}g(t,x)\,dx\\
    -\int_{\mathbb{R}}\Big[\dot x_{1}(t)\partial_{x}H^{x_{1}(t)}_{-1,0}(x)+\dot x_{2}(t)\partial_{x}H^{x_{2}(t)}_{0,1}(x)\Big]U^{(3)}(H^{x_{2}(t)}_{0,1}(x)+H^{x_{1}(t)}_{-1,0}(x))g(t,x)^{2}\,dx,
\end{multline*}
and, from the orthogonality conditions of the Modulation Lemma, we obtain
\begin{multline*}
    \frac{dF_{1}(t)}{dt}=2\int_{\mathbb{R}}\left[\dot U\left(H^{x_{1}(t)}_{-1,0}(x)\right)+\dot U\left(H^{x_{2}(t)}_{0,1}(x)\right)-\dot U\left(H^{x_{1}(t)}_{-1,0}(x)+H^{x_{2}(t)}_{0,1}(x)\right)\right]\partial_{t}g(t,x)\,dx\\
    -2\int_{\mathbb{R}}\left[U\left(H^{x_{2}(t)}_{0,1}(x)+H^{x_{1}(t)}_{-1,0}(x)+g(t,x)\right)-\dot U\left(H^{x_{1}(t)}_{-1,0}(x)+H^{x_{2}(t)}_{0,1}(x)\right)-\ddot U\left(H^{x_{2}(t)}_{0,1}(x)+H^{x_{1}(t)}_{-1,0}(x)\right)g(t,x)\right]\partial_{t}g(t,x)\,dx\\
    -2\int_{\mathbb{R}} \Big[\dot x_{1}(t)^{2}\partial^{2}_{x}H^{x_{1}(t)}_{-1,0}(x)+\dot x_{2}(t)^{2}\partial^{2}_{x}H^{x_{2}(t)}_{0,1}(x)\Big]\partial_{t}g(t,x)\,dx
    +2\int_{\mathbb{R}}\Big[\ddot x_{1}(t)\dot x_{1}(t)\partial^{2}_{x}H^{x_{1}(t)}_{-1,0}(x)+\ddot x_{2}(t)\dot x_{2}(t)\partial^{2}_{x}H^{x_{2}(t)}_{0,1}(x)\Big]g(t,x)\,dx\\
    -\int_{\mathbb{R}}\Big[\dot x_{1}(t)\partial_{x}H^{x_{1}(t)}_{-1,0}(x)+\dot x_{2}(t)\partial_{x}H^{x_{2}(t)}_{0,1}(x)\Big]U^{(3)}\left(H^{x_{2}(t)}_{0,1}(x)+H^{x_{1}(t)}_{-1,0}(x)\right)g(t,x)^{2}\,dx,
\end{multline*}
which implies
\begin{multline}\label{F1}
     \frac{dF_{1}(t)}{dt}=2\int_{\mathbb{R}}\left[\dot U\left(H^{x_{1}(t)}_{-1,0}(x)\right)+\dot U\left(H^{x_{2}(t)}_{0,1}(x)\right)-\dot U\left(H^{x_{1}(t)}_{-1,0}(x)+H^{x_{2}(t)}_{0,1}(x)\right)\right]\partial_{t}g(t,x)\,dx\\
    -2\int_{\mathbb{R}}\left[U\left(H^{x_{2}(t)}_{0,1}(x)+H^{x_{1}(t)}_{-1,0}(x)+g(t,x)\right)-\dot U\left(H^{x_{1}(t)}_{-1,0}(x)+H^{x_{2}(t)}_{0,1}(x)\right)-\ddot U\left(H^{x_{2}(t)}_{0,1}(x)+H^{x_{1}(t)}_{-1,0}(x)\right)g(t,x)\right]\partial_{t}g(t,x)\,dx\\
    -2\int_{\mathbb{R}} \Big[\dot x_{1}(t)^{2}\partial^{2}_{x}H^{x_{1}(t)}_{-1,0}(x)+\dot x_{2}(t)^{2}\partial^{2}_{x}H^{x_{2}(t)}_{0,1}(x)\Big]\partial_{t}g(t,x)\,dx
    \\-\int_{\mathbb{R}}\Big[\dot x_{1}(t)\partial_{x}H^{x_{1}(t)}_{-1,0}(x)+\dot x_{2}(t)\partial_{x}H^{x_{2}(t)}_{0,1}(x)\Big]U^{(3)}\left(H^{x_{2}(t)}_{0,1}(x)+H^{x_{1}(t)}_{-1,0}(x)\right)g(t,x)^{2}\,dx
    +O(\delta(t)).
\end{multline}
\textbf{Step 2.}(The derivative of $F_{2}(t).$)
It is not difficult to verify that
\begin{multline*}
    \frac{dF_{2}(t)}{dt}=-2\int_{\mathbb{R}}\partial_{t}g(t,x)\left[\dot U\left(H^{x_{1}(t)}_{-1,0}(x)\right)+\dot U\left(H^{x_{2}(t)}_{0,1}(x)\right)-\dot U\left(H^{x_{2}(t)}_{0,1}(x)+H^{x_{1}(t)}_{-1,0}(x)\right)\right]\,dx\\
    +2\int_{\mathbb{R}}g(t,x)\left[\ddot U\left(H^{x_{1}(t)}_{-1,0}(x)\right)\partial_{x}H^{x_{1}(t)}_{-1,0}(x)\dot x_{1}(t)+\ddot U\left(H^{x_{2}(t)}_{0,1}(x)\right)\partial_{x}H^{x_{2}(t)}_{0,1}(x)\dot x_{2}(t)\right]\,dx\\
    -2\int_{\mathbb{R}} \ddot U\left(H^{x_{2}(t)}_{0,1}(x)+H^{x_{1}(t)}_{-1,0}(x)\right)\left[\partial_{x}H^{x_{1}(t)}_{-1,0}(x)\dot x_{1}(t)+\partial_{x}H^{x_{2}(t)}_{0,1}(x)\dot x_{2}(t)\right]g(t,x)\,dx.
\end{multline*}
Since from the definition of the function $U$, we can deduce that
\begin{align*}
    \left|\ddot U\left(H^{x_{2}(t)}_{0,1}(x)+H^{x_{1}(t)}_{-1,0}(x)\right)-\ddot U\left(H^{x_{1}(t)}_{-1,0}(x)\right)\right|=O\Big(\left|H^{x_{1}(t)}_{-1,0}(x)H^{x_{2}(t)}_{0,1}(x)\right|+\left|H^{x_{2}(t)}_{0,1}(x)\right|^{2}\Big),\\
    |\ddot U\left(H^{x_{2}(t)}_{0,1}(x)+H^{x_{1}(t)}_{-1,0}(x)\right)-\ddot U\left(H^{x_{2}(t)}_{0,1}(x)\right)|=O\Big(\left|H^{x_{1}(t)}_{-1,0}(x)H^{x_{2}(t)}_{0,1}(x)\right|+\left|H^{x_{1}(t)}_{-1,0}(x)\right|^{2}\Big),
\end{align*}
we obtain from Lemma \ref{interact} and Cauchy-Schwarz Inequality that
\begin{align*}
    \left|\int_{\mathbb{R}}\left[\ddot U\left(H^{x_{2}(t)}_{0,1}(x)\right)-\ddot U\left(H^{x_{2}(t)}_{0,1}(x)+H^{x_{1}(t)}_{-1,0}(x)\right)\right]\partial_{x}H^{x_{2}(t)}_{0,1}(x)g(t,x)\,dx\right|=O\Big(\norm{\overrightarrow{g(t)}}e^{-\sqrt{2}z(t)}\Big),\\
    \left|\int_{\mathbb{R}}\left[\ddot U\left(H^{x_{1}(t)}_{-1,0}(x)\right)-\ddot U\left(H^{x_{2}(t)}_{0,1}(x)+H^{x_{1}(t)}_{-1,0}(x)\right)\right]\partial_{x}H^{x_{1}(t)}_{-1,0}(x)g(t,x)\,dx\right|=O\Big(\norm{\overrightarrow{g(t)}}e^{-\sqrt{2}z(t)}\Big).
\end{align*}
In conclusion, we obtain from the identity satisfied by $\frac{dF_{2}(t)}{dt}$ that 
\begin{equation}\label{F2}
    \frac{dF_{2}(t)}{dt}=-2\int_{\mathbb{R}}\partial_{t}g(t,x)\left[\dot U\left(H^{x_{1}(t)}_{-1,0}(x)\right)+\dot U\left(H^{x_{2}(t)}_{0,1}(x)\right)-\dot U\left(H^{x_{2}(t)}_{0,1}(x)+H^{x_{1}(t)}_{-1,0}(x)\right)\right]\,dx+O(\delta(t)).
\end{equation}
\textbf{Step 3.}(The derivative of $F_{3}(t).$)
From the definition of $F_{3}(t),$ we obtain that
\begin{multline*}
    \frac{dF_{3}(t)}{dt}=2\int_{\mathbb{R}}\partial_{t}g(t,x)\big[\dot x_{1}(t)^{2}\partial^{2}_{x}H^{x_{1}(t)}_{-1,0}(x)+\dot x_{2}(t)^{2}\partial^{2}_{x}H^{x_{2}(t)}_{0,1}(x)\big]\,dx
    \\-2\int_{\mathbb{R}}g(t,x)\big[\dot x_{1}(t)^{3}\partial^{3}_{x}H^{x_{1}(t)}_{-1,0}(x)+\dot x_{2}(t)^{3}\partial^{3}_{x}H^{x_{2}(t)}_{0,1}(x)\big]\,dx+4\int_{\mathbb{R}}g(t,x)\big[\dot x_{1}(t)\ddot x_{1}(t)\partial^{2}_{x}H^{x_{1}(t)}_{-1,0}(x)+\dot x_{2}(t)\ddot x_{2}(t)\partial^{2}_{x}H^{x_{2}(t)}_{0,1}(x)\big]\,dx,
\end{multline*}
which can be rewritten as
\begin{multline}\label{F3}
    \frac{dF_{3}(t)}{dt}=2\int_{\mathbb{R}}\partial_{t}g(t,x)\big[\dot x_{1}(t)^{2}\partial^{2}_{x}H^{x_{1}(t)}_{-1,0}(x)+\dot x_{2}(t)^{2}\partial^{2}_{x}H^{x_{2}(t)}_{0,1}(x)\big]\,dx-2\int_{\mathbb{R}}g(t,x)\big[\dot x_{1}(t)^{3}\partial^{3}_{x}H^{x_{1}(t)}_{-1,0}(x)+\dot x_{2}(t)^{3}\partial^{3}_{x}H^{x_{2}(t)}_{0,1}(x)\big]\,dx\\+O(\delta(t)).
\end{multline}
\textbf{Step 4.}(Sum of $\frac{dF_{1}}{dt}, \frac{dF_{2}}{dt}, \frac{dF_{3}}{dt}.$) If we sum the estimates \eqref{F1}, \eqref{F2} and \eqref{F3}, we obtain that
\begin{multline}
    \sum_{i=1}^{3}\frac{dF_{i}(t)}{dt}=-2\int_{\mathbb{R}}\Big[\dot U\left(H^{x_{2}(t)}_{0,1}(x)+H^{x_{1}(t)}_{-1,0}(x)+g(t,x)\right)-\dot U\left(H^{x_{2}(t)}_{0,1}(x)+H^{x_{1}(t)}_{-1,0}(x)\right)\\-\ddot U\left(H^{x_{2}(t)}_{0,1}(x)+H^{x_{1}(t)}_{-1,0}(x)\right)g(t,x)\Big]\partial_{t}g(t,x)\,dx\\
     -\int_{\mathbb{R}}\left[\dot x_{1}(t)\partial_{x}H^{x_{1}(t)}_{-1,0}(x)+\dot x_{2}(t)\partial_{x}H^{x_{2}(t)}_{0,1}(x)\right]U^{(3)}\left(H^{x_{2}(t)}_{0,1}(x)+H^{x_{1}(t)}_{-1,0}(x)\right)g(t,x)^{2}\,dx\\
     -2\int_{\mathbb{R}}g(t,x)\left[\dot x_{1}(t)^{3}\partial^{3}_{x}H^{x_{1}(t)}_{-1,0}(x)+\dot x_{2}(t)^{3}\partial^{3}_{x}H^{x_{2}(t)}_{0,1}(x)\right]\,dx
     +O(\delta(t)).
\end{multline}
More precisely, from Taylor's Expansion Theorem and since $\norm{\overrightarrow{g(t)}}^{4}\leq \delta(t),$
\begin{multline}\label{sum1}
    \sum_{i=1}^{3}\frac{dF_{i}(t)}{dt}=-\int_{\mathbb{R}}\left[U^{(3)}\left(H^{x_{2}(t)}_{0,1}(x)+H^{x_{1}(t)}_{-1,0}(x)\right)g(t,x)^{2}\right]\partial_{t}g(t,x)\,dx
     \\-\int_{\mathbb{R}}\left[\dot x_{1}(t)\partial_{x}H^{x_{1}(t)}_{-1,0}(x)+\dot x_{2}(t)\partial_{x}H^{x_{2}(t)}_{0,1}(x)\right]U^{(3)}\left(H^{x_{2}(t)}_{0,1}(x)+H^{x_{1}(t)}_{-1,0}(x)\right)g(t,x)^{2}\,dx
     \\-2\int_{\mathbb{R}}g(t,x)\left[\dot x_{1}(t)^{3}\partial^{3}_{x}H^{x_{1}(t)}_{-1,0}(x)+\dot x_{2}(t)^{3}\partial^{3}_{x}H^{x_{2}(t)}_{0,1}(x)\right]\,dx
     +O(\delta(t)).
\end{multline}
\textbf{Step 5.}(The derivative of $F_{4}(t).$) The computation of the derivative of $F_{4}(t)$ will be more careful, since the motivation for the addition of this term is to cancel with the expression
\begin{equation*}
     -\int_{\mathbb{R}}\left[\dot x_{1}(t)\partial_{x}H^{x_{1}(t)}_{-1,0}(x)+\dot x_{2}(t)\partial_{x}H^{x_{2}(t)}_{0,1}(x)\right]U^{(3)}\left(H^{x_{2}(t)}_{0,1}(x)+H^{x_{1}(t)}_{-1,0}(x)\right)g(t,x)^{2}\,dx
\end{equation*}
of \eqref{sum1}. The construction of functional $F_{4}(t)$ is based on the \textit{momentum correction term} of Lemma 4.2 of \cite{jkl}. To estimate $\frac{dF_{4}(t)}{dt}$ with precision of $O(\delta(t))$, it is just necessary to study the time derivative of
\begin{equation}\label{pp1}
    2\int_{\mathbb{R}}\partial_{t}g(t,x)\partial_{x}g(t,x)\dot x_{1}(t)\omega_{1}(t,x)\,dx,
\end{equation}
since the estimate of the other term in $F_{4}(t)$ is completely analogous.
 First, we have the identity
\begin{multline}
    \frac{d}{dt}\Big[2\int_{\mathbb{R}}\partial_{t}g(t,x)\partial_{x}g(t,x)\dot x_{1}(t)\omega_{1}(t,x)\,dx\Big]=
    2\ddot x_{1}(t)\int_{\mathbb{R}}\omega_{1}(t,x)\partial_{t}g(t,x)\partial_{x}g(t,x)\,dx\\+2\dot x_{1}(t)\int_{\mathbb{R}}\omega_{1}(t,x)\partial^{2}_{t}g(t,x)\partial_{x}g(t,x)\,dx
    +2\dot x_{1}(t)\int_{\mathbb{R}}\partial_{t}\omega_{1}(t,x)\partial_{t}g(t,x)\partial_{x}g(t,x)\,dx+
    2\dot x_{1}(t)\int_{\mathbb{R}} \omega_{1}(t,x)\partial^{2}_{t,x}g(t,x)\partial_{t}g(t,x)\,dx.
\end{multline}
From the definition of $\omega_{1}(t,x)=\omega\big(\frac{x-x_{1}(t)}{x_{2}(t)-x_{1}(t)}\big)$, we have
\begin{equation}
    \partial_{t}\omega_{1}(t,x)=\dot \omega\Big(\frac{x-x_{1}(t)}{x_{2}(t)-x_{1}(t)}\Big)\Big(\frac{-\dot x_{1}(t)z(t)-\dot z(t) (x-x_{1}(t))}{z(t)^{2}}\Big).
\end{equation}
Since in the support of $\dot\omega(x)$ is contained in the set $\frac{3}{4}\leq x \leq \frac{4}{5}$, we obtain the following estimate:
\begin{equation}\label{F4.1}
    2\dot x_{1}(t)\int_{\mathbb{R}}\partial_{t}\omega_{1}(t,x)\partial_{t}g(t,x)\partial_{x}g(t,x)\,dx=O\Big(\max_{j \in \{1,2\}}\frac{|\dot x_{j}(t)|}{z(t)}\norm{\overrightarrow{g(t)}}^{2}\Big)=O(\delta(t)).
\end{equation}
Clearly from integration by parts, we deduce that
\begin{equation}\label{F4.2}
    2\dot x_{1}(t)\int_{\mathbb{R}} \omega_{1}(t,x)\partial^{2}_{t,x}g(t,x)\partial_{t}g(t,x)\,dx=O\Big(\max_{j \in \{1,2\}}\frac{\md{\dot x_{j}(t)}}{z(t)}\norm{\overrightarrow{g(t)}}^{2}\Big)=O(\delta(t)).
\end{equation}
 Also, we have
\begin{equation}\label{F4.3}
    2\ddot x_{1}(t)\int_{\mathbb{R}}\omega_{1}(t,x)\partial_{t}g(t,x)\partial_{x}g(t,x)\,dx=O\Big(\max_{j \in \{1,2\}}\md{\ddot x_{j}(t)}\norm{\overrightarrow{g(t)}}^{2}\Big)=O(\delta(t)).
\end{equation}
So, to estimate the time derivative of \eqref{pp1} with precision $O(\delta(t)),$ it is enough to estimate
\begin{equation*}
2\dot x_{1}(t)\int_{\mathbb{R}}\omega_{1}(t,x)\partial^{2}_{t}g(t,x)\partial_{x}g(t,x)\,dx.
\end{equation*}
We have that
\begin{multline}\label{ff}
    2\dot x_{1}(t)\int_{\mathbb{R}}\omega_{1}(t,x)\partial^{2}_{t}g(t,x)\partial_{x}g(t,x)\,dx=
    2\dot x_{1}(t)\int_{\mathbb{R}}\omega_{1}(t,x)\partial^{2}_{x}g(t,x)\partial_{x}g(t,x)\,dx
    \\-2\dot x_{1}(t)\int_{\mathbb{R}}\omega_{1}(t,x) \ddot U\left(H^{x_{1}(t)}_{-1,0}(x)+H^{x_{2}(t)}_{0,1}(x)\right)g(t,x)\partial_{x}g(t,x)\,dx\\
    +2\dot x_{1}(t)\int_{\mathbb{R}}\omega_{1}(t,x)\left[\partial^{2}_{t}g(t,x)-\partial^{2}_{x}g(t,x)+\ddot U\left(H^{x_{1}(t)}_{-1,0}(x)+H^{x_{2}(t)}_{0,1}(x)\right)g(t,x)\right]\partial_{x}g(t,x)\,dx.
\end{multline}
From integration by parts, the first term of the equation \eqref{ff} satisfies
\begin{equation}\label{F4.4}
    2\dot x_{1}(t)\int_{\mathbb{R}}\omega_{1}(t,x)\partial^{2}_{x}g(t,x)\partial_{x}g(t,x)\,dx
    =O\Big(\max_{j \in \{1,2\}}\frac{\md{\dot x_{j}(t)}}{z(t)}\norm{\overrightarrow{g(t)}}^{2}\Big)=O(\delta(t)).
\end{equation}
From Taylor's Expansion Theorem, we have that
\begin{multline}\label{F4.5}
   \norm{\dot U\left(H^{x_{2}(t)}_{0,1}+H^{x_{1}(t)}_{-1,0}+g(t)\right)-\dot U\left(H^{x_{2}(t)}_{0,1}+H^{x_{1}(t)}_{-1,0}\right)-\ddot U\left(H^{x_{2}(t)}_{0,1}+H^{x_{1}(t)}_{-1,0}\right)g(t)-U^{(3)}\left(H^{x_{2}(t)}_{0,1}+H^{x_{1}(t)}_{-1,0}\right)\frac{g(t)^{2}}{2}}_{L^{2}(\mathbb{R})}\\=O\Big(\norm{\overrightarrow{g(t)}}^{3}\Big).
\end{multline}
Also, we have verified the identity
\begin{equation*}
    \dot U(\phi)+\dot U(\theta)-\dot U(\phi+\theta)=24\phi\theta(\phi+\theta)-6\Big(\sum_{j=1}^{4}\begin{pmatrix}
        5\\
        j
    \end{pmatrix}\phi^{j}\theta^{5-j}\Big ),
\end{equation*}
which clearly with the inequalities (D1), (D2) and Lemma \ref{interact} imply the estimate
\begin{equation}\label{F4.6}
    \norm{\dot U\left(H^{x_{2}(t)}_{0,1}\right)+\dot U\left(H^{x_{1}(t)}_{-1,0}\right)-\dot U\left(H^{x_{2}(t)}_{0,1}+H^{x_{1}(t)}_{-1,0}\right)}_{L^{2}(\mathbb{R})}=O(e^{-\sqrt{2}z(t)}).
\end{equation}
Finally, is not difficult to verify that
\begin{equation}\label{F4.7}
    \norm{-\dot x_{1}(t)^{2}\partial^{2}_{x}H^{x_{1}(t)}_{-1,0}-\dot x_{2}(t)^{2}\partial^{2}_{x}H^{x_{2}(t)}_{0,1}
    +\ddot x_{1}(t)\partial_{x}H^{x_{1}(t)}_{-1,0}+\ddot x_{2}(t)\partial_{x}H^{x_{2}(t)}_{0,1}}_{L^{2}(\mathbb{R})}=O\left(\max_{j\in \{1,2\}}\md{\dot x_{j}(t)}^{2}+\md{\ddot x_{j}(t)}\right).
\end{equation}
Then, from estimates \eqref{F4.5}, \eqref{F4.6} and \eqref{F4.7} and the Partial Differential Equation \eqref{NLW2} satisfied by $g(t,x)$, we can obtain the estimate
\begin{multline*}
    2\dot x_{1}(t)\int_{\mathbb{R}}\omega_{1}(t,x)\left[\partial^{2}_{t}g(t,x)-\partial^{2}_{x}g(t,x)+\ddot U\left(H^{x_{1}(t)}_{-1,0}(x)+H^{x_{2}(t)}_{0,1}(x)\right)g(t,x)\right]\partial_{x}g(t,x)\,dx=\\
    -\dot x_{1}(t)\int_{\mathbb{R}}\omega_{1}(t,x)U^{(3)}\left(H^{x_{1}(t)}_{-1,0}(x)+H^{x_{2}(t)}_{0,1}(x)\right)g(t,x)^{2}\partial_{x}g(t,x)\,dx
    -2\dot x_{1}(t)^{3}\int_{\mathbb{R}}\partial^{2}_{x}H^{x_{1}(t)}_{-1,0}(x)\partial_{x}g(t,x)\,dx\\
    -2\dot x_{1}(t)^{3}\int_{\mathbb{R}}(\omega_{1}(t,x)-1)\partial^{2}_{x}H^{x_{1}(t)}_{-1,0}(x)\partial_{x}g(t,x)\,dx-2\dot x_{1}(t)\dot x_{2}(t)^{2}\int_{\mathbb{R}}\omega_{1}(t,x)\partial^{2}_{x}H^{x_{2}(t)}_{0,1}(x)\partial_{x}g(t,x)\,dx\\
    +O\Big(\max_{j\in \{1,2\}}\md{\ddot x_{j}(t)\dot x_{j}(t)}\norm{\overrightarrow{g(t)}}+e^{-\sqrt{2}z(t)}\max_{j\in \{1,2\}}\md{\dot x_{j}(t)}\norm{\overrightarrow{g(t)}}+\norm{\overrightarrow{g(t)}}^{4}\max_{j\in \{1,2\}}\md{\dot x_{j}(t)}\Big),
\end{multline*}
which, by integration by parts and by Cauchy-Schwarz inequality using the estimate \eqref{supp1} for $\omega_{1},$ we obtain that
\begin{multline}\label{F4.7.7}
    2\dot x_{1}(t)\int_{\mathbb{R}}\omega_{1}(t,x)\left[\partial^{2}_{t}g(t,x)-\partial^{2}_{x}g(t,x)+\ddot U\left(H^{x_{1}(t)}_{-1,0}(x)+H^{x_{2}(t)}_{0,1}(x)\right)g(t,x)\right]\partial_{x}g(t,x)\,dx=\\
    \frac{\dot x_{1}(t)}{3}\int_{\mathbb{R}}\omega_{1}(t)U^{(4)}\left(H^{x_{1}(t)}_{-1,0}+H^{x_{2}(t)}_{0,1}\right)\left[\partial_{x}H^{x_{1}(t)}_{-1,0}+\partial_{x}H^{x_{2}(t)}_{0,1}\right]g(t)^{3}\,dx+O\Big(\max_{j\in \{1,2\}}\frac{\md{\dot x_{j}(t)}}{z(t)}\norm{\overrightarrow{g(t)}}^{3}\Big)\\
    -2\dot x_{1}(t)^{3}\int_{\mathbb{R}}\partial^{2}_{x}H^{x_{1}(t)}_{-1,0}(x)\partial_{x}g(t,x)\,dx+O\Big(\max_{j\in \{1,2\}}\md{\dot x_{j}(t)}^{3}e^{-\frac{\sqrt{2}z(t)}{5}}\norm{\overrightarrow{g(t)}}\Big)
    +O(\delta(t)).
\end{multline}
Now, to finish the estimate of
$2\dot x_{1}(t)\int_{\mathbb{R}}\omega_{1}(t,x)\partial^{2}_{t}g(t,x)\partial_{x}g(t,x)\,dx,$
it remains to study the integral given by
\begin{equation}
    -2\dot x_{1}(t)\int_{\mathbb{R}}\omega_{1}(t,x) \ddot U\left(H^{x_{1}(t)}_{-1,0}(x)+H^{x_{2}(t)}_{0,1}(x)\right)g(t,x)\partial_{x}g(t,x)\,dx,
\end{equation}
which by integration by parts is equal to 
\begin{equation}\label{F4F}
     \dot x_{1}(t)\int_{\mathbb{R}}\omega_{1}(t,x)U^{(3)}\left(H^{x_{1}(t)}_{-1,0}(x)+H^{x_{2}(t)}_{0,1}(x)\right)\big[\partial_{x}H^{x_{1}(t)}_{-1,0}(x)+\partial_{x}H^{x_{2}(t)}_{0,1}(x)\big]g(t,x)^{2}\,dx+O(\delta(t)).
\end{equation}
Since the support of $\omega_{1}(t,x)$ is included in $\{x|\,(x-x_{2}(t))\leq-\frac{z(t)}{5}\}$ and the support of $1-\omega_{1}(t,x)$ is included in $\{x|\,(x-x_{1}(t))\geq\frac{3z(t)}{4}\},$ from the exponential decay properties of the kink solutions in (D1), (D2), (D3), (D4) we obtain the estimates
\begin{align}\label{Fin1}
   \md{\dot x_{1}(t)\int_{\mathbb{R}}(\omega_{1}(t,x)-1)U^{(3)}\left(H^{x_{1}(t)}_{-1,0}(x)+H^{x_{2}(t)}_{0,1}(x)\right)(\partial_{x}H^{x_{1}(t)}_{-1,0}(x))g(t,x)^{2}\,dx} =O(\delta(t)),\\
\label{Fin2}
    \md{\dot x_{2}(t)\int_{\mathbb{R}}\omega_{1}(t,x)U^{(3)}\left(H^{x_{1}(t)}_{-1,0}(x)+H^{x_{2}(t)}_{0,1}(x)\right)\partial_{x}H^{x_{2}(t)}_{0,1}(x)g(t,x)^{2}\,dx}
    =O(\delta(t)),\\ \label{Fin5}
    \md{\frac{1}{3}\dot x_{1}(t)\int_{\mathbb{R}}(1-\omega_{1}(t))U^{(4)}(H^{x_{1}(t)}_{-1,0}+H^{x_{2}(t)}_{0,1})\partial_{x}H^{x_{1}(t)}_{-1,0}g(t)^{3}\,dt}
    = O(\delta(t)),\\
\label{Fin6}
    \md{\frac{1}{3}\dot x_{2}(t)\int_{\mathbb{R}}(\omega_{1}(t))U^{(4)}\left(H^{x_{1}(t)}_{-1,0}+H^{x_{2}(t)}_{0,1}\right)\partial_{x}H^{x_{2}(t)}_{0,1}g(t)^{3}\,dt}
    =O(\delta(t)).
\end{align}
In conclusion, we obtain that the estimates \eqref{Fin1}, \eqref{Fin2} imply the following estimate 
\begin{multline}\label{F43}
    -2\dot x_{1}(t)\int_{\mathbb{R}}\omega_{1}(t,x) \ddot U\left(H^{x_{1}(t)}_{-1,0}(x)+H^{x_{2}(t)}_{0,1}(x)\right)g(t,x)\partial_{x}g(t,x)\,dx=
    \int_{\mathbb{R}}\dot x_{1}(t)\partial_{x}H^{x_{1}(t)}_{-1,0}(x)U^{(3)}\left(H^{x_{2}(t)}_{0,1}+H^{x_{1}(t)}_{-1,0}\right)g(t)^{2}\,dx\\+O(\delta(t)).
\end{multline}
Then, the estimates \eqref{ff}, \eqref{F4.7.7}, \eqref{Fin5}, \eqref{Fin6} and \eqref{F43} imply that
\begin{gather*}
    2\frac{d}{dt}\left(\int_{\mathbb{R}}\partial_{t}g(t,x)\partial_{x}g(t,x)\dot x_{1}(t)\omega_{1}(t,x)\,dx\right)=\frac{1}{3}\int_{\mathbb{R}}U^{(4)} \left(H^{x_{1}(t)}_{-1,0}+H^{x_{2}(t)}_{0,1}\right)\big(\dot x_{1}(t)\partial_{x}H^{x_{1}(t)}_{-1,0}\big)g(t)^{3}\,dx\\+\int_{\mathbb{R}}\big(\dot x_{1}(t)\partial_{x}H^{x_{1}(t)}_{-1,0})U^{(3)}(H^{x_{2}(t)}_{0,1}+H^{x_{1}(t)}_{-1,0})g(t)^{2}\,dx-2\dot x_{1}(t)^{3}\int_{\mathbb{R}}\partial^{2}_{x}H^{x_{1}(t)}_{-1,0}(x)\partial_{x}g(t,x)\,dx.+O(\delta(t)).
\end{gather*}
By an analogous argument, we deduce that
\begin{gather*}
    2\frac{d}{dt}\left(\int_{\mathbb{R}}\partial_{t}g(t,x)\partial_{x}g(t,x)\dot x_{2}(t)(1-\omega_{1}(t,x))\,dx\right)=\frac{1}{3}\int_{\mathbb{R}}U^{(4)}\left(H^{x_{1}(t)}_{-1,0}+H^{x_{2}(t)}_{0,1}\right)\big(\dot x_{2}(t)\partial_{x}H^{x_{2}(t)}_{0,1}\big)g(t)^{3}\,dx\\+\int_{\mathbb{R}}\big(\dot x_{2}(t)\partial_{x}H^{x_{2}(t)}_{0,1})U^{(3)}\left(H^{x_{2}(t)}_{0,1}+H^{x_{1}(t)}_{-1,0}\right)g(t)^{2}\,dx-2\dot x_{2}(t)^{3}\int_{\mathbb{R}}\partial^{2}_{x}H^{x_{2}(t)}_{0,1}(x)\partial_{x}g(t,x)\,dx+O(\delta(t)).
\end{gather*}
In conclusion, we have that
\begin{multline}\label{F44}
    \frac{dF_{4}(t)}{dt} =\int_{\mathbb{R}}\left[\dot x_{1}(t)\partial_{x}H^{x_{1}(t)}_{-1,0}+\dot x_{2}(t)\partial_{x}H^{x_{2}(t)}_{0,1}\right]U^{(3)}\left(H^{x_{2}(t)}_{0,1}+H^{x_{1}(t)}_{-1,0}\right)g(t)^{2}\,dx-2\dot x_{2}(t)^{3}\int_{\mathbb{R}}\partial^{2}_{x}H^{x_{2}(t)}_{0,1}(x)\partial_{x}g(t,x)\,dx
    \\+\frac{1}{3}\int_{\mathbb{R}}U^{(4)} \left(H^{x_{1}(t)}_{-1,0}+H^{x_{2}(t)}_{0,1}\right)\left[\dot x_{1}(t)\partial_{x}H^{x_{1}(t)}_{-1,0}+\dot x_{2}(t)\partial_{x}H^{x_{2}(t)}_{0,1}\right]g(t)^{3}\,dx
    -2\dot x_{1}(t)^{3}\int_{\mathbb{R}}\partial^{2}_{x}H^{x_{1}(t)}_{-1,0}(x)\partial_{x}g(t,x)\,dx+O(\delta(t)).
\end{multline}
\textbf{Step 6.}(The derivative of $F_{5}(t).$) We have that
\begin{equation}\label{F55}
    \frac{dF_{5}(t)}{dt}=\int_{\mathbb{R}}U^{(3)} \left(H^{x_{1}(t)}_{-1,0}+H^{x_{2}(t)}_{0,1} \right)g(t)^{2}\partial_{t}g(t)\,dx
    -\frac{1}{3}\int_{\mathbb{R}}U^{(4)} \left(H^{x_{1}(t)}_{-1,0}+H^{x_{2}(t)}_{0,1}\right)\big(\dot x_{1}(t)\partial_{x}H^{x_{1}(t)}_{-1,0}+\dot x_{2}(t)\partial_{x}H^{x_{2}(t)}_{0,1}\big 
    )g(t)^{3}\,dx.
\end{equation}
\textbf{Step 7.}(Conclusion of estimate of $|\dot F(t)|$)
From the identities \eqref{F55} and \eqref{F44}, we obtain that
\begin{multline}\label{sum2}
    \frac{dF_{4}(t)}{dt}+\frac{dF_{5}(t)}{dt}=\int_{\mathbb{R}}\big(\dot x_{1}(t)\partial_{x}H^{x_{1}(t)}_{-1,0}+\dot x_{2}(t)\partial_{x}H^{x_{2}(t)}_{0,1}\big)U^{(3)}\left(H^{x_{2}(t)}_{0,1}+H^{x_{1}(t)}_{-1,0}\right)g(t)^{2}\,dx
   -2\dot x_{1}(t)^{3}\int_{\mathbb{R}}\partial^{2}_{x}H^{x_{1}(t)}_{-1,0}(x)\partial_{x}g(t,x)\,dx\\
    -2\dot x_{2}(t)^{3}\int_{\mathbb{R}}\partial^{2}_{x}H^{x_{2}(t)}_{0,1}(x)\partial_{x}g(t,x)\,dx+\int_{\mathbb{R}}U^{(3)} \left(H^{x_{1}(t)}_{-1,0}+H^{x_{2}(t)}_{0,1} \right)g(t)^{2}\partial_{t}g(t)\,dx+O(\delta(t)).
\end{multline}
Then, the sum of identities \eqref{sum1} and \eqref{sum2} implies
$\sum_{i=1}^{5}\frac{dF_{i}(t)}{dt}=O(\delta(t)),$
this finishes the proof of inequality $\md{\dot F(t)}=O(\delta(t)).\\$
\textbf{Proof of $F(t)+A_{1}\epsilon^{2}\geq A_{2}\epsilon^{2}.$} The Coercitivity Lemma implies that $\exists\,c>0,$ such that
$F_{1}(t)\geq c\norm{\overrightarrow{g(t)}}^{2}.$
Also, from Theorem \ref{Stab}, we have the global estimate
\begin{equation}
    \max_{j\in\{1,2\}}\md{\dot x_{j}(t)}^{2}+\md{\ddot x_{j}(t)}+e^{-\sqrt{2}z(t)}+\norm{\overrightarrow{g(t)}}^{2}=O(\epsilon)
\end{equation}
that implies that $\md{F_{3}(t)}=O\Big(\norm{\overrightarrow{g(t)}}\epsilon\Big),\,\md{F_{4}(t)}=O\Big(\norm{\overrightarrow{g(t)}}^{2}\epsilon^{\frac{1}{2}}\Big),\,\md{F_{5}(t)}=O\Big(\norm{\overrightarrow{g(t)}}^{2}\epsilon^{\frac{1}{2}}\Big).$
Also, since
\begin{equation*}
    \md{ U\left(H^{x_{1}(t)}_{-1,0}(x)\right)+\dot U\left(H^{x_{2}(t)}_{0,1}(x)\right)-\dot U\left(H^{x_{2}(t)}_{0,1}(x)+H^{x_{1}(t)}_{-1,0}(x)\right)}=O\Big(\md{H^{x_{1}(t)}_{-1,0}(x)H^{x_{2}(t)}_{0,1}(x)\big[H^{x_{2}(t)}_{0,1}(x)+H^{x_{1}(t)}_{-1,0}(x)\big]}\Big),
\end{equation*}
the Lemma \ref{interact} and Cauchy-Schwarz Inequality imply that \begin{center}$|F_{2}(t)|=O\Big(\norm{\overrightarrow{g(t)}}e^{-\sqrt{2}z(t)}\Big).$\end{center}
Then, the conclusion of $F(t)+A_{1}\epsilon^{2}\geq A_{2}\norm{\overrightarrow{g(t)}}^{2}$ follows from Young Inequality for $\epsilon$ small enough.
\end{proof}
\begin{remark}\label{help1}
In the proof of Theorem \ref{EnergyE}, from Theorem \ref{Stab} we have $\md{F_{2}(t)}+\md{F_{3}(t)}=O\left(\norm{\overrightarrow{g(t)}}\epsilon\right).$ Since $\md{F_{4}(t)}+\md{F_{5}(t)}=O\left(\norm{\overrightarrow{g(t)}}^{2}\epsilon^{\frac{1}{2}}\right)$ and $\md{F_{1}(t)}\lesssim \norm{\overrightarrow{g(t)}}^{2},$ then Young Inequality implies that
\begin{equation*}
    \md{F(t)}\lesssim \norm{\overrightarrow{g(t)}}^{2}+\epsilon^{2}.
\end{equation*}
\end{remark}
\begin{remark}[General Energy Estimate]\label{generalE}
For any $0<\theta, \gamma<1,$ we can create a smooth cut function  $0\leq\chi(x)\leq 1$ such that
\begin{equation*}
    \chi(x)=\begin{cases}
        0,\text{ if $x\leq \theta(1-\gamma),$}\\
        1, \text {if $x\geq \theta.$}
    \end{cases}
\end{equation*}
We define
\begin{equation*}
    \chi_{0}(t,x)=\chi\left(\frac{x-x_{1}(t)}{x_{2}(t)-x_{1}(t)}\right).
\end{equation*}
If we consider the following functional
\begin{multline}
    L(t)=\big\langle D^{2}E(H^{x_{2}(t)}_{0,1}+H^{x_{1}(t)}_{-1,0})\overrightarrow{g(t)},\,\overrightarrow{g(t)}\big\rangle_{L^{2}\times L^{2}}+2\int_{\mathbb{R}}\partial_{t}g(t,x)\partial_{x}g(t,x)\Big[\dot x_{1}(t)\chi_{0}(t,x)+\dot x_{2}(t)(1-\chi_{0}(t,x))\Big]\,dx\\
    -2\int_{\mathbb{R}}g(t,x)\Big(\dot U(H^{x_{1}(t)}_{-1,0}(x))+\dot U(H^{x_{2}(t)}_{0,1}(x))-\dot U(H^{x_{2}(t)}_{0,1}(x)+H^{x_{1}(t)}_{-1,0}(x))\Big)\,dx\\
    +2\int_{\mathbb{R}}g(t,x)\Big[(\dot x_{1}(t))^{2}\partial^{2}_{x}H^{x_{1}(t)}_{-1,0}(x)+(\dot x_{2}(t))^{2}\partial^{2}_{x}H^{x_{2}(t)}_{0,1}(x)\Big]\,dx
    +\frac{1}{3}\int_{\mathbb{R}}U^{(3)}(H^{x_{2}(t)}_{0,1}(x)+H^{x_{1}(t)}_{-1,0}(x))g(t,x)^{3}\,dx,
\end{multline}
then, by a similar proof to the Theorem \ref{EnergyE}, we obtain that if $0<\epsilon\ll 1$ and
\begin{equation}\label{newE}
   \delta_{1}(t)=\delta(t)+\max_{j\in\{1,2\}}\md{\dot x_{j}(t)}^{3}\max(e^{-\sqrt{2}z(t)(1-\theta)},e^{-\sqrt{2}z(t)\theta(1-\gamma)})\norm{\overrightarrow{g(t)}}- \max_{j\in\{1,2\}}\md{\dot x_{j}(t)}^{3}e^{-\frac{\sqrt{2}}{5}z(t)}\norm{\overrightarrow{g(t)}},
\end{equation}
then there are positive constants $A_{1},\,A_{2}>0$ such that 
\begin{equation*}
    \md{\dot L(t)}=O(\delta_{1}(t)),\,L(t)+A_{1}\epsilon^{2}\geq A_{2}\epsilon^{2}.
\end{equation*}
\end{remark}
Our first application of Theorem \ref{EnergyE} is to estimate the size of the remainder $\norm{\overrightarrow{g(t)}}$ during a long time interval. More precisely, this corresponds to the following theorem, which is a weaker version of Theorem \ref{T1}.
\begin{theorem}\label{TT1}
There is $\delta>0,$ such that if $0<\epsilon<\delta$ enough, $(\phi(0),\partial_{t}\phi(0)) \in S\times L^{2}(\mathbb{R})$
and $E_{total}(\phi(0),\partial_{t}\phi(0))=2E_{pot}(H_{0,1})+\epsilon,$ then 
there are $x_{2},x_{1} \in C^{2}(\mathbb{R})$  functions such that the unique global time solution of \eqref{nlww} is given, for
\begin{equation}
    \phi(t)=H_{0,1}(x-x_{2}(t))+ H_{-1,0}(x-x_{1}(t))+g(t),
\end{equation}
with $g(t)$ satisfying orthogonality conditions of the Modulation Lemma and
\begin{equation}\label{evolution}
\norm{(g(t),\partial_{t}g(t))}_{H^{1}\times L^{2}}^{2}\leq C\left[\norm{(g(0,x),\partial_{t}g(0,x))}_{H^{1}\times L^{2}}^{2}+\epsilon^{2}\ln{\Big(\frac{1}{\epsilon}\Big)}^{2}\right]\exp\left(C\md{t}\left(\frac{\epsilon^{\frac{1}{2}}}{\ln{(\frac{1}{\epsilon})}}\right)\right).
\end{equation}
\end{theorem}
\begin{proof}[Proof of Theorem \ref{TT1}]
In notation of Theorem \ref{EnergyE}, from Theorem \ref{EnergyE} and Remark \ref{help1}, there are uniform positive constants $A_{2},\,A_{1}$ such that for all $t\geq 0$
\begin{equation}\label{CC1}
    A_{2}\norm{\overrightarrow{g(t)}}^{2}\leq F(t)+A_{1}\epsilon^{2}\leq C\Big(\norm{\overrightarrow{g(t)}}^{2}+\epsilon^{2}\Big).
\end{equation}
From now on, we denote
$G(t)\coloneqq F(t)+A_{1}\epsilon^{2}\ln{(\frac{1}{\epsilon})}^{2}.$ From the inequality \eqref{CC1} and Remark \ref{ttt1}, there is a constant $C>0$ such that, for all $t \geq 0,$ $G(t)$ satisfies
\begin{equation}
G(t)\leq G(0)+C\Big(\int_{0}^{t} G(s)\frac{\epsilon^{\frac{1}{2}}}{\ln{(\frac{1}{\epsilon})}},ds\Big).
\end{equation}
In conclusion, from the Fundamental Theorem of Calculus, we obtain that
$G(t)\leq G(0)\exp\Big(\frac{Ct\epsilon^{\frac{1}{2}}}{\ln{(\frac{1}{\epsilon})}}\Big).$
Then, from the definition of $G$ and inequality \eqref{CC1}, we verify the inequality \eqref{evolution}.
\end{proof}
\section{Global Dynamics of Modulation Parameters}
\begin{lemma}\label{aux22}
In notation of Theorem \ref{T1}, $\exists C>0,$ such that if the hypotheses of Theorem \ref{T1} are true,
then for $(g_{0}(x),g_{1}(x))=\left(g(0,x),\partial_{t}g(0,x)\right)$ we have that there are functions $p_{1}(t),\,p_{2}(t) \in C^{1}(R_{\geq 0}),$ such that for $j \in \{1,\,2\},$ we have:
\begin{equation}\label{modest1}
    \md{\dot x_{j}(t)-p_{j}(t)}\lesssim \Big(\norm{(g_{0},g_{1})}_{H^{1}\times L^{2}}+\epsilon\ln{\Big(\frac{1}{\epsilon}\Big)}\Big)\epsilon^{\frac{1}{2}}\exp\Big(\frac{2C\epsilon^{\frac{1}{2}}t}{\ln{(\frac{1}{\epsilon})}}\Big),
    \end{equation}
\begin{equation}\label{modest2}
    \md{\dot p_{j}(t)-(-1)^{j}8\sqrt{2}e^{-\sqrt{2}z(t)}}\lesssim\frac{\Big(\norm{(g_{0},g_{1})}_{H^{1}\times L^{2}}+\epsilon\ln{\Big(\frac{1}{\epsilon}\Big)}\Big)^{2}}{\ln{\ln{(\frac{1}{\epsilon})}}}\exp\Big(\frac{2Ct\epsilon^{\frac{1}{2}}}{\ln{(\frac{1}{\epsilon})}}\Big).
\end{equation}
\end{lemma}
\begin{proof}
In the notation of Lemma \ref{modueq}, we consider the functions $p_{j}(t)$ for $j \in \{1,\,2\}$ and we consider $\theta=\frac{1-\gamma}{2-\gamma},$ the value of $\gamma$ will be chosen later. From Lemma \ref{modueq}, we have
that
\begin{equation*}
    \md{\dot x_{j}(t)-p_{j}(t)}\lesssim \Big[1+\frac{1}{\gamma z(t)}\Big]\Big(\max_{j \in \{1,2\}}\md{\dot x_{j}(t)}\norm{\overrightarrow{g(t)}}+\norm{\overrightarrow{g(t)}}^{2}\Big)+\max_{j \in \{1,2\}}\md{\dot x_{j}(t)}z(t)e^{-\sqrt{2}z(t)}.
\end{equation*}
We recall from Theorem \ref{Stab} the  estimates $\max_{j\in\{1,\,2\}}\md{\dot x_{j}(t)}=O(\epsilon^{\frac{1}{2}}),\,e^{-\sqrt{2}z(t)}=O(\epsilon).$
From Theorem \ref{TT1}, we have that
\begin{equation*}
    \norm{\overrightarrow{g(t)}}\lesssim \left(\norm{\overrightarrow{g(0)}}+\epsilon\ln{\Big(\frac{1}{\epsilon}\Big)}\right) \exp\Big(\frac{C\epsilon^{\frac{1}{2}}t}{\ln{(\frac{1}{\epsilon})}}\Big).
\end{equation*}
To simplify our computations we denote $c_{0}=\frac{\norm{\overrightarrow{g(0)}}+\epsilon\ln(\frac{1}{\epsilon})}{\epsilon\ln(\frac{1}{\epsilon})}.$
Then, we obtain for $j\in\{1,\,2\}$ that
\begin{equation}\label{modest11}
    \md{\dot x_{j}(t)-p_{j}(t)}\lesssim \left[1+\frac{1}{\gamma \ln{(\frac{1}{\epsilon})}}\right]\Big(c_{0}\epsilon^{\frac{3}{2}}\ln{\Big(\frac{1}{\epsilon}\Big)}\exp\Big(\frac{C\epsilon^{\frac{1}{2}}t}{\ln{(\frac{1}{\epsilon})}}\Big)+c_{0}^{2}\epsilon^{2}\ln{\Big(\frac{1}{\epsilon}\Big)}^{2}\exp\Big(\frac{2C\epsilon^{\frac{1}{2}}t}{\ln{\frac{1}{\epsilon}}}\Big)\Big).
\end{equation}
Since $e^{-\sqrt{2}z(t)}\lesssim \epsilon,$ we deduce for $\epsilon\ll 1$ that $z(t)e^{-\sqrt{2}z(t)}\lesssim \epsilon\ln{(\frac{1}{\epsilon})}<\epsilon^{1-\frac{\gamma}{(2-\gamma)2}}\ln{(\frac{1}{\epsilon})}.$  Then, we obtain from the same estimates and the definition \eqref{alpha} of $\alpha(t),$ that
\begin{multline}\label{alpha11}
    \alpha(t)\lesssim c_{0}^{2}\Big(\epsilon \ln{\Big(\frac{1}{\epsilon}\Big)}\Big)^{2}\left[\max_{k\in\{1,\,2\}}\Big(\frac{1}{\gamma z(t)}\Big)^{k}+\epsilon^{\frac{1-\gamma}{2-\gamma}}\right]\exp\Big(2\frac{C\epsilon^{\frac{1}{2}}t}{\ln{(\frac{1}{\epsilon})}}\Big)
    \\+ c_{0}\epsilon^{2-\frac{\gamma}{(2-\gamma)2}}\ln{\Big(\frac{1}{\epsilon}\Big)}\exp\Big(\frac{C\epsilon^{\frac{1}{2}}t}{\ln{(\frac{1}{\epsilon})}}\Big)\left[1+\frac{1}{\gamma z(t)}+\frac{\epsilon^{\frac{1}{2}}}{(\gamma z(t))^{2}}\right]+\frac{\epsilon^{1+\frac{2(1-\gamma)}{2-\gamma}}}{z(t)\gamma}.
\end{multline}
However, if 
$\gamma \ln{(\frac{1}{\epsilon})}\leq 1$ and $z(0)\cong \ln{(\frac{1}{\epsilon})},$ which is possible, then the right-hand side of inequality \eqref{alpha11} is greater than or equivalent to
$\epsilon^{2}\ln{(\frac{1}{\epsilon})}^{2}$while $t\lesssim \frac{\ln{(\frac{1}{\epsilon})}}{\epsilon^{\frac{1}{2}}}.$
But, it is not difficult to verify that for 
$\gamma=\frac{\ln{\ln{(\frac{1}{\epsilon})}}}{\ln{(\frac{1}{\epsilon})}},$
the right-hand side of inequality \eqref{alpha11} is smaller than $\epsilon^{2}\ln{(\frac{1}{\epsilon})}^{2}$. In conclusion, from now on, we are going to study the right-hand side of \eqref{alpha11} for
$\frac{1}{\ln(\frac{1}{\epsilon})}<\gamma<1$. Since we know that $\ln{(\frac{1}{\epsilon})}\lesssim z(t)$ from Theorem \ref{Stab}, the inequality \eqref{alpha11} implies for $\frac{1}{\ln{(\frac{1}{\epsilon})}}<\gamma<1$ that
\begin{multline}\label{betha}
   \alpha(t)\lesssim \beta(t)\coloneqq\Big(c_{0}\epsilon \ln{\Big(\frac{1}{\epsilon}\Big)}\Big)^{2}\left[\frac{1}{\gamma \ln{(\frac{1}{\epsilon}})}+\epsilon^{\frac{1-\gamma}{2-\gamma}}\right]\exp\Big(2\frac{C\epsilon^{\frac{1}{2}}t}{\ln{(\frac{1}{\epsilon})}}\Big)
    + c_{0}\epsilon^{2-\frac{\gamma}{2(2-\gamma)}}\ln{\Big(\frac{1}{\epsilon}\Big)}\exp\Big(\frac{C\epsilon^{\frac{1}{2}}t}{\ln{(\frac{1}{\epsilon})}}\Big)+\frac{\epsilon^{1+\frac{2(1-\gamma)}{2-\gamma}}}{\gamma\ln{(\frac{1}{\epsilon})}}\\=\beta_{1}(t)+\beta_{2}(t)+\beta_{3}(t), \text{ respectively.}
\end{multline}
For $\epsilon>0$ small enough, it is not difficult to verify that if $
\beta_{3}(t)\geq \beta_{1}(t),$ then $\gamma\geq \frac{\ln{\ln{(\frac{1}{\epsilon})}}}{\ln{(\frac{1}{\epsilon})}}.$ Moreover, if we have that
$1>\gamma>8\frac{\ln{\ln{(\frac{1}{\epsilon})}}}{\ln{(\frac{1}{\epsilon})}},$ we obtain from the following estimate
\begin{equation*}
    \beta_{3}(t)=\frac{\epsilon^{2}\epsilon^{\frac{-\gamma}{2-\gamma}}}{\gamma \ln{(\frac{1}{\epsilon})}}
    > \frac{\epsilon^{2}}{\ln{(\frac{1}{\epsilon})}}\exp\left(\frac{8\ln{\ln\left(\frac{1}{\epsilon}\right)}}{2-\gamma}\right)= \frac{\epsilon^{2}}{\ln{(\frac{1}{\epsilon})}}\ln{\left(\frac{1}{\epsilon}\right)}^{\frac{8}{2-\gamma}},
\end{equation*}
that $\beta_{3}(t)>\frac{\epsilon^{2}\ln{(\frac{1}{\epsilon})}^{2}}{\ln{\ln{(\frac{1}{\epsilon})}}}.$ If $\gamma\leq \frac{\ln{\ln{(\frac{1}{\epsilon})}}}{\ln{(\frac{1}{\epsilon})}},$ then $\frac{\epsilon^{2}\ln{(\frac{1}{\epsilon})}^{2}}{\ln{\ln{(\frac{1}{\epsilon})}}}\lesssim\beta_{1}(t).$ In conclusion, for any case we have that $\frac{\epsilon^{2}\ln{(\frac{1}{\epsilon})}^{2}}{\ln{\ln{(\frac{1}{\epsilon})}}}\lesssim\beta(t),$
so we choose
$\gamma=\frac{\ln{\ln{(\frac{1}{\epsilon})}}}{\ln{(\frac{1}{\epsilon})}}.$ As a consequence, $\alpha(t)$ is less than or equivalent to
\begin{equation}\label{alpha22}
     c_{0}^{2}\epsilon^{2}\frac{\ln{(\frac{1}{\epsilon}})^{2}}{\ln{\ln{(\frac{1}{\epsilon})}}}\exp\Big(\frac{2C\epsilon^{\frac{1}{2}}t}{\ln{(\frac{1}{\epsilon})}}\Big).
\end{equation}
So, the estimates \eqref{modest11}, \eqref{alpha22}, Remark \ref{odez} and our choice of $\gamma$ imply the inequalities \eqref{modest1} and \eqref{modest2}.
\end{proof}
\begin{remark}\label{worst}
If $\frac{\epsilon^{\frac{1}{2}}}{\ln{(\frac{1}{\epsilon})}^{m}}\lesssim\norm{\overrightarrow{g(0)}}$ for a constant $m>0,$ then, for $\gamma=\frac{1}{8},$ we have from  Lemma \ref{modueq} that there is $p(t) \in C^{2}(\mathbb{R})$ such that for all $t\geq 0$
\begin{align}\label{moddd1}
    \md{\dot z(t)-p(t)}\lesssim \epsilon^{\frac{1}{2}}\norm{\overrightarrow{g(0)}},\\ \label{moddd2}
    \md{\dot p(t) -16\sqrt{2} e^{-\sqrt{2}z(t)}}\lesssim \frac{\norm{\overrightarrow{g(0)}}^{2}}{z(t)}.
\end{align}
Then, for the smooth real function $d(t)$ satisfying 
\begin{equation*}
    \ddot d(t)=16\sqrt{2}e^{-\sqrt{2}d(t)},\,(d(0),\dot d(0))=(z(0),\dot z(0)),
\end{equation*}
and since $e^{-\sqrt{2}z(t)}\lesssim\epsilon,\,\ln{(\frac{1}{\epsilon})}\lesssim z(t),$ we can deduce that $Y(t)=(z(t)-d(t))$ satisfies the following integral inequality for a constant $K>0$
\begin{equation}\label{grol}
    \md{Y(t)}\leq K\left (\epsilon^{\frac{1}{2}}\norm{\overrightarrow{g(0)}}t+\frac{\norm{\overrightarrow{g(0)}}^{2}}{\ln{(\frac{1}{\epsilon})}}t^{2}+\int_{0}^{t}\int_{0}^{s}\epsilon \md{Y(s_{1})}\,ds_{1}\,ds\right),\\
    Y(0)=0,\,\dot Y(0)=0.
\end{equation}
In conclusion, from the Gronwall Lemma, we obtain that
$\md{Y(t)}\lesssim Q(t K^{\frac{1}{2}}),$ where $Q(t)$ is the solution of the following integral equation
\begin{equation*}
    Q(t)=\epsilon^{\frac{1}{2}}\norm{\overrightarrow{g(0)}}t+\frac{\norm{\overrightarrow{g(0)}}^{2}}{\ln{(\frac{1}{\epsilon})}}t^{2}+\int_{0}^{t}\int_{0}^{s}\epsilon Q(s_{1})\,ds_{1}\,ds.
\end{equation*}
By standard ordinary differential equation techniques,
we deduce that
\begin{equation}\label{cdd1}
    \md{z(t)-d(t)}\lesssim Q(t K^{\frac{1}{2}})=\left(\frac{\norm{\overrightarrow{g(0)}}}{2}+\frac{\norm{\overrightarrow{g(0)}}^{2}}{\epsilon\ln{(\frac{1}{\epsilon})}}\right)e^{\epsilon^{\frac{1}{2}}t K^{\frac{1}{2}}}+\left(\frac{-\norm{\overrightarrow{g(0)}}}{2}+\frac{\norm{\overrightarrow{g(0)}}^{2}}{\epsilon\ln{(\frac{1}{\epsilon})}}\right)e^{-\epsilon^{\frac{1}{2}}t K^{\frac{1}{2}}}-2\frac{\norm{\overrightarrow{g(0)}}^{2}}{\epsilon\ln{(\frac{1}{\epsilon})}},
\end{equation}
and from $\dot z(0)=\dot d(0)$ and the estimates \eqref{moddd1} and \eqref{moddd2}, we obtain that
\begin{equation}\label{fim1}
    \md{\dot z(t)-\dot d(t)}\lesssim \md{p(0)-\dot z(0)}+\int_{0}^{t}\epsilon \md{z(s)-d(s)}\,ds,
\end{equation}
from which with \eqref{cdd1}, we obtain that
\begin{equation}\label{fim2}
    \md{\dot z(t)-\dot d(t)}\lesssim e^{\epsilon^{\frac{1}{2}}\md{t}K^{\frac{1}{2}}}\epsilon^{\frac{1}{2}}\left(\norm{\overrightarrow{g(0)}}+\frac{\norm{\overrightarrow{g(0)}}^{2}}{\epsilon\ln{(\frac{1}{\epsilon})}}\right).
\end{equation}
However, the precision of the estimates \eqref{cdd1} and \eqref{fim2} is very bad when $\epsilon^{-\frac{1}{2}}\ll t,$ which motivate us to apply Lemma \eqref{modueq} to estimate the modulations parameters $x_{1}(t),\,x_{2}(t)$ for $t\lesssim \frac{\ln{(\frac{1}{\epsilon})}}{\epsilon^{\frac{1}{2}}}.$    
\end{remark}
We recall from Theorem \ref{trueTheo2} the definitions of the functions $d_{1}(t),\,d_{2}(t).$ If $\norm{\overrightarrow{g(0)}}\geq \frac{\epsilon^{\frac{1}{2}}}{\ln{(\frac{1}{\epsilon})}^{5}},$ because Theorem \ref{Stab} and $\max_{j\in\{1,2\}}\md{\dot d_{j}(0)-\dot x_{j}(0)}=0$ imply that $\max_{j\in\{1,\,2\}}\md{d_{j}(t)-x_{j}(t)}=O(\min(\epsilon t,\epsilon^{\frac{1}{2}}t)),\,\max_{j\in\{1,\,2\}}\md{\dot d_{j}(t)-\dot x_{j}(t)}=O(\epsilon t),$ we deduce for a constant $C>0$ large enough the estimates \eqref{oded1} and \eqref{oded2} of Theorem \ref{trueTheo2}. For the case $\norm{\overrightarrow{g(0)}}\leq \frac{\epsilon^{\frac{1}{2}}}{\ln{(\frac{1}{\epsilon})}^{5}}$, the estimates of $\max_{j\in\{1,2\}}\md{x_{j}(t)-d_{j}(t)},\,\max_{j\in\{1,\,2\}}\md{\dot x_{j}(t)-\dot d_{j}(t)}$ will be done by studying separated cases depending on the initial data $z(0),\,\dot z(0).$ 
\begin{lemma}\label{t21} $\exists K>0$ such that
if $\norm{\overrightarrow{g(0)}}\leq \frac{\epsilon^{\frac{1}{2}}}{\ln{(\frac{1}{\epsilon})}^{5}},\,(g_{0}(x),g_{1}(x))=\left(g(0,x),\partial_{t}g(0,x)\right)$ and all the hypotheses of Theorem \ref{trueTheo2} are true and $\frac{\epsilon}{\ln{(\frac{1}{\epsilon})}^{8}}\lesssim e^{-\sqrt{2}z(0)}\lesssim \epsilon,$ then we have for $0\leq t$ that
\begin{gather}\label{theo2d1}
    \max_{j \in \{1,\,2\}}\md{x_{j}(t)-d_{j}(t)}=O\left(\frac{\max\Big(\norm{(g_{0},g_{1})},\epsilon\ln{(\frac{1}{\epsilon})}\Big)^{2}\ln{(\frac{1}{\epsilon})}^{6}}{\epsilon\ln{\ln(\frac{1}{\epsilon})}}\exp\Big(\frac{K\epsilon^{\frac{1}{2}}t}{\ln{(\frac{1}{\epsilon})}}\Big)\right),\\ \label{theo2d2}
    \max_{j \in \{1,\,2\}}\md{\dot x_{j}(t)-\dot d_{j}(t)}=O\left(\max\Big(\norm{(g_{0},g_{1})},\epsilon\ln{\Big(\frac{1}{\epsilon}\Big)}\Big)^{2}\frac{\ln{(\frac{1}{\epsilon})}^{6}}{\epsilon^{\frac{1}{2}}\ln{\ln{(\frac{1}{\epsilon})}}}\exp\Big(\frac{K\epsilon^{\frac{1}{2}}t}{\ln{(\frac{1}{\epsilon})}}\Big)\right).
\end{gather}
\end{lemma}
\begin{proof}[Proof of Lemma \ref{t21}]
First, in notation of Lemma \ref{aux22}, we define
\begin{equation*}
    p(t)\coloneqq p_{2}(t)-p_{1}(t),\,
    z(t)\coloneqq x_{2}(t)-x_{1}(t),\,
    \dot z(t)\coloneqq \dot x_{2}(t)-\dot x_{1}(t).
\end{equation*}
Also, motivated by Remark \ref{odez}, we consider the smooth function $d(t)$  
solution of the following ordinary differential equation
\[\begin{cases*}
\ddot d(t)=16\sqrt{2}e^{-\sqrt{2}d(t)},\\
(d(0),\dot d(0))=(z(0),\dot z(0)).
\end{cases*}\]
\textbf{Step 1.}(Estimate of $z(t),\,\dot z(t)$)
From now on, we denote the functions $W(t)=z(t)-d(t),\,V(t)=p(t)-\dot d(t).$
Then, Lemma \ref{aux22} implies that $W,\,V$ satisfy the following ordinary differential estimates
\begin{align*}
    \md{\dot W(t)-V(t)}=O\left(\max\Big(\norm{(g_{0},g_{1})},\epsilon\ln{\Big(\frac{1}{\epsilon}\Big)}\Big)\epsilon^{\frac{1}{2}}\exp\Big(\frac{2C\epsilon^{\frac{1}{2}}t}{\ln{(\frac{1}{\epsilon})}}\Big)\right),\\
   \md{\dot V(t)+16\sqrt{2}e^{-\sqrt{2}d(t)}-16\sqrt{2}e^{-\sqrt{2}z(t)}}=O\left(\frac{\max\Big(\norm{(g_{0},g_{1})},\epsilon\ln{\Big(\frac{1}{\epsilon}\Big)}\Big)^{2}}{\ln{\ln{(\frac{1}{\epsilon})}}}\exp\Big(\frac{2C\epsilon^{\frac{1}{2}}t}{\ln{(\frac{1}{\epsilon})}}\Big)\right).
\end{align*}
From the above estimates and the Taylor's Expansion Theorem, we deduce the following almost ordinary differential system of equations, while $\md{W(t)}<1:$
\[\begin{cases}\label{almostode}
\dot W(t)=V(t)+O\Big(\max\Big(\norm{(g_{0},g_{1})},\epsilon\ln{\Big(\frac{1}{\epsilon}\Big)}\Big)\epsilon^{\frac{1}{2}}\exp\Big(\frac{2C\epsilon^{\frac{1}{2}}t}{\ln{(\frac{1}{\epsilon})}}\Big)\Big),\\
\dot V(t)=-32e^{-\sqrt{2}d(t)}W(t)+O\left(e^{-\sqrt{2}d(t)}W(t)^{2}\right)+O\left(\frac{\max\Big(\norm{(g_{0},g_{1})},\epsilon\ln{(\frac{1}{\epsilon})}\Big)^{2}}{\ln{\ln{(\frac{1}{\epsilon})}}}\exp\Big(\frac{2C\epsilon^{\frac{1}{2}}t}{\ln{(\frac{1}{\epsilon})}}\Big)\right).
\end{cases}\]
Recalling Remark \ref{odez},
we have that 
\begin{equation}\label{d(t)}
d(t)=\frac{1}{\sqrt{2}}\ln{\Big(\frac{8}{v^{2}}\cosh{(\sqrt{2}vt+c)}^{2}\Big)},
\end{equation}
where $v>0$ and $c\in\mathbb{R}$ are chosen such that $(d(0),\dot d(0))=(z(0),\dot z(0)).$
Moreover, it is not difficult to verify that
\begin{align*}
    v=\Big(\frac{\dot z(0)^{2}}{4}+8 e^{-\sqrt{2}z(0)}\Big)^{\frac{1}{2}},\,
    c=\arctanh{\left(\frac{\dot z(0)}{\big[32 e^{-\sqrt{2}z(0)}+\dot z(0)^{2}\big]^{\frac{1}{2}}}\right)}.
\end{align*}
Moreover, since
$
8e^{-\sqrt{2}z(0)}=v^{2}\sech{(c)}^{2}\leq 4v^{2}e^{-2\md{c}},    
$ 
we obtain
from the hypothesis for $e^{-\sqrt{2}z(0)}$ that
$\frac{\epsilon^{\frac{1}{2}}}{\ln{(\frac{1}{\epsilon})}^{4}}\lesssim v \lesssim \epsilon^{\frac{1}{2}}$ and as a consequence the estimate $|c|\lesssim \ln{(\ln{(\frac{1}{\epsilon})})}.$
Also, it is not difficult to verify that the functions
\begin{equation*}
n(t)=(\sqrt{2}vt+c)\tanh{(\sqrt{2}vt+c)}-1,\,m(t)=\tanh{(\sqrt{2}vt+c)}
\end{equation*}
generate all solutions of the following ordinary differential equation
\begin{equation}\label{almostode2}
\ddot y(t)=-32e^{-\sqrt{2}d(t)}y(t),
\end{equation}
which is obtained from the linear part of the system \eqref{almostode}.
\par To simplify our computations we use the following notation
\begin{align*}
    error_{1}(t)=\max\Big(\norm{(g_{0},g_{1})},\epsilon\ln{\Big(\frac{1}{\epsilon}\Big)}\Big)\epsilon^{\frac{1}{2}}\exp\Big(\frac{2C\epsilon^{\frac{1}{2}}t}{\ln{(\frac{1}{\epsilon})}}\Big),\\
    error_{2}(t)=e^{-\sqrt{2}d(t)}(z(t)-d(t))^{2}+\frac{\max\Big(\norm{(g_{0},g_{1})},\epsilon\ln{\Big(\frac{1}{\epsilon}\Big)}\Big)^{2}}{\ln{\ln{(\frac{1}{\epsilon})}}}\exp\Big(\frac{2C\epsilon^{\frac{1}{2}}t}{\ln{(\frac{1}{\epsilon})}}\Big).
\end{align*}
From the variation of parameters technique for ordinary differential equation, we can write that
\begin{equation}\label{solutionode}
    \begin{bmatrix}
    W(t)\\
    V(t)
    \end{bmatrix}=
    c_{1}(t)\begin{bmatrix}
    m(t)\\
    \dot m(t)
    \end{bmatrix}
    +c_{2}(t)
    \begin{bmatrix}
    n(t)\\
    \dot n(t)
    \end{bmatrix},
\end{equation}
such that
\[\begin{cases}
    \begin{bmatrix}
    m(t) & n(t)\\
    \dot m(t) & \dot n(t)
    \end{bmatrix}
    \begin{bmatrix}
        \dot c_{1}(t)\\
        \dot c_{2}(t)
    \end{bmatrix}
    =\begin{bmatrix}
        O(error_{1}(t))\\
        O(error_{2}(t))
    \end{bmatrix},\\
    \\
     \begin{bmatrix}
    m(0) & n(0)\\
    \dot m(0) & \dot n(0)
    \end{bmatrix}
    
    \begin{bmatrix}
         c_{1}(0)\\
         c_{2}(0)
    \end{bmatrix}
    =\begin{bmatrix}
        0\\
        O\Big(\Big[\norm{(g_{0},g_{1})}+\epsilon\ln{(\frac{1}{\epsilon})}\Big]\epsilon^{\frac{1}{2}}\Big)
    \end{bmatrix}.
\end{cases}\]
The presence of an error in the condition of the initial data $c_{1}(0),\,c_{2}(0)$ comes from estimate \eqref{modest1} of Lemma \ref{aux22}.
Since for all $t \in \mathbb{R}$ $m(t)\dot n(t)-\dot m(t)n(t)=\sqrt{2}v$, we can verify by Cramer's rule and from the fact that $\frac{\epsilon^{\frac{1}{2}}}{\ln{(\frac{1}{\epsilon})}^{4}}\lesssim v$ that 
\begin{align}\label{c100}
    c_{1}(0)=O\left(\max\Big(\norm{\overrightarrow{g(0)}},\epsilon\ln{\Big(\frac{1}{\epsilon}\Big)}\Big)\md{c\tanh{(c)}-1}\ln{\Big(\frac{1}{\epsilon}\Big)}^{4}\right),\\ \label{c200}
    c_{2}(0)=O\left(\max\Big(\norm{\overrightarrow{g(0)}},\epsilon\ln{\Big(\frac{1}{\epsilon}\Big)}\Big)\md{\tanh{(c)}}\ln{\Big(\frac{1}{\epsilon}\Big)}^{4}\right),
\end{align}
and
\begin{multline}\label{dc1}
    \md{\dot c_{1}(t)}=O\left(\md{m(t)+(\sqrt{2}vt+c)\sech{(\sqrt{2}vt+c)}^{2}}\max\left(\norm{(g_{0},g_{1})},\epsilon\ln{\Big(\frac{1}{\epsilon}\Big)}\right)\epsilon^{\frac{1}{2}}\exp\Big(\frac{2C\epsilon^{\frac{1}{2}}t}{\ln{(\frac{1}{\epsilon})}}\Big)\right)\\+O\left(\left[v\sech{(\sqrt{2}vt+c)}^{2}\md{W(t)}^{2}+\frac{\max\Big(\norm{(g_{0},g_{1})},\epsilon\ln{(\frac{1}{\epsilon})}\Big)^{2}}{v\ln{\ln{(\frac{1}{\epsilon})}}}\exp\Big(\frac{2C\epsilon^{\frac{1}{2}}t}{\ln{(\frac{1}{\epsilon})}}\Big)\right]\md{n(t)}\right),
\end{multline}
\begin{multline}\label{dc2}
    \md{\dot c_{2}(t)}=O\left(\left[v\sech{(\sqrt{2}vt+c)}^{2}\md{W(t)}^{2}+\frac{\max\Big(\norm{(g_{0},g_{1})},\epsilon\ln{\Big(\frac{1}{\epsilon}\Big)}\Big)^{2}}{v\ln{\ln{(\frac{1}{\epsilon})}}}\exp\Big(\frac{2C\epsilon^{\frac{1}{2}}t}{\ln{(\frac{1}{\epsilon})}}\Big)\right]\md{m(t)}\right)\\
    +O\left(\max\left(\norm{(g_{0},g_{1})},\epsilon\ln{\Big(\frac{1}{\epsilon}\Big)}\right)\exp\Big(\frac{2C\epsilon^{\frac{1}{2}}t}{\ln{(\frac{1}{\epsilon})}}\Big)\epsilon^{\frac{1}{2}}\sech(\sqrt{2}vt+c)^{2}\right).
\end{multline}
Since we have for all $x\geq 0$ that
\begin{equation*}
    \frac{d}{dx}\left(-\frac{\sech{(x)}^{2}x}{2}+\frac{3\tanh{(x)}}{2}\right)=\frac{\sech{(x)}^{2}}{2}+x\tanh{(x)}\sech{(x)}^{2}\geq \frac{\md{x\tanh{(x)}-1}\sech{(x)}^{2}}{2},
\end{equation*}
we deduce from the Fundamental Theorem of Calculus, the identity $n(t)=(\sqrt{2}vt+c)\tanh(\sqrt{2}vt+c)-1,$ the inequality $\frac{\epsilon^{\frac{1}{2}}}{\ln{(\frac{1}{\epsilon})}^{4}}\lesssim v$ and the estimates \eqref{dc1}, \eqref{dc2} that
\begin{multline}\label{c1estimate}
    |c_{1}(t)-c_{1}(0)|=O\left(\max\left(\norm{(g_{0},g_{1})},\epsilon\ln{\Big(\frac{1}{\epsilon}\Big)}\right)\ln{\Big(\frac{1}{\epsilon}\Big)}\left[\exp\Big(\frac{2Ct\epsilon^{\frac{1}{2}}}{\ln{(\frac{1}{\epsilon})}}\Big)-1\right]\right)\\+O\left(\left[\exp\Big(\frac{2C\epsilon^{\frac{1}{2}}t}{\ln{(\frac{1}{\epsilon})}}\Big)-1\right]\norm{n(s)}_{L^{\infty}_{s}[0,t]}\max\Big(\norm{(g_{0},g_{1})},\epsilon\ln{\Big(\frac{1}{\epsilon}\Big)}\Big)^{2}\frac{\ln{(\frac{1}{\epsilon})}^{5}}{\epsilon\ln{\ln(\frac{1}{\epsilon})}}\right)\\+O\left(\md{-\frac{\sech{(x)}^{2}x}{2}+\frac{3\tanh{(x)}}{2}\Big\vert_{c}^{\sqrt{2}vt+c}}\norm{W(s)}_{L^{\infty}_{s}[0,t]}^{2}\right).
\end{multline}
From a similar argument, we deduce that
\begin{multline}\label{c2estimate}
    \md{c_{2}(t)-c_{2}(0)}=O\left(\norm{W(s)}_{L^{\infty}_{s}[0,t]}^{2}\Big[\tanh{(\sqrt{2}vt+c)}-\tanh{(c)}\Big]\right)\\+O\left(\max\Big(\norm{(g_{0},g_{1})},\epsilon\ln{\Big(\frac{1}{\epsilon}\Big)}\Big)^{2}\left[\exp\left(\frac{2C\epsilon^{\frac{1}{2}}t}{\ln{(\frac{1}{\epsilon})}}\right)-1\right]\frac{\ln{(\frac{1}{\epsilon})}^{5}}{\epsilon\ln{\ln{(\frac{1}{\epsilon})}}}+\max\Big(\norm{(g_{0},g_{1})},\epsilon\ln{\Big(\frac{1}{\epsilon}\Big)}\Big)\ln{\Big(\frac{1}{\epsilon}\Big)}\exp\left(2Ct\frac{\epsilon^{\frac{1}{2}}}{\ln{(\frac{1}{\epsilon})}}\right)\right).
\end{multline}
From the estimates $v\lesssim \epsilon^{\frac{1}{2}},\,|c|\lesssim \ln{\ln{(\frac{1}{\epsilon})}},$ we obtain for $\epsilon\ll 1$ 
while
\begin{equation}\label{hold}
    \norm{W(s)}_{L^{\infty}_{s}[0,t]}\Big[\epsilon^{\frac{1}{2}}t+\ln{\ln{\Big(\frac{1}{\epsilon}\Big)}}\Big]\ln{\ln{\Big(\frac{1}{\epsilon}\Big)}}\leq 1,
\end{equation}
that
\begin{equation}\label{boot}
     \norm{W(s)}_{L^{\infty}_{s}[0,t]}^{2}(1+\md{n(t)})\lesssim  \norm{W(s)}_{L^{\infty}_{s}[0,t]}\frac{1}{\ln{\ln{(\frac{1}{\epsilon})}}}.
\end{equation}
Also, from
\begin{equation*}
    \md{n(t)}\leq (\sqrt{2}v|t|+|c|),
\end{equation*}
we deduce for $t\geq 0$ that
\begin{equation}\label{n(t)ineq}
    \md{n(t)}\lesssim \epsilon^{\frac{1}{2}}t+\ln{\ln{\left(\frac{1}{\epsilon}\right)}}\lesssim \ln{\left(\frac{1}{\epsilon}\right)}\exp\left(\frac{\epsilon^{\frac{1}{2}}t}{\ln{(\frac{1}{\epsilon})}}\right)
\end{equation}
In conclusion, the estimates \eqref{c1estimate}, \eqref{c2estimate}, \eqref{boot}, \eqref{n(t)ineq} and the definition of $W(t)=z(t)-d(t)$ implies that while the condition \eqref{hold} is true, then
\begin{equation}\label{y1t}
    \md{W(t)}\lesssim f(t)=\frac{\max\Big(\norm{(g_{0},g_{1})},\epsilon\ln{(\frac{1}{\epsilon})}\Big)^{2}\ln{(\frac{1}{\epsilon})}^{6}}{\epsilon\ln{\ln(\frac{1}{\epsilon})}}\exp\left(\frac{(2C+1)\epsilon^{\frac{1}{2}}t}{\ln{(\frac{1}{\epsilon})}}\right)
\end{equation}
Then, from the expression for $V(t)$ in the equation \eqref{solutionode} and the estimates \eqref{c1estimate}, \eqref{c2estimate}, \eqref{n(t)ineq}, we obtain that if inequality \eqref{y1t} is true, then
\begin{multline}\label{y2t}
\md{V(t)}\lesssim  \max\Big(\norm{(g_{0},g_{1})},\epsilon\ln{\Big(\frac{1}{\epsilon}\Big)}\Big)^{2}\frac{\ln{(\frac{1}{\epsilon})}^{6}}{\epsilon^{\frac{1}{2}}\ln{\ln{(\frac{1}{\epsilon})}}}\exp\left(\frac{(4C+3)\epsilon^{\frac{1}{2}}t}{\ln{(\frac{1}{\epsilon})}}\right)
    \\+\max\Big(\norm{(g_{0},g_{1})},\epsilon\ln{\Big(\frac{1}{\epsilon}\Big)}\Big)^{4}\frac{\ln{(\frac{1}{\epsilon})}^{12}}{\epsilon^{\frac{3}{2}}\big[\ln{\ln{(\frac{1}{\epsilon})}}\big]^{2}}\exp\Big(\frac{(4C+3)\epsilon^{\frac{1}{2}}t}{\ln{(\frac{1}{\epsilon})}}\Big),
\end{multline}
which implies the following estimate
\begin{equation}\label{doty1t}
    \md{\dot W(t)}\lesssim \max\Big(\norm{\overrightarrow{g(0)}},\epsilon\ln{\Big(\frac{1}{\epsilon}\Big)}\Big)^{2}\frac{\ln{(\frac{1}{\epsilon})}^{6}}{\epsilon^{\frac{1}{2}}\ln{\ln{(\frac{1}{\epsilon})}}}\exp\Big(\frac{(4C+3)\epsilon^{\frac{1}{2}}t}{\ln{(\frac{1}{\epsilon})}}\Big).
\end{equation}
Indeed, from the bound $\norm{\overrightarrow{g(0)}}\lesssim \frac{\epsilon^{\frac{1}{2}}}{\ln{(\frac{1}{\epsilon})}^{4}},$ we deduce that \eqref{hold} is true for $0\leq t\leq \frac{\ln{\ln{(\frac{1}{\epsilon})}}\ln{(\frac{1}{\epsilon})}}{(4C+2)\epsilon^{\frac{1}{2}}}.$ As a consequence, the estimates \eqref{y1t} and \eqref{doty1t} are true for $0\leq t\leq \frac{\ln{\ln{(\frac{1}{\epsilon})}}\ln{(\frac{1}{\epsilon})}}{(4C+2)\epsilon^{\frac{1}{2}}}.$ But, for $t\geq 0,$ we have that 
\begin{equation}\label{badglobal}
\md{W(t)}\lesssim \epsilon^{\frac{1}{2}}t\lesssim 3\ln{\left(\frac{1}{\epsilon}\right)}\exp\left(\frac{\epsilon^{\frac{1}{2}}t}{3\ln{(\frac{1}{\epsilon})}}\right),\,\md{\dot W(t)}\lesssim\epsilon t\lesssim 3\epsilon^{\frac{1}{2}}\ln{\left(\frac{1}{\epsilon}\right)}\exp\left(\frac{\epsilon^{\frac{1}{2}}t}{3\ln{(\frac{1}{\epsilon})}}\right).
\end{equation}
Since $f(t)$ defined in inequality \eqref{y1t} is strictly increasing and $f(0)\lesssim \frac{1}{\ln{(\frac{1}{\epsilon})}^{2}\ln{\ln{(\frac{1}{\epsilon})}}}$, there is an instant $T_{M}>0$ such that
\begin{equation}\label{defT}
    \exp\left(\frac{\epsilon^{\frac{1}{2}}T_{M}}{\ln{(\frac{1}{\epsilon})}}\right)f(T_{M})=\frac{1}{\ln{(\frac{1}{\epsilon})}\ln{\ln{(\frac{1}{\epsilon})}}^{2}},
\end{equation}
from which with estimate \eqref{y1t} and condition \eqref{hold} we deduce that \eqref{y1t} is true for $0\leq t\leq T_{M}.$ Also, from the identity \eqref{defT} and the fact that $\norm{\overrightarrow{g(0)}}\lesssim \frac{\epsilon^{\frac{1}{2}}}{\ln{(\frac{1}{\epsilon})}^{4}}$ we deduce 
\begin{equation*}
   \frac{1}{\ln{(\frac{1}{\epsilon})}\ln{\ln{(\frac{1}{\epsilon})}}^{2}}\lesssim \frac{1}{\ln{(\frac{1}{\epsilon})}^{2}\ln{\ln{(\frac{1}{\epsilon})}}}\exp\left(\frac{(2C+2)\epsilon^{\frac{1}{2}}T_{M}}{\ln{(\frac{1}{\epsilon})}}\right),
\end{equation*}
from which we obtain that $T_{M}\geq \frac{3}{8(C+1)}\frac{\ln{\ln{(\frac{1}{\epsilon})}}\ln{(\frac{1}{\epsilon})}}{\epsilon^{\frac{1}{2}}}$ for $\epsilon \ll 1.$
In conclusion, since $f(t)$ is an increasing function, we have for $t\geq T_{M}$ and $\epsilon \ll 1$ that
\begin{equation*}
    f(t)\exp\left(\frac{[17(C+1)+4]\epsilon^{\frac{1}{2}}t}{3\ln{(\frac{1}{\epsilon})}}\right)\geq \frac{1}{\ln{(\frac{1}{\epsilon})}\ln{\ln{(\frac{1}{\epsilon})}}^{2}} \exp\left(\frac{[17(C+1)+1]\epsilon^{\frac{1}{2}}t}{3\ln{(\frac{1}{\epsilon})}}\right)\geq \frac{\ln{(\frac{1}{\epsilon})}^{1+\frac{1}{8}}}{\ln{\ln{(\frac{1}{\epsilon})}}^{2}}\exp\left(\frac{\epsilon^{\frac{1}{2}}t}{3\ln{(\frac{1}{\epsilon}})}\right),
\end{equation*}
from which with the estimates \eqref{badglobal} and \eqref{y1t} we deduce for all $t\geq 0$ that
\begin{equation}\label{r1t}
    \md{W(t)}\lesssim \frac{\max\Big(\norm{(g_{0},g_{1})},\epsilon\ln{(\frac{1}{\epsilon})}\Big)^{2}\ln{(\frac{1}{\epsilon})}^{6}}{\epsilon\ln{\ln(\frac{1}{\epsilon})}}\exp\left(\frac{(8C+9)\epsilon^{\frac{1}{2}}t}{\ln{(\frac{1}{\epsilon})}}\right).
\end{equation}
As consequence, we obtain from the estimates \eqref{c100}, \eqref{c200}, \eqref{c1estimate}, \eqref{c2estimate} and \eqref{r1t} that
\begin{equation}\label{r2t}
    \md{\dot W(t)}\lesssim \frac{\max\Big(\norm{(g_{0},g_{1})},\epsilon\ln{(\frac{1}{\epsilon})}\Big)^{2}\ln{(\frac{1}{\epsilon})}^{6}}{\epsilon^{\frac{1}{2}}\ln{\ln(\frac{1}{\epsilon})}}\exp\left(\frac{(16C+18)\epsilon^{\frac{1}{2}}t}{\ln{(\frac{1}{\epsilon})}}\right)
\end{equation}
for all $t\geq 0.$
\\
\textbf{Step 2.}(Estimate of $\md{x_{1}(t)+x_{2}(t)},\,\md{\dot x_{1}(t)+\dot x_{2}(t)}.$)
First, we define
\begin{equation}
    M(t)\coloneqq (x_{1}(t)+x_{2}(t))-(d_{1}(t)+d_{2}(t)),\,
    N(t)\coloneqq (p_{1}(t)+p_{2}(t))-(\dot d_{1}(t)+\dot d_{2}(t)).
\end{equation}
From the inequalities \eqref{modest1}, \eqref{modest2} of Lemma \ref{aux22}, we obtain, respectively: 
\begin{equation*}\md{\dot M(t)-N(t)}\lesssim\max\Big(\norm{(g_{0},g_{1})},\epsilon\ln{\left(\frac{1}{\epsilon}\right)}\Big)\epsilon^{\frac{1}{2}}\exp\Big(\frac{C\epsilon^{\frac{1}{2}}t}{\ln{\left(\frac{1}{\epsilon}\right)}}\Big),\, \md{\dot N(t)}\lesssim\frac{\max\Big(\norm{(g_{0},g_{1})},\epsilon\ln{(\frac{1}{\epsilon})}\Big)^{2}}{\ln{\ln(\frac{1}{\epsilon})}}\exp\Big(\frac{2C\epsilon^{\frac{1}{2}}t}{\ln{(\frac{1}{\epsilon})}}\Big).
\end{equation*}
Also, from inequality \eqref{modest1} and the fact that for $j\in\{1,2\}\, d_{j}(0)=x_{j}(0),\,\dot d_{j}(0)=\dot x_{j}(0),$ we deduce that $M(0)=0$ and $\md{N(0)}\lesssim \max\left(\norm{\overrightarrow{g(0)}},\epsilon\ln{(\frac{1}{\epsilon})}\right)\epsilon^{\frac{1}{2}}.$ Then,
from the Fundamental Theorem of Calculus, we obtain that
\begin{gather}\label{Nt}
    N(t)=O\left(\frac{\max\Big(\norm{(g_{0},g_{1})},\epsilon\ln{(\frac{1}{\epsilon})}\Big)^{2}\ln{(\frac{1}{\epsilon})}}{\epsilon^{\frac{1}{2}}\ln{\ln{(\frac{1}{\epsilon})}}}\exp\Big(\frac{4C\epsilon^{\frac{1}{2}}t}{\ln{(\frac{1}{\epsilon})}}\Big)\right),
\\ \label{Mt}
    M(t)=O\left(\frac{\max\Big(\norm{(g_{0},g_{1})},\epsilon\ln{(\frac{1}{\epsilon})}\Big)^{2}\ln{(\frac{1}{\epsilon})}^{2}}{\epsilon\ln{\ln{(\frac{1}{\epsilon})}}}\exp\Big(\frac{4C\epsilon^{\frac{1}{2}}t}{\ln{(\frac{1}{\epsilon})}}\Big)\right).
\end{gather}
In conclusion, for $K=16C+18,$ we verify from triangle inequality that the estimates \eqref{r1t} and \eqref{Mt} imply \eqref{theo2d1} and the estimates \eqref{r2t} and \eqref{Nt} imply \eqref{theo2d2}.
\end{proof}
\begin{remark}\label{MN}
The estimates \eqref{Mt} and \eqref{Nt} are true for any initial data $(g_{0},g_{1})\in H^{1}(\mathbb{R})\times L^{2}(\mathbb{R})$ such that the hypothesis of Theorem \ref{trueTheo2} are true.
\end{remark}
\begin{remark}[Similar Case]\label{similarcase}
If we add the following conditions
\begin{equation*}
    e^{-\sqrt{2}z(0)}\ll \frac{\epsilon}{\ln{(\frac{1}{\epsilon})}^{8}},\,\frac{\epsilon^{\frac{1}{2}}}{\ln{(\frac{1}{\epsilon})}^{4}}\lesssim v \lesssim \epsilon^{\frac{1}{2}},\, -\ln{\Big(\frac{1}{\epsilon}\Big)}^{2}<c<0,
\end{equation*}
to the hypotheses of Theorem \ref{trueTheo2}, then, by repeating the above proof of Lemma \ref{t21}, we would still obtain \eqref{dc1}, \eqref{dc2}, \eqref{c1estimate} and \eqref{c2estimate}.
\par However, since now $\md{c}\leq \ln{(\frac{1}{\epsilon})}^{2},$ if $\epsilon\ll 1$ enough, we can verify while
\begin{equation}\label{conditioncc}
\norm{W(s)}_{L^{\infty}_{s}[0,t]}\left(\epsilon^{\frac{1}{2}}t+\ln{\left(\frac{1}{\epsilon}\right)}^{2}\right)\ln{\ln{\left(\frac{1}{\epsilon}\right)}}\leq 1,
\end{equation}  
that
\begin{equation*}
    \norm{W(s)}_{L^{\infty}_{s}[0,t]}^{2}(1+\md{n(t)})\lesssim  \norm{W(s)}_{L^{\infty}_{s}[0,t]}\frac{1}{\ln{\ln{(\frac{1}{\epsilon})}}},
\end{equation*}
which implies by a similar reasoning to the proof of Lemma \ref{t21} for a uniform constant $C>1$ the following estimates
\begin{gather}\label{mmt}
    \md{W(t)}\lesssim \frac{\max\Big(\norm{(g_{0},g_{1})},\epsilon\ln{(\frac{1}{\epsilon})}\Big)^{2}\ln{(\frac{1}{\epsilon})}^{7}}{\epsilon\ln{\ln(\frac{1}{\epsilon})}}\exp\Big(\frac{C\epsilon^{\frac{1}{2}}t}{\ln{(\frac{1}{\epsilon})}}\Big)=f_{1}(t,C),\\ \label{dmmt}
    \md{\dot W(t)}\lesssim \max\Big(\norm{(g_{0},g_{1})},\epsilon\ln{\Big(\frac{1}{\epsilon}\Big)}\Big)^{2}\frac{\ln{(\frac{1}{\epsilon})}^{7}}{\epsilon^{\frac{1}{2}}\ln{\ln{(\frac{1}{\epsilon})}}}\exp\Big(\frac{C\epsilon^{\frac{1}{2}}t}{\ln{(\frac{1}{\epsilon})}}\Big)=f_{2}(t,C).
\end{gather}
From the estimates \eqref{mmt}, \eqref{dmmt} and $\norm{\overrightarrow{g(0)}}\leq \frac{\epsilon^{\frac{1}{2}}}{\ln{(\frac{1}{\epsilon})}^{5}}$, we deduce that the condition \eqref{conditioncc} holds while $0\leq t\leq \frac{\ln{\ln{(\frac{1}{\epsilon})}}\ln{(\frac{1}{\epsilon})}}{4(C+1)\epsilon^{\frac{1}{2}}}.$
Indeed, since $\norm{\overrightarrow{g(0)}}^{2}\leq\frac{\epsilon}{\ln{(\frac{1}{\epsilon})}^{10}},$ we can verify that there is an instant $\frac{\ln{\ln{(\frac{1}{\epsilon})}}\ln{(\frac{1}{\epsilon})}}{4(C+1)\epsilon^{\frac{1}{2}}}\leq T_{M}$ such that \eqref{conditioncc} and \eqref{mmt} are true for $0\leq t \leq T_{M}$ and
\begin{equation*}
    f_{1}(T_{M},C)\exp\left(\frac{\epsilon^{\frac{1}{2}}T_{M}}{\ln{(\frac{1}{\epsilon})}}\right)=\frac{1}{\ln{(\frac{1}{\epsilon})}^{2+\frac{1}{2}}\ln{\ln{(\frac{1}{\epsilon})}}}.
\end{equation*}
In conclusion, we can repeat the argument in the proof of step $1$ of Lemma \ref{t21} and deduce that there is $1<K\lesssim C+1$ such that for all $t\geq 0$
\begin{equation}\label{case2}
    |W(t)|\lesssim f_{1}(t,K),\,|\dot W(t)|\lesssim f_{2}(t,K).
\end{equation}
\end{remark}
\begin{lemma}\label{Theo22}
In notation of Theorem \ref{trueTheo2}, $\exists K>1,\,\delta>0$ such that if $0<\epsilon<\delta,\,0<v\leq \frac{\epsilon^{\frac{1}{2}}}{\ln{(\frac{1}{\epsilon})}^{4}},\, (g_{0}(x),g_{1}(x))=\left(g(0,x),\partial_{t}g(0,x)\right)$ and $\norm{\overrightarrow{g(0)}}\leq \frac{\epsilon^{\frac{1}{2}}}{\ln{(\frac{1}{\epsilon})}^{5}},$ then we have for $0\leq t$ that
\begin{align}\label{dddd1}
    \max_{j \in\{1,\,2\}}\md{d_{j}(t)-x_{j}(t)}=O\left(\frac{\max\Big(\norm{(g_{0},g_{1})},\epsilon\ln{(\frac{1}{\epsilon})}\Big)^{2}}{\epsilon\ln{\ln{(\frac{1}{\epsilon})}}}\ln{\Big({\frac{1}{\epsilon}}\Big)}^{2}\exp\Big(\frac{Kt\epsilon^{\frac{1}{2}}}{\ln{\frac{1}{\epsilon}}}\Big)\right),\\ \label{dddd2}
     \max_{j \in\{1,\,2\}}\md{\dot d_{j}(t)-\dot x_{j}(t)}=O\left(\frac{\max\Big(\norm{(g_{0},g_{1})},\epsilon\ln{(\frac{1}{\epsilon})}\Big)^{2}}{\epsilon^{\frac{1}{2}}\ln{\ln{(\frac{1}{\epsilon})}}}\ln{\Big({\frac{1}{\epsilon}}\Big)}\exp\Big(\frac{Kt\epsilon^{\frac{1}{2}}}{\ln{\frac{1}{\epsilon}}}\Big)\right).
\end{align}
\end{lemma}
\begin{proof}[Proof of Lemma \ref{Theo22}]
First, we recall that 
\begin{equation*}
    d(t)=\frac{1}{\sqrt{2}}\ln{\left(\frac{8}{v^{2}}\cosh{\left(\sqrt{2}vt+c\right)}\right)},
\end{equation*}
which implies that
\begin{equation}\label{decayyy}
    e^{-\sqrt{2}d(t)}=\frac{v^{2}}{8}\sech{\left(\sqrt{2}vt+c\right)}^{2}.
\end{equation}
We recall the notation
$W(t)=z(t)-d(t),\,V(t)=p(t)-\dot d(t).$ From the first inequality of Lemma \ref{aux22}, we have that 
\begin{equation}\label{V(0)}
    \md{V(0)}\lesssim \max\left(\norm{(g_{0},g_{1})},\epsilon\ln{\left(\frac{1}{\epsilon}\right)}\right)\epsilon^{\frac{1}{2}}.
\end{equation}
We already verified that $W,\, V$ satisfy the following ordinary differential system
\[\begin{cases}\label{oddd}
\dot W(t)=V(t)+O\Big(\max\Big(\norm{(g_{0},g_{1})},\epsilon\ln{\Big(\frac{1}{\epsilon}\Big)}\Big)\epsilon^{\frac{1}{2}}\exp\Big(\frac{C\epsilon^{\frac{1}{2}}t}{\ln{(\frac{1}{\epsilon})}}\Big)\Big),\\ \\
\dot V(t)=-32e^{-\sqrt{2}d(t)}W(t)+O(e^{-\sqrt{2}z(t)}(W(t))^{2})+O\Big(\frac{\max(\norm{(g_{0},g_{1})},\epsilon\ln{(\frac{1}{\epsilon})})^{2}}{\ln{\ln{(\frac{1}{\epsilon})}}}\exp\Big(\frac{2C\epsilon^{\frac{1}{2}}t}{\ln{(\frac{1}{\epsilon})}}\Big)\Big).
\end{cases}\]
However, since $v^{2}\leq \frac{\epsilon}{\ln{(\frac{1}{\epsilon})}^{8}},$ we deduce from \eqref{decayyy} that $e^{-\sqrt{2}d(t)}\lesssim \frac{\epsilon}{\ln{(\frac{1}{\epsilon})}^{8}}$ for all $t\geq 0.$
So, while $\norm{W(s)}_{L^{\infty}[0,t]}<1,$ we have from the differential ordinary system \eqref{oddd} for $t\geq 0$ and some constant $C>0$ independent of $\epsilon$ that
\begin{equation*}
    \md{\dot V(t)}\lesssim \frac{\epsilon}{\ln{(\frac{1}{\epsilon})}^{8}}\norm{W(s)}_{L^{\infty}[0,t]}+\frac{\max\Big(\norm{(g_{0},g_{1})},\epsilon\ln{(\frac{1}{\epsilon})}\Big)^{2}}{\ln{\ln{(\frac{1}{\epsilon})}}}\exp\Big(\frac{2C\epsilon^{\frac{1}{2}}t}{\ln{(\frac{1}{\epsilon})}}\Big),
\end{equation*}
from which we deduce the following estimate
\begin{equation*}
    \md{V(t)- V(0)}=O\left(\frac{\epsilon t}{\ln{(\frac{1}{\epsilon})}^{8}}\norm{W(s)}_{L^{\infty}[0,t]}\right)
    +O\left(\frac{\max\Big(\norm{(g_{0},g_{1})},\epsilon\ln{(\frac{1}{\epsilon})}\Big)^{2}\ln{(\frac{1}{\epsilon})}}{\epsilon^{\frac{1}{2}}\ln{\ln{(\frac{1}{\epsilon})}}}\exp\Big(\frac{2C\epsilon^{\frac{1}{2}}t}{\ln{(\frac{1}{\epsilon})}}\Big)\right).
\end{equation*}
In conclusion, while $\norm{W(s)}_{L^{\infty}[0,t]}<1,$ we have that
\begin{equation}\label{difference1}
    \md{\dot W(t)}\leq \md{V(0)}
    +O\left(\frac{\max\Big(\norm{(g_{0},g_{1})},\epsilon\ln{(\frac{1}{\epsilon})}\Big)^{2}\ln{(\frac{1}{\epsilon})}}{\epsilon^{\frac{1}{2}}\ln{\ln{(\frac{1}{\epsilon})}}}\exp\Big(\frac{2C\epsilon^{\frac{1}{2}}t}{\ln{(\frac{1}{\epsilon})}}\Big)\right)+O\left(\frac{\epsilon t}{\ln{(\frac{1}{\epsilon})}^{8}}\norm{W(s)}_{L^{\infty}[0,t]}\right).
\end{equation}
Finally, since $W(0)=0,$ the fundamental theorem of calculus and \eqref{difference1} imply the following estimate
\begin{equation}\label{difference2}
    \norm{W(s)}_{L^{\infty}[0,t]}\leq \md{V(0)}t+O\left(\frac{\epsilon t^{2}}{\ln{(\frac{1}{\epsilon})}^{8}}\norm{W(s)}_{L^{\infty}[0,t]}
    +\frac{\max\Big(\norm{(g_{0},g_{1})},\epsilon\ln{(\frac{1}{\epsilon})}\Big)^{2}\ln{(\frac{1}{\epsilon})}^{2}}{\epsilon\ln{\ln{(\frac{1}{\epsilon})}}}\exp\Big(\frac{2C\epsilon^{\frac{1}{2}} t}{\ln{(\frac{1}{\epsilon})}}\Big)\right).
\end{equation}
Then, the estimates \eqref{V(0)}  and \eqref{difference2} imply if $\epsilon\ll 1$  that 
\begin{equation}\label{ww1}
    \md{W(t)}\lesssim \frac{\max\Big(\norm{(g_{0},g_{1})},\epsilon\ln{(\frac{1}{\epsilon})}\Big)^{2}\ln{(\frac{1}{\epsilon})}^{2}}{\epsilon\ln{\ln{(\frac{1}{\epsilon})}}}\exp\Big(\frac{(2C+1)\epsilon^{\frac{1}{2}} t}{\ln{(\frac{1}{\epsilon})}}\Big),
\end{equation}
for $0 \leq t \leq \frac{\ln{(\frac{1}{\epsilon})}\ln{\ln(\frac{1}{\epsilon})}}{(8C+4)\epsilon^{\frac{1}{2}}}.$
From \eqref{ww1} and \eqref{difference1}, we deduce for $0 \leq t \leq  \frac{\ln{(\frac{1}{\epsilon})}\ln{\ln{(\frac{1}{\epsilon})}}}{(8C+4)\epsilon^{\frac{1}{2}}}$ that
\begin{equation}\label{ww2}
    \md{\dot W(t)}\lesssim \frac{\max\Big(\norm{(g_{0},g_{1})},\epsilon\ln{(\frac{1}{\epsilon})}\Big)^{2}\ln{(\frac{1}{\epsilon})}^{2}}{\epsilon^{\frac{1}{2}}\ln{\ln{(\frac{1}{\epsilon})}}}\exp\Big(\frac{(2C+1)\epsilon^{\frac{1}{2}} t}{\ln{(\frac{1}{\epsilon})}}\Big).
\end{equation}
Since $\md{W(t)}\lesssim \epsilon^{\frac{1}{2}}t,\, \md{\dot W(t)}\lesssim \epsilon t$ for all $t\geq 0,$ we can verify by a similar argument to the proof of Step $1$ of Lemma \ref{t21} that for all $t\geq 0$ there is a constant $1<K\lesssim (C+1)$ such that 
\begin{align}\label{fff1}
     \md{W(t)}\lesssim \frac{\max\Big(\norm{(g_{0},g_{1})},\epsilon\ln{(\frac{1}{\epsilon})}\Big)^{2}\ln{(\frac{1}{\epsilon})}^{2}}{\epsilon\ln{\ln{(\frac{1}{\epsilon})}}}\exp\Big(\frac{K\epsilon^{\frac{1}{2}} t}{\ln{(\frac{1}{\epsilon})}}\Big),\\ \label{fff2}
     \md{\dot W(t)}\lesssim \frac{\max\Big(\norm{(g_{0},g_{1})},\epsilon\ln{(\frac{1}{\epsilon})}\Big)^{2}\ln{(\frac{1}{\epsilon})}^{2}}{\epsilon^{\frac{1}{2}}\ln{\ln{(\frac{1}{\epsilon})}}}\exp\Big(\frac{K\epsilon^{\frac{1}{2}} t}{\ln{(\frac{1}{\epsilon})}}\Big).
\end{align}
In conclusion, estimates \eqref{dddd1} and \eqref{dddd2} follow from Remark \ref{MN}, inequalities \eqref{fff1}, \eqref{fff2} and triangle inequality.
\end{proof}
\begin{remark}\label{rfinal}
We recall the definition \eqref{d(t)} of $d(t).$ It is not difficult to verify that if $\norm{\overrightarrow{g(0)}}\leq \frac{\epsilon^{\frac{1}{2}}}{\ln(\frac{1}{\epsilon})^{5}},\, \frac{\epsilon^{\frac{1}{2}}}{\ln(\frac{1}{\epsilon})^{4}}\lesssim v$
and one of the following statements 
\begin{enumerate}
    \item $e^{-\sqrt{2}z(0)}\ll \frac{\epsilon}{\ln{(\frac{1}{\epsilon})}^{8}}$ and $c>0,$
    \item $e^{-\sqrt{2}z(0)}\ll \frac{\epsilon}{\ln{(\frac{1}{\epsilon})}^{8}}$ and $c\leq-\ln{(\frac{1}{\epsilon})}^{2}$
\end{enumerate}
were true, then we would have that $e^{-\sqrt{2}d(t)}\ll \frac{\epsilon}{\ln{(\frac{1}{\epsilon})}^{8}}$ for $0\leq t \lesssim \frac{\ln{(\frac{1}{\epsilon})}^{2}}{\epsilon^{\frac{1}{2}}}.$ 
Moreover, assuming $e^{-\sqrt{2}z(0)}\ln{(\frac{1}{\epsilon})}^{8}\ll\epsilon ,$ if $c>0,$ then we have for all $t\geq 0$ that
\begin{equation*}
    e^{-\sqrt{2}d(t)}=\frac{v^{2}}{8}sech{(\sqrt{2}vt+c)}^{2}\leq \frac{v^{2}}{8}sech{(c)}^{2}=e^{-\sqrt{2}z(0)}\ll \frac{\epsilon}{\ln{(\frac{1}{\epsilon})}^{8}},
\end{equation*}
otherwise if $c\leq -\ln{(\frac{1}{\epsilon})}^{2},$ since $v\lesssim \epsilon^{\frac{1}{2}},$ then there is $1\lesssim K$ such that for $0\leq t \leq \frac{K\ln{(\frac{1}{\epsilon})}^{2}}{\epsilon^{\frac{1}{2}}},$ then $2\md{\sqrt{2}vt+c}>\md{c},$ and so
\begin{equation*}
    e^{-\sqrt{2}d(t)}\leq v^{2}\sech{\left(-\frac{c}{2}\right)}^{2}\ll \frac{\epsilon}{\ln{(\frac{1}{\epsilon})}^{8}}.
\end{equation*}
In conclusion, the result of Lemma \ref{Theo22} would be true for these two cases.
\end{remark}
From the following inequality
\begin{equation*}
\max\Big(\norm{\overrightarrow{g(0)}},\epsilon\ln{\Big(\frac{1}{\epsilon}\Big)}\Big)\leq \ln{\Big(\frac{1}{\epsilon}\Big)}   \max\Big(\norm{\overrightarrow{g(0)}},\epsilon\Big),
\end{equation*}
we deduce from  Lemmas \ref{t21}, \ref{Theo22} and Remarks \ref{MN}, \ref{similarcase} and \ref{rfinal} the statement of Theorem \ref{trueTheo2}.
\section{Proof of Theorem \ref{T1}}
If $\norm{\overrightarrow{g(0)}}\geq \epsilon\ln{(\frac{1}{\epsilon})}$ the result of Theorem \ref{T1} is a direct consequence of Theorem \ref{TT1}. So, from now on, we assume that $\norm{\overrightarrow{g(0)}}<\epsilon\ln{(\frac{1}{\epsilon})}.$
We recall from Theorem \ref{trueTheo2} the notations $v,\,c,\,d_{1}(t),\,d_{2}(t)$ and we denote $d(t)=d_{2}(t)-d_{1}(t)$ that satisfies
\begin{equation*}
    d(t)=\frac{1}{\sqrt{2}}\ln{\Big(\frac{8}{v^{2}}\cosh{(\sqrt{2}vt+c)}^{2}\Big)},\, e^{-\sqrt{2}d(t)}=\frac{v^{2}}{8}\sech{(\sqrt{2}vt+c)}^{2}.
\end{equation*} 
From the definition of $d_{1}(t),\,d_{2}(t),\,d(t)$, we know that $\max_{j\{1,\,2\}}\md{\ddot d_{j}(t)}+e^{-\sqrt{2}d(t)}=O\Big(v^{2}\sech{(\sqrt{2}vt+c)}^{2}\Big)$ and since $z(0)=d(0),\,\dot z(0)= \dot d(0),$ we have that $v,\,c $ satisfy the following identities
\begin{equation*}
    v=\left(e^{-\sqrt{2}z(0)}+\Big(\frac{\dot x_{2}(0)-\dot x_{1}(0)}{2}\Big)^{2}\right)^{\frac{1}{2}},\,
    c= \arctanh{\Big(\frac{\dot x_{2}(0)-\dot x_{1}(0)}{2v}\Big)},
\end{equation*}
so Theorem \ref{Stab} implies that $v\lesssim \epsilon^{\frac{1}{2}}.$ 
From the Corollary \ref{colo2} and the Theorem \ref{trueTheo2}, we deduce that $\exists C>0$ such that if $\epsilon\ll1$ and $0\leq t \leq \frac{\ln{\ln{(\frac{1}{\epsilon})}}\ln{(\frac{1}{\epsilon})}}{\epsilon^{\frac{1}{2}}},$ then we have that
\begin{align}\label{refined1}
    \max_{j\in\{1,\,2\}}\md{\ddot x_{j}(t)}=O\Big(\max_{j\in\{1,2\}}\md{\ddot d_{j}(t)}\Big)+O\left(\epsilon^{\frac{3}{2}}\ln{\Big(\frac{1}{\epsilon}\Big)}^{9}\exp\Big(\frac{Ct\epsilon^{\frac{1}{2}}}{\ln{(\frac{1}{\epsilon})}}\Big)\right),\\ \label{refined2}
    e^{-\sqrt{2}z(t)}=e^{-\sqrt{2}d(t)}+O\left(\max\left(e^{-\sqrt{2}d(t)},e^{-\sqrt{2}z(t)}\right)\md{z(t)-d(t)}\right)=e^{-\sqrt{2}d(t)}+O\left(\epsilon^{2}\ln{\Big(\frac{1}{\epsilon}\Big)}^{9}\exp\Big(\frac{Ct\epsilon^{\frac{1}{2}}}{\ln{(\frac{1}{\epsilon})}}\Big)\right).
\end{align}
Next, we consider a smooth function $0\leq \chi_{2}(x)\leq 1$ that satisfies
\begin{equation*}
\chi_{2}(x)=\begin{cases}
 1, \text{ if $x\leq \frac{9}{20}$,}\\
 0, \text{ if $x\geq \frac{1}{2}$.}
\end{cases}
\end{equation*}
We denote 
\begin{equation*}
    \chi_{2}(t,x)=\chi_{2}\Big(\frac{x-x_{1}(t)}{x-x_{2}(t)}\Big).
\end{equation*}From Theorem \ref{EnergyE} and Remark \ref{generalE}, the estimates \eqref{refined1} and \eqref{refined2} of the modulation parameters imply that for the following functional
\begin{multline*}\label{refinedenergy}
    L_{1}(t)=\big\langle D^{2}E\left(H^{x_{2}(t)}_{0,1}+H^{x_{1}(t)}_{-1,0}\right)\overrightarrow{g(t)},\,\overrightarrow{g(t)}\big\rangle_{L^{2}\times L^{2}}+2\int_{\mathbb{R}}\partial_{t}g(t,x)\partial_{x}g(t,x)\Big[\dot x_{1}(t)\chi_{2}(t,x)+\dot x_{2}(t)(1-\chi_{2}(t,x))\Big]\,dx\\
    -2\int_{\mathbb{R}}g(t,x)\Big(\dot U\left(H^{x_{1}(t)}_{-1,0}(x)\right)+\dot U\left(H^{x_{2}(t)}_{0,1}(x)\right)-\dot U\left(H^{x_{2}(t)}_{0,1}(x)+H^{x_{1}(t)}_{-1,0}(x)\right)\Big)\,dx\\
    +2\int_{\mathbb{R}}g(t,x)\Big[(\dot x_{1}(t))^{2}\partial^{2}_{x}H^{x_{1}(t)}_{-1,0}(x)+(\dot x_{2}(t))^{2}\partial^{2}_{x}H^{x_{2}(t)}_{0,1}(x)\Big]\,dx
    +\frac{1}{3}\int_{\mathbb{R}}U^{(3)}\left(H^{x_{2}(t)}_{0,1}(x)+H^{x_{1}(t)}_{-1,0}(x)\right)g(t,x)^{3}\,dx,
\end{multline*}
and the following quantity $\delta_{1}(t)$ denoted by
\begin{multline*}
\delta_{1}(t)=\norm{\overrightarrow{g(t)}}\Big(e^{-\sqrt{2}z(t)}\max_{j \in \{1,2\}}\md{\dot x_{j}(t)}+\max_{j \in \{1,2\}}\md{\dot x_{j}(t)}^{3}e^{-\frac{9\sqrt{2}z(t)}{20}}+\max_{j \in \{1,2\}}\md{\dot x_{j}(t)} \md{\ddot x_{j}(t)}\Big)\\+\norm{\overrightarrow{g(t)}}^{2}\Big(\frac{\max_{j\in\{1,\,2\}}\md{\dot x_{j}(t)}}{z(t)}+\max_{j\in\{1,\,2\}}\dot x_{j}(t)^{2}+\max_{j\in\{1,\,2\}}\md{\ddot x_{j}(t)}\Big)+\norm{\overrightarrow{g(t)}}^{4},
\end{multline*}
we have $\md{\dot L_{1}(t)}=O(\delta_{1}(t))$ for $t\geq 0.$
Moreover, estimates \eqref{refined1}, \eqref{refined2} and the bound $\dot L_{1}(t)=O(\delta_{1}(t))$ implies that for 
\begin{multline*}
       \delta_{2}(t)=\norm{\overrightarrow{g(t)}}v^{2}\epsilon^{\frac{1}{2}}\sech{(\sqrt{2}v t+c)}^{2}+\norm{\overrightarrow{g(t)}}\epsilon^{2}\ln{\Big(\frac{1}{\epsilon}\Big)}^{9}\exp\Big(\frac{Ct\epsilon^{\frac{1}{2}}}{\ln{(\frac{1}{\epsilon})}}\Big)
   \\+\epsilon^{\frac{3}{2}}e^{-\frac{9\sqrt{2}z(t)}{20}}\norm{\overrightarrow{g(t)}}+
   \max_{j \in \{1,2\}}\frac{\md{\dot x_{j}(t)}}{z(t)}\norm{\overrightarrow{g(t)}}^{2}+\norm{\overrightarrow{g(t)}}^{4},
\end{multline*}
$\md{\dot L_{1}(t)}=O(\delta_{2}(t))$ if $0\leq t \leq \frac{\ln{\ln{(\frac{1}{\epsilon})}}\ln{(\frac{1}{\epsilon})}}{\epsilon^{\frac{1}{2}}}.$
Now, similarly to the proof of Theorem \ref{TT1}, we denote $G(s)=\max\big(\norm{\overrightarrow{g(s)}},\epsilon\big).$
From Theorem \ref{EnergyE} and Remark \ref{generalE}, we have that there are positive constants $K,k>0$ independent of $\epsilon$ such that
\begin{equation*}
    k\norm{\overrightarrow{g(t)}}^{2}\leq L_{1}(t)+K\epsilon^{2}.
\end{equation*}
We recall that Theorem \ref{Stab} implies that
\begin{equation*}
   \ln{\Big(\frac{1}{\epsilon}\Big)}\lesssim z(t),\, e^{-\sqrt{2}z(t)}+\max_{j}\md{\dot x_{j}(t)}^{2}+\max_{j\in\{1,\,2\}}\md{\ddot x_{j}(t)}=O(\epsilon),
\end{equation*}
from which with the definition of $G(s)$ and estimates \eqref{refined1} and \eqref{refined2} we deduce that
\begin{equation*}
    \delta_{1}(t)\lesssim G(t)v^{2}\sech{(\sqrt{2}vt+c)}^{2}\epsilon^{\frac{1}{2}}+G(t)\epsilon^{\frac{39}{20}}+G(t)^{2}\frac{\epsilon^{\frac{1}{2}}}{\ln{(\frac{1}{\epsilon})}},
\end{equation*}
while $0\leq t \leq \frac{\ln{\ln{(\frac{1}{\epsilon})}}\ln{(\frac{1}{\epsilon})}}{\epsilon^{\frac{1}{2}}}.$
In conclusion, the Fundamental Theorem of Calculus implies that $\exists K>0$ independent of $\epsilon$ such that 
\begin{equation}\label{integralll}
    G(t)^{2}\leq K\left(G(0)^{2} +\int_{0}^{t} G(s)v^{2}\sech{(\sqrt{2}vs+c)}^{2}\epsilon^{\frac{1}{2}}+G(s)\epsilon^{\frac{39}{20}}+G(s)^{2}\frac{\epsilon^{\frac{1}{2}}}{\ln{(\frac{1}{\epsilon})}}\,ds\right),
\end{equation}
while $0\leq t \leq \frac{\ln{\ln{(\frac{1}{\epsilon})}}\ln{(\frac{1}{\epsilon})}}{\epsilon^{\frac{1}{2}}}.$ Since $\frac{d}{dt} [\tanh{(\sqrt{2}vt+c)}]=\sqrt{2}v\sech{(\sqrt{2}vt+c)}^{2},$ we verify that while the term $G(s)v^{2}\sech{(\sqrt{2}vt+c)}^{2}\epsilon^{\frac{1}{2}}$ is the dominant in the integral of the estimate \eqref{integralll}, then $G(t)\lesssim G(0).$
The remaining case corresponds when $G(s)^{2}\frac{\epsilon^{\frac{1}{2}}}{\ln{(\frac{1}{\epsilon})}}$ is the dominant term in the integral of \eqref{integralll} from an instant $0\leq t_{0}\leq \frac{\ln{\ln{(\frac{1}{\epsilon})}}\ln{(\frac{1}{\epsilon})}}{\epsilon^{\frac{1}{2}}}.$ Similarly to the proof of \ref{TT1}, we have for $t_{0}\leq t\leq \frac{\ln{\ln{(\frac{1}{\epsilon})}}\ln{(\frac{1}{\epsilon})}}{\epsilon^{\frac{1}{2}}}$ that
$G(t)\lesssim G(t_{0})\exp\Big(C\frac{(t-t_{0})\epsilon^{\frac{1}{2}}}{\ln{(\frac{1}{\epsilon})}}\Big).
$
In conclusion, in any case we have for $0\leq t \leq \frac{\ln{\ln{(\frac{1}{\epsilon})}}\ln{(\frac{1}{\epsilon})}}{\epsilon^{\frac{1}{2}}}$ that
\begin{equation}\label{Compppp2}
    G(t)\lesssim G(0)\exp\Big(C\frac{t\epsilon^{\frac{1}{2}}}{\ln{(\frac{1}{\epsilon})}}\Big).
\end{equation}
However, for $T\geq \frac{\ln{\ln{(\frac{1}{\epsilon})}}\ln{(\frac{1}{\epsilon})}}{\epsilon^{\frac{1}{2}}}$ and $K>2$ we have that
\begin{equation*}
    \epsilon \ln{\Big(\frac{1}{\epsilon}\Big)}\exp \left(K\frac{\epsilon^{\frac{1}{2}}T}{\ln{(\frac{1}{2})}}\right)\leq \epsilon \exp \left(\frac{2K\epsilon^{\frac{1}{2}}T}{\ln{(\frac{1}{2})}}\right).
\end{equation*}
In conclusion, from the result of Theorem \ref{TT1}, we can exchange the constant $C>0$ by a larger constant such that estimate \eqref{Compppp2} is true for all $t\geq 0.$
\appendix
\section{Auxiliary Results}\label{auxil}
We start the Appendix Section by presenting the following lemma: 
\begin{lemma}\label{LemmaLL}
With the same hypothesis as in Theorem \ref{trueTheo2} and using its notation, we have while $\max_{j\in\{1\,,2\}}\md{d_{j}(t)-x_{j}(t)}<1$ that 
$\max_{j\in\{1,\,2\}}\md{\ddot d_{j}(t)-\ddot x_{j}(t)}=O\Big(\max_{j\in\{1,\,2\}}\md{d_{j}(t)-x_{j}(t)}\epsilon+\epsilon z(t)e^{-\sqrt{2}z(t)}
 +\norm{\overrightarrow{g(t)}}\epsilon^{\frac{1}{2}}\Big).$
\end{lemma}
\begin{lemma}\label{Lint}
For $U(\phi)=\phi^{2}(1-\phi^{2})^{2},$ we have that
\begin{multline*}
    \dot U\left(H^{x_{1}(t)}_{-1,0}(x)+H^{x_{2}(t)}_{0,1}(x)\right)- \dot U\left(H^{x_{1}(t)}_{-1,0}(x)\right)-\dot U\left(H^{x_{2}(t)}_{0,1}(x)\right)=
    24e^{-\sqrt{2}z(t)}\left(\frac{H^{x_{1}(t)}_{-1,0}(x)}{(1+e^{-2\sqrt{2}(x-x_{1}(t))})^{\frac{1}{2}}}+\frac{H^{x_{2}(t)}_{0,1}(x)}{(1+e^{2\sqrt{2}(x-x_{2}(t))})^{\frac{1}{2}}}\right)
    \\-30e^{-\sqrt{2}z(t)}\left(\frac{H^{x_{1}(t)}_{-1,0}(x)^{3}}{(1+e^{-2\sqrt{2}(x-x_{1}(t))})^{\frac{1}{2}}}+\frac{H^{x_{2}(t)}_{0,1}(x)^{3}}{(1+e^{2\sqrt{2}(x-x_{2}(t))})^{\frac{1}{2}}}\right)+r(t,x),
\end{multline*}
such that $\norm{r(t)}_{L^{2}_{x}(\mathbb{R})}=O(e^{-2\sqrt{2}z(t)}).$ 
\end{lemma}
\begin{proof}
By directly computations, we verify that
\begin{multline*}
    \dot U\left(H^{x_{1}(t)}_{-1,0}(x)+H^{x_{2}(t)}_{0,1}(x)\right)- \dot U\left(H^{x_{1}(t)}_{-1,0}(x)\right)-\dot U\left(H^{x_{2}(t)}_{0,1}(x)\right)=
   -24H^{x_{1}(t)}_{-1,0}H^{x_{2}(t)}_{0,1}(H^{x_{1}(t)}_{-1,0}+H^{x_{2}(t)}_{0,1})
   \\+30H^{x_{1}(t)}_{-1,0}H^{x_{2}(t)}_{0,1}((H^{x_{1}(t)}_{-1,0})^{3}+(H^{x_{2}(t)}_{0,1})^{3})
   +60(H^{x_{1}(t)}_{-1,0}H^{x_{2}(t)}_{0,1})^{2}(H^{x_{1}(t)}_{-1,0}+H^{x_{2}(t)}_{0,1}).
\end{multline*}
First, from the definition of $H_{0,1}(x),$ we verify that
\begin{gather*}
60(H^{x_{1}(t)}_{-1,0}H^{x_{2}(t)}_{0,1})^{2}(H^{x_{1}(t)}_{-1,0}+H^{x_{2}(t)}_{0,1})=60e^{-2\sqrt{2}z(t)}\left(\frac{H^{x_{2}(t)}_{0,1}}{(1+e^{2\sqrt{2}(x-x_{2}(t))})(1+e^{-2\sqrt{2}(x-x_{1}(t))})}\right)\\+60e^{-2\sqrt{2}z(t)}\left(\frac{H^{x_{1}(t)}_{-1,0}}{(1+e^{-2\sqrt{2}(x-x_{1}(t))})(1+e^{2\sqrt{2}(x-x_{2}(t))})}\right).
\end{gather*}
Using \eqref{kinkestimate}, we can verify using by induction for any $k\in \mathbb{N}$ that
\begin{equation}\label{bbound}
    \md{\frac{d^{k}}{dx^{k}}\left[\frac{1}{(1+e^{2\sqrt{2}x})}\right]}=\md{\frac{d^{k}}{dx^{k}}\left[1-\frac{e^{2\sqrt{2}x}}{(1+e^{2\sqrt{2}x})}\right]}=\md{\frac{d^{k}}{dx^{k}}\left[H_{0,1}(x)^{2}\right]}=O(1),
\end{equation}
and since
$
    \frac{H_{0,1}(x)}{(1+e^{2\sqrt{2}x})}=\frac{e^{\sqrt{2}x}}{(1+e^{2\sqrt{2}x})^{\frac{3}{2}}}
$
is a Schwartz function, we deduce that 
$60(H^{x_{1}(t)}_{-1,0}H^{x_{2}(t)}_{0,1})^{2}(H^{x_{1}(t)}_{-1,0}+H^{x_{2}(t)}_{0,1})$ is in $H^{k}_{x}(\mathbb{R})$ for all $k>0$ and
\begin{equation}
    \norm{(H^{x_{1}(t)}_{-1,0}H^{x_{2}(t)}_{0,1})^{2}(H^{x_{1}(t)}_{-1,0}+H^{x_{2}(t)}_{0,1})}_{H^{k}(\mathbb{R})}=O(e^{-2\sqrt{2}z(t)}). 
\end{equation}
Next, using the identity
\begin{equation}\label{identinter}
    H^{x_{1}(t)}_{-1,0}(x)H^{x_{2}(t)}_{0,1}(x)=-\frac{e^{-\sqrt{2}z(t)}}{(1+e^{2\sqrt{2}(x-x_{2}(t))})^{\frac{1}{2}}(1+e^{-2\sqrt{2}(x-x_{1}(t))})^{\frac{1}{2}}},
\end{equation}
the identity
\begin{equation*}
    1-\frac{1}{(1+e^{2\sqrt{2}x})^{\frac{1}{2}}}=\frac{e^{2\sqrt{2}x}}{(1+e^{2\sqrt{2}x})^{\frac{1}{2}}+(1+e^{2\sqrt{2}x})},
\end{equation*}
and Lemma \ref{interact}, we deduce that
\begin{align}\label{MMMM1}
    \norm{24(H^{x_{1}(t)}_{-1,0})^{2}H^{x_{2}(t)}_{0,1}+24e^{-\sqrt{2}z(t)}\frac{H^{x_{1}(t)}_{-1,0}(x)}{(1+e^{-2\sqrt{2}(x-x_{1}(t))})^{\frac{1}{2}}}}_{L^{2}_{x}(\mathbb{R})}=O(e^{-2\sqrt{2}z(t)}),\\ \label{MMMM2}
    \norm{30(H^{x_{1}(t)}_{-1,0})^{4}H^{x_{2}(t)}_{0,1}+30e^{-\sqrt{2}z(t)}\left(\frac{(H^{x_{1}(t)}_{-1,0}(x))^{3}}{(1+e^{-2\sqrt{2}(x-x_{1}(t))})^{\frac{1}{2}}}\right)}_{L^{2}_{x}(\mathbb{R})}=O(e^{-3\sqrt{2}z(t)}).
\end{align}
The estimate of the remaining terms $-24H^{x_{1}(t)}_{-1,0}(H^{x_{2}(t)}_{0,1})^{2},\,30H^{x_{1}(t)}_{-1,0}(H^{x_{2}(t)}_{0,1})^{4}$ is completely analogous to \eqref{MMMM1} and \eqref{MMMM2} respectively. In conclusion, all of the estimates above imply the estimate stated in the Lemma \ref{Lint}. 
 \end{proof}
\begin{proof}[Proof of Lemma \ref{LemmaLL}]
First, we recall the global estimate $e^{-\sqrt{2}z(t)}\lesssim \epsilon.$ We also recall the identity \eqref{Kenergy}
\begin{equation*}
    \int_{\mathbb{R}}\big (8(H_{0,1}(x))^{3}-6(H_{0,1}(x))^{5}\big )e^{-\sqrt{2}x}\,dx=2\sqrt{2},
\end{equation*}
and the global estimate $e^{-\sqrt{2}z(t)}\lesssim \epsilon.$
which, by integration by parts, implies that
\begin{equation}\label{MMMM3}
    \int_{\mathbb{R}}24\frac{H_{0,1}(x)\partial_{x}H_{0,1}(x)}{(1+e^{2\sqrt{2}(x)})^{\frac{1}{2}}}-30\frac{(H_{0,1}(x))^{3}\partial_{x}H_{0,1}(x)}{(1+e^{2\sqrt{2}(x)})^{\frac{1}{2}}}\,dx=4.
\end{equation}
We recall $d_{1}(t),\, d_{2}(t)$ defined in \eqref{d1} and \eqref{d2} respectively and $d(t)=d_{2}(t)-d_{1}(t).$ Since, we have for $j \in \{1,\, 2\}$ that
$ \ddot d_{j}(t)=(-1)^{j}8\sqrt{2}e^{-\sqrt{2}d(t)},$ we have $\ddot d(t)=16\sqrt{2}e^{-\sqrt{2}d(t)},$ which clearly with the fact that 
\begin{center}$\norm{\partial_{x}H_{0,1}}_{L^{2}}^{2}=\norm{\partial^{2}_{x}H_{0,1}}_{L^{2}}^{2}=\frac{1}{2\sqrt{2}},$\end{center}
imply that $\ddot d_{j}(t)\norm{\partial_{x}H_{0,1}}_{L^{2}}^{2}=(-1)^{j}4e^{-\sqrt{2}d(t)}.$
We also recall the partial differential equation satisfied by the remainder $g(t,x)$ \eqref{NLW2}, which can be rewritten as
\begin{multline}\label{rewnlw}
   \dot U\left(H^{x_{2}(t)}_{0,1}(x)+H^{x_{1}(t)}_{-1,0}(x)\right)-\dot U\left(H^{x_{1}(t)}_{-1,0}(x)\right)-\dot U\left(H^{x_{2}(t)}_{0,1}(x)\right)-\ddot x_{2}(t)\partial_{x}H^{x_{2}(t)}_{0,1}(x)=\\
    -\left(\partial^{2}_{t}g(t,x)-\partial^{2}_{x}g(t,x)+\ddot U\left(H^{x_{2}(t)}_{0,1}(x)+H^{x_{1}(t)}_{-1,0}(x)\right)g(t,x)\right)+\sum_{k=3}^{6} U^{(k)}\left(H^{x_{1}(t)}_{-1,0}+H^{x_{2}(t)}_{0,1}\right)\frac{g(t)^{k-1}}{(k-1)!}\\
    -\dot x_{1}(t)^{2}\partial^{2}_{x}H^{x_{1}(t)}_{-1,0}(x)-\dot x_{2}(t)^{2}\partial^{2}_{x}H^{x_{2}(t)}_{0,1}(x)
    +\ddot x_{1}(t)\partial_{x}H^{x_{1}(t)}_{-1,0}(x).
\end{multline}
In conclusion, from the estimate \eqref{MMMM3}, Lemma \ref{Lint} and Lemma \ref{interact}, we obtain that 
\begin{multline}\label{ddx2}
    \left\langle \dot U\left(H^{x_{1}(t)}_{-1,0}+H^{x_{2}(t)}_{0,1}\right)- \dot U\left(H^{x_{1}(t)}_{-1,0}\right)-\dot U\left(H^{x_{2}(t)}_{0,1}\right),\,\partial_{x}H^{x_{2}(t)}_{0,1}\right\rangle_{L^{2}(\mathbb{R})}-\ddot x_{2}(t)\norm{\partial_{x}H_{0,1}}_{L^{2}}^{2}=
    -(\ddot x_{2}(t)-\ddot d_{2}(t))\norm{\partial_{x}H_{0,1}}_{L^{2}}^{2}\\+O\Big(\md{\ddot x_{1}(t)}z(t)e^{-\sqrt{2}z(t)}+e^{-\sqrt{2}z(t)} \max_{j\in\{1,\,2\}}\md{x_{j}(t)-d_{j}(t)}
    +e^{-2\sqrt{2}z(t)}z(t)\Big).
\end{multline}
We recall from the proof of Theorem \ref{EnergyE} the following estimate 
\begin{equation*}
   \md{\int_{\mathbb{R}}\left[\ddot U\left(H^{x_{2}(t)}_{0,1}(x)\right)-\ddot U\left(H^{x_{2}(t)}_{0,1}(x)+H^{x_{1}(t)}_{-1,0}(x)\right)\right]\partial_{x}H^{x_{2}(t)}_{0,1}(x)g(t,x)\,dx}=O\Big(\norm{\overrightarrow{g(t)}}e^{-\sqrt{2}z(t)}\Big).
\end{equation*}  
Also, from the Modulation Lemma, we have that
\begin{align*}
    \langle \partial^{2}_{t}g(t),\partial_{x}H^{x_{2}(t)}_{0,1}\rangle_{L^{2}}&=\frac{d}{dt}\left[ \langle \partial_{t}g(t),\partial_{x}H^{x_{2}(t)}_{0,1}\rangle_{L^{2}} \right]+\dot x_{2}(t)\langle \partial_{t}g(t),\partial_{x}H^{x_{2}(t)}_{0,1}\rangle_{L^{2}} \\
    &=\frac{d}{dt}\Big[\dot x_{2}(t)\langle g(t),\partial^{2}_{x}H^{x_{2}(t)}_{0,1}\rangle_{L^{2}} \Big]+\dot x_{2}(t)\langle \partial_{t}g(t),\partial_{x}H^{x_{2}(t)}_{0,1}\rangle_{L^{2}}\\
    &=\ddot x_{2}(t)\langle g(t),\partial^{2}_{x}H^{x_{2}(t)}_{0,1}\rangle_{L^{2}} +2\dot x_{2}(t)\langle \partial_{t}g(t),\partial_{x}H^{x_{2}(t)}_{0,1}\rangle_{L^{2}}.
\end{align*}
In conclusion, since $\partial_{x}H^{x_{2}(t)}_{0,1}\in ker D^{2}E_{pot}\left(H^{x_{2}(t)}_{0,1}\right),$ we obtain from \eqref{ddx2} and \eqref{rewnlw} that
\begin{equation*}
    \md{\ddot x_{2}(t)-\ddot d_{2}(t)}=O\Big(\max_{j\in\{1,\,2\}}\md{d_{j}(t)-x_{j}(t)}\epsilon+\epsilon z(t)e^{-\sqrt{2}z(t)}
 \\+\norm{\overrightarrow{g(t)}}e^{-\sqrt{2}z(t)}+\norm{\overrightarrow{g(t)}}\epsilon^{\frac{1}{2}}\Big),
\end{equation*}
the estimate of $\md{\ddot x_{1}(t)-\ddot d_{1}(t)}$ is completely analogous, which finishes the demonstration. 
\end{proof}
\begin{lemma}\label{Lgta}
For any $\delta>0$ there is a $\epsilon(\delta)>0$ such that if 
\begin{equation}\label{mmmm1}
    \norm{\phi(x) -H_{0,1}(x)}_{H^{1}(\mathbb{R})}<+\infty,\,
    0<E_{pot}(\phi(x))-E_{pot}(H_{0,1})<\epsilon(\delta),
\end{equation}
then there is a real number $y$ such that
\begin{equation*}
    \norm{\phi(x)-H_{0,1}(x-y)}_{H^{1}}\leq \delta.
\end{equation*}
\end{lemma}
\begin{proof}[Proof of Lemma \ref{Lgta}.]
 The proof of Lemma \ref{L} will follow by a contradiction argument. We assume that there is a $c>0$ and sequence of real functions $\left(\phi_{n}(x)\right)_{n}$ satisfying 
\begin{align}\label{e1}
    \lim_{n\to +\infty} E_{pot}(\phi_{n})=E_{pot}(H_{0,1}),\\\label{e2} \norm{\phi_{n}(x)-H_{0,1}(x)}_{H^{1}(\mathbb{R})}<+\infty,
\end{align}
such that
\begin{equation}\label{contradict}
   \inf_{y\in\mathbb{R}} \norm{\phi_{n}(x)-H_{0,1}(x+y)}_{H^{1}(\mathbb{R})}>c.
\end{equation}
 Since $H_{0,1}$ has minimum potential energy $E_{pot}$ for all real functions $\phi(x)$ with $\norm{\phi(x)-H_{0,1}(x)}_{H^{1}_{x}}<+\infty,$
 we deduce that following condition 
 \begin{equation*}
     E_{pot}(H_{0,1})\leq E_{pot}(\phi_{n})
 \end{equation*}
 holds for all $n\in\mathbb{N}.$
 Also, for each $n\in \mathbb{N},$ the function $h_{n}(x)=\min\left(\phi_{n}(x),1\right)$ satisfies
 \begin{equation*}
     \frac{d h_{n}(x)}{dx}=\begin{cases}
        0, \text{if $\phi_{n}(x)\geq 1,$}\\
        \frac{d\phi_{n}(x)}{dx} \text{, if $\phi_{n}(x)<1,$}
     \end{cases}
 \end{equation*}
 for almost every $x\in \mathbb{R}$ with respect to the Lebesgue measure.
 In conclusion, from the definition of the potential energy $E_{pot}$ and that $U(\phi)=\phi^{2}(1-\phi^{2})^{2},$ we obtain that
 \begin{equation*}
     E_{pot}(H_{0,1})\leq E_{pot}(h_{n})\leq E_{pot}(\phi_{n}),
 \end{equation*}
from which we conclude that
\begin{equation*}
    \norm{\phi_{n}(x)-1}_{L^{2}\left( \{x|\phi_{n}(x)>1\}\right)}^{2}+\norm{\frac{d\phi_{n}(x)}{dx}}_{L^{2}\left(\{x|\phi_{n}(x)>1\}\right)}^{2}\lesssim \md{E_{pot}(\phi_{n})-E_{pot}(H_{0,1})}.
\end{equation*}
By an analogous argument, we can verify that
\begin{equation*}
    \norm{\phi_{n}(x)}_{L^{2}\left( \{x\vert-\frac{1}{2}<\phi_{n}(x)<0\}\right)}^{2}+\norm{\frac{d\phi_{n}(x)}{dx}}_{L^{2}\left( \{x\vert-\frac{1}{2}<\phi_{n}(x)<0\}\right)}^{2}\lesssim \md{E_{pot}(\phi_{n})-E_{pot}(H_{0,1})},
\end{equation*}
and if there is $x_{0}\in\mathbb{R}$ such that $\phi_{n}(x_{0})\leq -\frac{1}{2},$ we would obtain that
\begin{gather*}
    \int_{x_{0}}^{+\infty}\frac{1}{2}\frac{d\phi_{n}(x)}{dx}^{2}+U(\phi_{n}(x))\,dx =\int_{x_{0}}^{+\infty}\sqrt{2U(\phi_{n}(x))}\md{\frac{d\phi_{n}(x)}{dx}}\,dx+\frac{1}{2}\int_{x_{0}}^{+\infty}\left(\md{\frac{d\phi_{n}(x)}{dx}}-\sqrt{2U(\phi_{n}(x))}\right)^{2}\,dx\\\geq \int_{-\frac{1}{2}}^{1}\sqrt{2U(\phi)}\,d\phi= E_{pot}(H_{0,1})+\int_{-\frac{1}{2}}^{0}\sqrt{2U(\phi)}\,d\phi>E_{pot}(H_{0,1}),
\end{gather*}
which contradicts \eqref{e1} if $n\gg 1.$
In conclusion, after we replace $\phi_{n}(x)$ with $\phi_{1,n}=\max\left(\min\left(\phi_{n}(x),1\right),0\right),$ we can restrict the proof to the case where $0\leq\phi_{n}(x)\leq 1$ and $n\gg 1.$ Now, from the density of $H^{2}(\mathbb{R})$ in $H^{1}(\mathbb{R})$ and with the use of a mollifier distribution, we can also restrict the contradiction hypotheses to the situation where $0\leq \phi_{n}(x)\leq 1$ and $\frac{d\phi_{n}}{dx}(x)$ is a continuous function for all $n\in\mathbb{N}.$
Also, we have that if $\norm{\phi(x)-H_{0,1}(x)}_{H^{1}(\mathbb{R})}<+\infty,$ then $E_{pot}(\phi(x))\geq E_{pot}(H_{0,1}(x)).$ In conclusion, there is a sequence of positive numbers $\left(\epsilon_{n}\right)_{n}$ such that
\begin{equation*}
    E_{pot}(\phi_{n})=E_{pot}(H_{0,1})+\epsilon_{n},\,\lim_{n\to+\infty}\epsilon_{n}=0.
\end{equation*}
Also, $\tau_{y}\phi(x)=\phi(x-y)$ satisfies $E_{pot}(\phi(x))=E_{pot}(\tau_{y}\phi(x))$ for any $y\in\mathbb{R}.$ In conclusion, since for all $n\in \mathbb{N},$ $\lim_{x\to+\infty}\phi_{n}(x)=1$ and $\lim_{x\to-\infty}\phi_{n}(x)=0,$ we can restrict to the case where
\begin{equation*}
    \phi_{n}(0)=\frac{1}{\sqrt{2}},
\end{equation*}
for all $n\in\mathbb{N}.$ For $(v)_{+}=\max(v,0)$ and $(v)_{-}=-\left(v-(v)_{+}\right),$ since $\frac{d\phi_{n}(x)}{dx} $ is a continuous function on $x,$ we deduce that $\left(\frac{d\phi_{n}(x)}{dx} \right)_{+}$ and $\left(\frac{d\phi_{n}(x)}{dx}\right)_{-}$ are also continuous functions on $x$ for all $n \in \mathbb{N}.$ In conclusion, for any $n\in \mathbb{N},$ we have that the set 
\begin{equation}
    U=\left\{x\in\mathbb{R}\vert\,\frac{d\phi_{n}(x)}{dx}<0\right\}
\end{equation}
is an enumerable union of disjoint open intervals $(a_{k,n},b_{k,n})_{k \in \mathbb{N}}$, which are bounded, since $\lim_{x\to +\infty}\phi_{n}(x)=1,\,\lim_{x\to -\infty}\phi_{n}(x)=0$ and $0\leq\phi_{n}(x)\leq 1.$ Now, let $E$ be a set of open bounded intervals $(h_{i,n},l_{i,n})\subset \mathbb{R}$ satisfying the conditions
\begin{equation}\label{condcond}
    \phi_{n}(h_{i,n})=\phi_{n}(l_{i,n}),
\end{equation}
$\{i\vert\,(h_{i,n},l_{i,n})\in E \}=I\subset \mathbb{Z}$ and $l_{i,n}<h_{i+1,n}.$ For any $i\in I,$ the following function 
\begin{equation*}
    f_{i,n}(x)=\begin{cases}
    \phi_{n}(x)\text{ if $x\leq h_{i,n}$},\\
    \phi_{n}(x+l_{i,n}-h_{i,n}) \text{ if $x>h_{i,n},$}\\
    \end{cases}
\end{equation*}
satisfies $E_{pot}(H_{0,1})\leq E_{pot}(f_{i,n})\leq E_{pot}(\phi_{n})=E_{pot}(H_{0,1})+\epsilon_{n},$ which implies that
\begin{equation*}
   \int_{h_{i,n}}^{l_{i,n}}\frac{1}{2}\frac{d\phi_{n}(x)}{dx}^{2}+U(\phi_{n}(x))\leq \epsilon_{n}. 
\end{equation*}
Furthermore, we can deduce from Lebesgue's dominated convergence theorem that
\begin{equation}\label{conddd3}
   \sum_{i\in I} \int_{h_{i,n}}^{l_{i,n}}\frac{1}{2}\frac{d\phi_{n}(x)}{dx}^{2}+U(\phi_{n}(x))\leq \epsilon_{n},
\end{equation}
for every finite or enumerable collection $E$ of disjoint open bounded intervals $(h_{i,n},l_{i,n})\subset \mathbb{R},\,i\in I\subset \mathbb{Z}$ such that
$\phi_{n}(h_{i,n})=\phi_{n}(l_{i,n}).$ In conclusion, we can deduce from \eqref{conddd3} that
\begin{equation}\label{fimalmost}
    \int_{\mathbb{R}}\left(\frac{d\phi_{n}(x)}{dx}\right)_{-}^{2}\,dx\leq 2\epsilon_{n},
\end{equation}
and so for $1 \ll n$ we have that
\begin{equation}\label{almostivp}
   \norm{ \frac{d\phi_{n}(x)}{dx}-\md{\frac{d\phi_{n}(x)}{dx}}}_{L^{2}(\mathbb{R})}^{2}\leq 8\epsilon_{n},\,\phi_{n}(0)=\frac{1}{\sqrt{2}}.
\end{equation}
Moreover, we can verify that
\begin{equation*}
    E_{pot}(\phi_{n})=\frac{1}{2}\left[\int_{\mathbb{R}}\left(\md{\frac{d\phi_{n}(x)}{dx}}-\sqrt{2U(\phi_{n}(x))}\right)^{2}\,dx\right]+\int_{\mathbb{R}}\sqrt{2U(\phi_{n}(x))}\md{\frac{d\phi_{n}(x)}{dx}}\,dx,
\end{equation*}
from which we deduce with $\lim_{x\to-\infty}\phi_{n}(x)=0$ and $\lim_{x\to +\infty}\phi_{n}(x)=1$ that
\begin{align*}
    E_{pot}(H_{0,1})+\epsilon_{n}=E_{pot}(\phi_{n})&\geq \frac{1}{2}\left[\int_{\mathbb{R}}\left(\md{\frac{d\phi_{n}(x)}{d(x)}}-\sqrt{2U(\phi_{n}(x))}\right)^{2}\,dx\right]+\int_{0}^{1}\sqrt{2 U(\phi)}\,d\phi\\&=\frac{1}{2}\left[\int_{\mathbb{R}}\left(\md{\frac{d\phi_{n}(x)}{d(x)}}-\sqrt{2U(\phi_{n}(x))}\right)^{2}\,dx\right]+E_{pot}(H_{0,1}).
\end{align*}
Then, from estimate \eqref{almostivp}, we have that
\begin{equation}\label{IVPF}
   \frac{d\phi_{n}(x)}{dx}=\sqrt{2U(\phi_{n}(x))}+r_{n}(x),\,\phi_{n}(0)=\frac{1}{\sqrt{2}},
\end{equation}
with $\norm{r_{n}}_{L^{2}(\mathbb{R})}\lesssim \epsilon_{n}$ for all $1\ll n.$ We recall that $U(\phi)=\phi^{2}(1-\phi^{2})^{2}$ is a Lipschitz function in the set $\{\phi\vert\,0\leq \phi\leq 1\}.$ Then, because $H_{0,1}(x)$ is the unique solution of the following ordinary differential equation
\begin{equation*}
    \begin{cases}
    \frac{d\phi(x)}{dx}=\sqrt{2U(\phi(x))},\\
    \phi(0)=\frac{1}{\sqrt{2}},
    \end{cases}
\end{equation*}
we deduce from Gronwall Lemma that for any $K>0$ we have
\begin{equation}\label{compact}
   \lim_{n\to+\infty} \norm{\phi_{n}(x)-H_{0,1}(x)}_{L^{\infty}[-K,K]}=0,\,\lim_{n\to +\infty}\norm{\frac{d\phi_{n}(x)}{dx}-\dot H_{0,1}(x)}_{L^{2}[-K,K]}=0.
\end{equation}
Also, if $1\ll n,$ then $\norm{\frac{d\phi_{n}(x)}{dx}}_{L^{2}_{x}(\mathbb{R})}^{2}<2E_{pot}(H_{0,1})+1,$ and so we obtain from Cauchy-Schwartz inequality that
\begin{equation}\label{holder}
    \md{\phi_{n}(x)-\phi_{n}(y)}\leq \md{x-y}^{\frac{1}{2}}\norm{\frac{d\phi_{n}}{dx}}_{L^{2}(\mathbb{R})}^{2}<M\md{x-y}^{\frac{1}{2}},
\end{equation}
for a constant $M>0.$ The inequality \eqref{holder} implies that for any $1>\omega>0$ there is a number $h(\omega)\in \mathbb{N}$ such that
if $n\geq h(\omega)$ then
\begin{equation}\label{uniformbb}
    \norm{\phi_{n}(x)-H_{0,1}(x)}_{L^{\infty}\{x|\,\frac{1}{\omega}<\md{x}\}}<\omega,
\end{equation}
otherwise we would obtain that there are $0<\theta<\frac{1}{4},$ a subsequence $(m_{n})_{n\in \mathbb{N}}$ and a sequence of real numbers $(x_{n})_{n\in\mathbb{N}}$ with $\lim_{n\to+\infty}m_{n}=+\infty,\,\md{x_{n}}> n+1$ such that 
\begin{align}\label{Lcont}
   \md{\phi_{m_{n}}(x_{n})-1}>\theta \text{ if $x_{n}>0,$}\\ \label{Lcont1}
   \md{\phi_{m_{n}}(x_{n})}|>\theta \text{ if $x_{n}<0.$}
\end{align}
However, since we are considering $\phi_{n}(x)\in C^{1}(\mathbb{R})$ and $0\leq \phi_{n}\leq 1,$ we would obtain from the mean value theorem that there would exist a sequence $(y_{n})_{n}$ with $y_{n}>x_{n}>n+1$ or $y_{n}<x_{n}<-n-1$ such that
\begin{align}\label{lcont2}
    1-\theta \leq \phi_{m_{n}}(y_{n})\leq 1+\theta, \text{ if $y_{n}>0,$}\\ \label{lcont3}
    \phi_{m_{n}}(y_{n})= \theta \text{ otherwise.}
\end{align}
But, estimates \eqref{holder}, \eqref{lcont2}, \eqref{lcont3} and identity $U(\phi)=\phi^{2}(1-\phi^{2})^{2}$ would imply that 
\begin{equation}\label{contra2}
    1\lesssim\int_{\md{x}\geq n-2}U(\phi_{m_{n}}(x))\,dx \text{ for all $n\gg 1,$}
\end{equation}
and because of estimate \eqref{compact} and the following identity
\begin{equation}
    \lim_{K\to+\infty}\int_{-K}^{K}\frac{1}{2}\dot H_{0,1}(x)^{2}+U(H_{0,1}(x))=E_{pot}(H_{0,1}(x)),
\end{equation}
estimate \eqref{contra2} would imply that $\lim_{n\to+\infty}E_{pot}(\phi_{m_{n}})>E_{pot}(H_{0,1})$ which contradicts our hypotheses. In conclusion, for any $1>\omega>0$ there is a number $h(\omega)$ such that if $n\geq h(\omega)$ then \eqref{uniformbb} holds. So we deduce for any $0<\omega<1$ that there is a number $h_{1}(\omega)$ such that 
\begin{equation}\label{omegaga}
 \text{if $n\geq h_{1}(\omega),$ then }  \md{\phi_{n}(x)-H_{0,1}(x)}\leq \omega \text{ for all $x\in\mathbb{R}.$}
\end{equation}
Then, if $\omega \leq \frac{1}{100},\,n\geq h(\omega)$ and $K\geq 200,$ estimates \eqref{omegaga} and \eqref{compact} imply that
\begin{align}\label{fimmm1}
    \int_{K}^{+\infty}U(\phi_{n}(x))+\frac{1}{2}\frac{d\phi_{n}(x)}{dx}^{2}\, dx\geq \frac{1}{2}\int_{K}^{+\infty}\left(1-\phi_{n}(x)\right)^{2}+\frac{d\phi_{n}(x)}{dx}^{2}\,dx,\\ \label{fimmm2}
    \int_{-\infty}^{-K}U(\phi_{n}(x))+\frac{1}{2}\frac{d\phi_{n}(x)}{dx}^{2}\, dx\geq \frac{1}{2}\int_{-\infty}^{-K}\phi_{n}(x)^{2}+\frac{d\phi_{n}(x)}{dx}^{2}\,dx.
\end{align}
In conclusion, from estimates \eqref{omegaga}, \eqref{fimmm1}, \eqref{fimmm2} and
\begin{equation*}
    \lim_{K\to+\infty}\int_{|x|\geq K}\frac{1}{2}\dot H_{0,1}(x)^{2}+U(H_{0,1}(x))\,dx=0,
\end{equation*}
we obtain that $\lim_{n\to +\infty}\norm{\phi_{n}-H_{0,1}(x)}_{L^{2}(\mathbb{R})}=0$ and, from the equation in \eqref{IVPF} is satisfied for each $\phi_{n},$ we conclude that $\lim_{n\to +\infty}\norm{\frac{d\phi_{n}}{dx}-\dot H_{0,1}(x)}_{L^{2}(\mathbb{R})}=0.$ In conclusion, if $1\ll n,$ inequality \eqref{contradict} is false.  
\end{proof}
From Lemma \ref{Lgta}, we obtain the following corollary:
\begin{corollary}\label{remarkestimate}
For any $\delta>0$ there is a $\epsilon_{0}>0$ such that if $\epsilon\leq \epsilon_{0},\,\norm{\phi(x)-H_{0,1}(x)-H_{-1,0}(x)}_{H^{1}(\mathbb{R})}<+\infty$ and $E_{pot}(\phi)=2E_{pot}(H_{0,1})+\epsilon,$ then there are $x_{2},x_{1} \in \mathbb{R}$ such that
\begin{equation}
    x_{2}-x_{1}\geq \frac{1}{\delta},\,\norm{\phi(x)-H_{0,1}(x-x_{2})+H_{-1,0}(x-x_{1})}_{H^{1}_{x}(\mathbb{R})}\leq \delta.
\end{equation}
\end{corollary}
\begin{proof}[proof of Corollary \ref{remarkestimate}.]
First, from a similar reasoning to the proof of Lemma \ref{Lgta} we can assume by density that $\frac{d\phi(x)}{dx}\in H^{1}_{x}(\mathbb{R}).$ Next, from hypothesis $\norm{\phi(x)-H_{0,1}(x)-H_{-1,0}(x)}_{H^{1}(\mathbb{R})}<+\infty,$ we deduce using the mean value theorem that there is an $y\in\mathbb{R}$ such that $\phi(y)=0.$ Now, we consider the functions
\begin{equation*}
    \phi_{-}(x)=\begin{cases}
    \phi(x) \text{ if $x\leq y,$}\\
    0 \text{ otherwise,}
    \end{cases}
\end{equation*}
and
\begin{equation*}
    \phi_{+}(x)=\begin{cases}
    0 \text{ if $x\leq y,$}\\
    \phi(x) \text{ otherwise.}
    \end{cases}
\end{equation*}
Clearly, $\phi(x)=\phi_{-}(x)$ for $x<y$ and $\phi(x)=\phi_{+}(x)$ for $x>y.$ From identity $U(0)=0,$ we deduce that 
\begin{equation*}
E_{pot}(\phi)=E_{pot}(\phi_{-})+E_{pot}(\phi_{+}),    
\end{equation*}
also we have that
\begin{equation*}
   E_{pot}(H_{-1,0}) < E_{pot}(\phi_{-}),\,E_{pot}(H_{0,1})<E_{pot}(\phi_{+}).
\end{equation*}
In conclusion, since $E_{pot}(\phi)=2E_{pot}(H_{0,1})+\epsilon,$ Lemma \ref{Lgta} implies that if $\epsilon<\epsilon_{0}\ll 1,$ then there is $x_{2},\,x_{1} \in \mathbb{R}$ such that
\begin{equation}\label{orbity}
  \norm{\phi(x)-H_{0,1}(x-x_{2})-H_{-1,0}(x-x_{1})}_{H^{1}}\leq \norm{\phi_{+}-H_{0,1}(x-x_{2})}_{H^{1}}+\norm{\phi_{-}-H_{-1,0}(x-x_{1})}_{H^{1}}\leq\delta.
\end{equation}
So, to finish the proof of Corollary \ref{remarkestimate}, we need only to verify that we have $x_{2}-x_{1}\geq \frac{1}{\delta}$ if $0<\epsilon_{0}\ll 1.$ But, we recall that $H_{0,1}(0)=\frac{1}{\sqrt{2}},$ from which with estimate \eqref{orbity} we deduce that
\begin{equation}
   \md{\phi_{+}(x_{2})-\frac{1}{\sqrt{2}}}\lesssim\delta, \md{\phi_{-}(x_{1})+\frac{1}{\sqrt{2}}}\lesssim\delta,
\end{equation}
so if $\epsilon_{0}\ll 1,$ then $x_{1}<y<x_{2}.$ From Lemma \ref{DEl2}, we can verify that $f(z)=\norm{DE_{pot}(H_{0,1}^{z}(x)+H_{-1,0}(x))}_{L^{2}}$ is a bounded function in $\mathbb{R}_{+},$ from which with estimate \eqref{orbity} we deduce that if $0<\epsilon_{0}\ll 1,$ then  
\begin{equation*}
   \md{ E_{pot}(\phi)-E_{pot}\left(H_{0,1}(x-x_{2})+H_{-1,0}(x-x_{1})\right)}< e^{-2\sqrt{2}\frac{1}{\delta}}.
\end{equation*}
In conclusion, we obtain from Lemma \ref{LL.1} and the estimate above that $x_{2}-x_{1}\geq \frac{1}{\delta}$ if $0<\epsilon_{0}\ll1$ and $\epsilon<\epsilon_{0}.$
\end{proof}
Now, we complement our material by presenting the proof of Identity \eqref{Kenergy} and the proof of The Modulation Lemma.
\begin{proof}[Proof of Identity \eqref{Kenergy}]
From the definition of the function $H_{0,1}(x),$ we have 
\begin{equation*}
    \int_{\mathbb{R}}\big (8(H_{0,1}(x))^{3}-6(H_{0,1}(x))^{5}\big )e^{-\sqrt{2}x}\,dx=\int_{\mathbb{R}}\frac{8e^{2\sqrt{2}x}+2e^{4\sqrt{2}x}}{(1+e^{2\sqrt{2}x})^{\frac{5}{2}}}\,dx,
\end{equation*}
by the change of variable $y(x)=(1+e^{2\sqrt{2}x}),$ we obtain
\begin{align*}
     \int_{\mathbb{R}}\big (8(H_{0,1}(x))^{3}-6(H_{0,1}(x))^{5}\big )e^{-\sqrt{2}x}\,dx&=\frac{1}{2\sqrt{2}}\int_{1}^{\infty}\frac{8}{y^{\frac{5}{2}}}+\frac{2(y-1)}{y^{\frac{5}{2}}}\,dy
     =\frac{1}{2\sqrt{2}}\int_{1}^{\infty}\frac{6}{y^{\frac{5}{2}}}+\frac{2}{y^{\frac{3}{2}}}\,dy,\\
     &=\frac{1}{2\sqrt{2}}(-4y^{-\frac{3}{2}}-4y^{-\frac{1}{2}})\Big \vert_1^\infty
     =2\sqrt{2}.
\end{align*}
\end{proof}
\begin{proof}[Proof of the Modulation Lemma]
First, let $x_{2},\,x_{1}\in \mathbb{R}$ and $g \in H^{1}(\mathbb{R})$  such that $x_{2}-x_{1}\geq \frac{1}{\delta_{0}}$ with $\delta_{0}>0$ small enough to be chosen later. Then, we define the following map $F:\mathbb{R}^{2}\times H^{1}(\mathbb{R})\to\mathbb{R}^{2}$ by
\begin{equation*}
    F((h_{2},h_{1}),g(x))=
    \begin{bmatrix}
    \langle \partial_{x}H^{x_{2}+h_{2}}_{0,1},\,H^{x_{2}}_{0,1}+H^{x_{1}}_{-1,0}-H^{x_{1}+h_{1}}_{-1,0}+g\rangle_{L^{2}}\\
    \langle \partial_{x}H^{x_{1}+h_{1}}_{-1,0},\,H^{x_{2}}_{0,1}+H^{x_{1}}_{-1,0}-H^{x_{2}+h_{2}}_{0,1}+g\rangle_{L^{2}}\\
    \end{bmatrix}
\end{equation*}
for any $((h_{1},h_{2}),g)\in \mathbb{R}^{2}\times H^{1}(\mathbb{R}).$ Clearly, $F(0,0,0)=(0,0),$ also, 
we can verify that the Derivative $DF_{h_{2},h_{1}}((0,0),g)$ is given by
\begin{equation*}
    \begin{bmatrix}
    \norm{\partial_{x}H_{0,1}}_{L^{2}}^{2}+\langle \partial_{x}H^{x_{2}}_{0,1},\,\frac{dg}{dx}\rangle & \langle \partial_{x}H^{x_{2}}_{0,1},\,\partial_{x}H^{x_{1}}_{-1,0}\rangle\\
    \langle \partial_{x}H^{x_{2}}_{0,1},\,\partial_{x}H^{x_{1}}_{-1,0}\rangle & \norm{\partial_{x}H_{0,1}}_{L^{2}}^{2}+\langle \partial_{x}H^{x_{1}}_{-1,0},\,\frac{dg}{dx}\rangle
    \end{bmatrix}.
\end{equation*}
Then, $R_{g}(h_{2},h_{1})=F(h_{2},h_{1},g)-F(0,0,g)-DF_{h_{2},h_{1}}(0,0,g)(h_{2},h_{1})$ satisfies the following identity
\begin{multline}
    R_{g}(h_{2},h_{1})=
    \begin{bmatrix}
    \langle \partial_{x}H^{x_{2}+h_{2}}_{0,1}-\partial_{x}H^{x_{2}}_{0,1}+h_{2}\partial^{2}_{x}H^{x_{2}}_{0,1},\,g\rangle\\
     \langle \partial_{x}H^{x_{1}+h_{1}}_{-1,0}-\partial_{x}H^{x_{1}}_{-1,0}+h_{1}\partial^{2}_{x}H^{x_{1}}_{-1,0},\,g\rangle
    \end{bmatrix}
    +
    \begin{bmatrix}
    \langle\partial_{x}H^{x_{2}+h_{2}}_{0,1},\,H^{x_{2}}_{0,1}-H^{x_{2}+h_{2}}_{0,1}-h_{2}\partial_{x}H^{x_{2}+h_{2}}_{0,1}\rangle\\
    \langle\partial_{x}H^{x_{1}+h_{1}}_{-1,0},\,H^{x_{1}}_{-1,0}-H^{x_{1}+h_{1}}_{-1,0}-h_{1}\partial_{x}H^{x_{1}+h_{1}}_{-1,0}\rangle
    \end{bmatrix}\\
    +
    \begin{bmatrix}
    \langle\partial_{x}H^{x_{2}+h_{2}}_{0,1},\,H^{x_{1}}_{-1,0}-H^{x_{1}+h_{1}}_{-1,0}\rangle-h_{1}\langle\partial_{x}H^{x_{2}}_{0,1},\,\partial_{x}H^{x_{1}}_{-1,0}\rangle\\
    \langle\partial_{x}H^{x_{1}+h_{1}}_{-1,0},\,H^{x_{2}}_{0,1}-H^{x_{2}+h_{2}}_{0,1}\rangle-h_{2}\langle\partial_{x}H^{x_{2}}_{0,1},\,\partial_{x}H^{x_{1}}_{-1,0}\rangle
    \end{bmatrix}
\end{multline}
for all $(h_{2},h_{1}) \in \mathbb{R}^{2},$ also it is not difficult to verify that $R_{g}(0,0)=(0,0).$ Also for $\delta_{0}>0$ small enough, if $\max\{\md{h_{1}},\md{h_{2}}\}=O(\delta_{0}),\,\norm{g}_{H^{1}(\mathbb{R})}\leq \delta_{0},$ it is not difficult to verify from Lemma \ref{interact} that
\begin{equation*}
    \md{R_{g}(h_{2},h_{1})}\lesssim \norm{g}_{H^{1}}\max\{\md{h_{1}},\md{h_{2}}\}^{2}+\max\{\md{h_{1}},\md{h_{2}}\}^{3}+\max\{\md{h_{1}},\md{h_{2}}\}^{2}(x_{2}-x_{1})e^{-\sqrt{2}(x_{2}-x_{1})},
\end{equation*}
from which we deduce that
\begin{equation}\label{perturb}
    \md{R_{g}(h_{2},h_{1})}=O\Big((\delta_{0}^{2}+e^{-\frac{\sqrt{2}}{2\delta_{0}}})(\md{h_{1}}+\md{h_{2}})\Big),
\end{equation}
for any $((h_{2},h_{1}),g) \in \mathbb{R}^{2}\times H^{1}(\mathbb{R})$ such that $\max\{\md{h_{2}},\md{h_{1}}\}= O(\delta_{0})$ and  $\norm{g}_{H^{1}}\leq \delta_{0}.$ 
In particular, estimate \eqref{perturb} implies that  
$DF_{h_{2},h_{1}}((h_{2},h_{1}),g)$ is an uniformly non-degenerate matrix, for any $(h_{2},h_{1}),\, (x_{2},x_{1})\in \mathbb{R}^{2}$ and $g\in H^{1}(\mathbb{R})$ such $x_{2}-x_{1}\geq\frac{1}{\delta_{0}},\norm{g}_{H^{1}}\leq \delta_{0}$ and $\max\{\md{h_{2}},\md{h_{1}}\}=O(\delta_{0}).$ As a consequence, the result of the Modulation Lemma follows from the Implicit Function Theorem for Banach Spaces with the fact that $F((0,0),0)=(0,0)$. 
\end{proof}
\section{Optimality of Theorem \ref{T1}}\label{opt}
\begin{theorem}\label{optimal}
In notation of Theorem \ref{T1}, for any constant $C>0$ and any function $s:\mathbb{R}_{+}\to \mathbb{R}_{+}$ with $\lim_{h\to 0}s(h)=0,$ we can find a positive value $\delta(s)$ such that if $0<\epsilon\leq \delta(s),$ then for any $\norm{\overrightarrow{g(0)}}\leq \epsilon s(\epsilon)$ there is a $0<T\lesssim \frac{\ln{(\frac{1}{\epsilon})}}{\epsilon^{\frac{1}{2}}}$
such that
$\epsilon\lesssim \norm{\overrightarrow{g(T)}}.$
\end{theorem}
\begin{proof}[Proof of Optimality of Theorem \ref{T1}]
We use the notations of Theorem \ref{T1} and Theorem \ref{trueTheo2}.
Clearly, if the result of Theorem \ref{optimal} is false, then by contradiction there is a function $q:\mathbb{R}_{+}\to \mathbb{R}_{+}$ with $\lim_{h\to 0}q(h)=0$ such that for any $1\ll N \in \mathbb{N}$ is possible to have
\begin{equation}\label{contradicthypo}
    \norm{\overrightarrow{g(t)}}\leq q(\epsilon)\epsilon
\end{equation}
for all $0\leq t\leq N\frac{\ln{(\frac{1}{\epsilon})}}{\epsilon^{\frac{1}{2}}}=T$ if $\epsilon\ll 1$ enough. From Modulation Lemma, we can denote the solution $\phi(t,x)$ as
\begin{equation*}
    \phi(t,x)=H^{x_{1}(t)}_{-1,0}(x)+H^{x_{2}(t)}_{0,1}(x)+g(t,x),
\end{equation*}
such that \begin{equation*}
    \langle g(t,x),\,\partial_{x}H^{x_{1}(t)}_{-1,0}(x)\rangle_{L^{2}(\mathbb{R})}=0,\,\langle g(t,x),\,\partial_{x}H^{x_{2}(t)}_{0,1}(x)\rangle_{L^{2}(\mathbb{R})}=0.
\end{equation*} 
Also, for all $t\geq 0,$ we have that $g(t,x)$ has a unique representation as
\begin{equation}\label{formform}
    g(t,x)=P_{1}(t)\partial^{2}_{x}H^{x_{1}(t)}_{-1,0}(x)+P_{2}(t)\partial^{2}_{x}H^{x_{2}(t)}_{0,1}(x)+r(t,x),
\end{equation}
such that $r(t)$ satisfies the following new orthogonality conditions
\begin{equation}\label{orto2}
    \langle r(t),\,\partial^{2}_{x}H^{x_{1}(t)}_{-1,0}\rangle _{L^{2}(\mathbb{R})}=0,\,
    \langle r(t),\,\partial^{2}_{x}H^{x_{2}(t)}_{0,1}\rangle _{L^{2}(\mathbb{R})}=0.    
\end{equation}
In conclusion, we deduce that
\begin{equation}\label{quadot}
    \norm{g(t,x)}_{L^{2}(\mathbb{R})}^{2}=\norm{\ddot H_{0,1}(x)}_{L^{2}}^{2}(P_{1}^{2}+P_{2}^{2})+\norm{r(t)}_{L^{2}_{x}(\mathbb{R}}^{2}+2P_{1} P_{2}\langle \ddot H^{z(t)}_{0,1}(x),\,\ddot H_{-1,0}(x)\rangle_{L^{2}(\mathbb{R})}.
\end{equation}
We recall from Theorem \ref{Stab} that $\frac{1}{\sqrt{2}}\ln{(\frac{1}{\epsilon})}<z(t)$ for all $t\geq 0.$
Since, from Lemma \ref{interact}, we have that $\langle \partial^{2}_{x}H^{x_{1}(t)}_{-1,0}, \,\partial^{2}_{x}H^{x_{2}(t)}_{0,1}\rangle \lesssim z(t) e^{-\sqrt{2}z(t)}$ and $z(t)e^{-\sqrt{2}z(t)}\lesssim \epsilon\ln{(\frac{1}{\epsilon})}$ if $0<\epsilon\ll 1,$ we deduce from the equation \eqref{quadot} that there is a uniform constant $K>1$  such that for all $t\geq 0$ we have the following estimate
\begin{equation}\label{P1}
   \frac{\norm{g(t)}_{L^{2}}}{K} \leq \md{P_{1}(t)}+\md{P_{2}(t)}+\norm{r(t)}_{L^{2}(\mathbb{R})}\leq K \norm{\overrightarrow{g(t)}}.
\end{equation}
From Theorem \ref{Stab} and orthogonality condition \eqref{orto2}, we deduce that
\begin{equation*}
    \left\langle \partial_{t}r(t,x),\,\partial^{2}_{x}H^{x_{2}(t)}_{0,1}(x)\right\rangle_{L^{2}}=\dot x_{2}(t)\left\langle r(t,x),\,\partial^{3}_{x}H^{x_{2}(t)}_{0,1}(x)\right\rangle_{L^{2}}=O\Big(\norm{r(t)}_{L^{2}}\epsilon^{\frac{1}{2}}\Big). 
\end{equation*}
In conclusion, estimate \eqref{P1} and Lemma \ref{interact} imply that there is a $K>1$ such that
\begin{equation}\label{P2}
    \md{\dot P_{1}(t)}+\md{\dot P_{2}(t)}+\norm{\partial_{t}r(t)}_{L^{2}(\mathbb{R})}\leq K\norm{\overrightarrow{g(t)}}
\end{equation}
for all $t\geq 0.$ Finally, Minkowski inequality and estimate \eqref{P1} imply that there is a uniform constant $K>1$ such that
\begin{equation}\label{P3}
    \norm{\partial_{x}r(t,x)}_{L^{2}(\mathbb{R})}\leq K\norm{\overrightarrow{g(t)}}.
\end{equation}
We recall from Theorem \ref{T2} the following estimate
\begin{equation}\label{PP}
  \frac{\epsilon}{K}  \leq \norm{\overrightarrow{g(t)}}^{2}+\dot x_{1}(t)^{2}+\dot x_{2}(t)^{2}+ e^{-\sqrt{2}z(t)}\leq K \epsilon
\end{equation}
for some uniform constant $K>1.$ Now, from hypothesis \eqref{contradicthypo}, we obtain from Theorem \ref{trueTheo2} and Corollary \ref{colo2} that there are constants $M\in \mathbb{N}$ and $C>0$ such that for all $t\geq 0$ the following inequalities are true
\begin{align}\label{AB1}
    \max_{j\in\{1,\,2\}}\md{x_{j}(t)-d_{j}(t)}\leq \epsilon \ln{\Big(\frac{1}{\epsilon}\Big)}^{M+1}\exp\Big(\frac{10C\epsilon^{\frac{1}{2}}t}{\ln{(\frac{1}{\epsilon})}}\Big),\\ \label{AB2}
     \max_{j\in\{1,\,2\}}\md{\dot x_{j}(t)-\dot d_{j}(t)}\leq \epsilon^{\frac{3}{2}}\ln{\Big(\frac{1}{\epsilon}\Big)}^{M}\exp\Big(\frac{10C\epsilon^{\frac{1}{2}}t}{\ln{(\frac{1}{\epsilon})}}\Big),\\ \label{AB3}
     \max_{j\in\{1,\,2\}}\md{\ddot x_{j}(t)-\ddot d_{j}(t)}\leq \epsilon^{\frac{3}{2}}\ln{\Big(\frac{1}{\epsilon}\Big)}\exp\Big(\frac{10C\epsilon^{\frac{1}{2}}t}{\ln{(\frac{1}{\epsilon})}}\Big),
\end{align}
for a uniform constant $C>0.$
From the partial differential equation \eqref{nlww} satisfied by $\phi(t,x)$ and the representation \eqref{formform} of $g(t,x)$, we deduce in the distributional sense that for any $h(x)\in H^{1}(\mathbb{R})$ that
\begin{multline}\label{pdegeral}
   \left \langle h(x),\,(\ddot P_{1}(t)+\dot x_{1}(t)^{2})\partial^{2}_{x}H^{x_{1}(t)}_{-1,0}+ (\ddot P_{2}(t)+\dot x_{2}(t)^{2})\partial^{2}_{x}H^{x_{2}(t)}_{0,1}\right\rangle_{L^{2}_{x}(\mathbb{R})}=
   \left\langle h(x), -P_{1}(t)\Big[\Big(-\partial^{2}_{x}+\ddot U(H^{x_{1}(t)}_{-1,0})\Big)\partial^{2}_{x}H^{x_{1}(t)}_{-1,0}\Big]\right\rangle\\\left \langle h(x),\,-P_{2}(t)\Big[\Big(-\partial^{2}_{x}+\ddot U(H^{x_{2}(t)}_{0,1})\Big)\partial^{2}_{x}H^{x_{2}(t)}_{0,1}\Big]-\Big[\partial^{2}_{t}r(t)
    -\partial^{2}_{x}r(t)+\ddot U(H^{x_{2}(t)}_{0,1}+H^{x_{1}(t)}_{-1,0})r(t)\Big]\right\rangle_{L^{2}_{x}(\mathbb{R})}\\
    -\left\langle h(x),\,\left[ \dot U(H^{x_{2}(t)}_{0,1}+H^{x_{1}(t)}_{-1,0})-\dot U(H_{0,1}^{x_{2}(t)})-\dot U(H_{-1,0}^{x_{1}(t)})\right] 
    -\ddot x_{1}(t)\partial_{x}H^{x_{1}(t)}_{-1,0}(x)-\ddot x_{2}(t)\partial_{x}H^{x_{2}(t)}_{0,1}(x)\right\rangle_{L^{2}_{x}(\mathbb{R})}
    \\ \left\langle h(x),\, -P_{1}(t)\Big[\Big(\ddot U(H^{x_{2}(t)}_{0,1}+H^{x_{1}(t)}_{-1,0})-\ddot U(H^{x_{1}(t)}_{-1,0})\Big)\partial^{2}_{x}H^{x_{1}(t)}_{-1,0}\Big]-P_{2}(t)\Big[\Big(\ddot U(H^{x_{2}(t)}_{0,1}+H^{x_{1}(t)}_{-1,0})-\ddot U(H^{x_{2}(t)}_{0,1})\Big)\partial^{2}_{x}H^{x_{2}(t)}_{0,1}\Big]\right\rangle\\
    +O\left(\norm{h}_{L^{2}}\left[\norm{g(t)}_{H^{1}}^{2}+\max_{j\in\{1,\,2\}}\md{\ddot x_{j}(t)}+\max_{j\in\{1,\,2\}}\md{\dot P_{j}(x)\dot x_{j}(t)}+\max_{j \in\{1,\,2\}} \md{P_{j}(t)}e^{-\sqrt{2}z(t)}+\md{P_{j}(t)\ddot x_{j}(t)}+\md{P_{j}(t)\dot x_{j}(t)^{2}}\right]\right).
\end{multline}
From Lemma \ref{Lint} and estimates \eqref{AB1} and \eqref{AB3}, we obtain from \eqref{pdegeral} that
\begin{multline}\label{pdegeral1}
   \left \langle h(x),\,(\ddot P_{1}(t)+\dot x_{1}(t)^{2})\partial^{2}_{x}H^{x_{1}(t)}_{-1,0}+ (\ddot P_{2}(t)+\dot x_{2}(t)^{2})\partial^{2}_{x}H^{x_{2}(t)}_{0,1}\right\rangle_{L^{2}_{x}(\mathbb{R})}=
   \left\langle h(x), -P_{1}(t)\Big[\Big(-\partial^{2}_{x}+\ddot U(H^{x_{1}(t)}_{-1,0})\Big)\partial^{2}_{x}H^{x_{1}(t)}_{-1,0}\Big]\right\rangle_{L^{2}(\mathbb{R})}\\\left \langle h(x),\,-P_{2}(t)\Big[\Big(-\partial^{2}_{x}+\ddot U(H^{x_{2}(t)}_{0,1})\Big)\partial^{2}_{x}H^{x_{2}(t)}_{0,1}\Big]-\Big[\partial^{2}_{t}r(t)
    -\partial^{2}_{x}r(t)+\ddot U(H^{x_{2}(t)}_{0,1}+H^{x_{1}(t)}_{-1,0})r(t)\Big]\right\rangle_{L^{2}_{x}(\mathbb{R})}\\+O\left(\norm{h}_{L^{2}}\left[\max_{j\in\{1,\,2\}}\md{\ddot x_{j}(t)-\ddot d_{j}(t)}+e^{-\sqrt{2}d(t)}+\md{z(t)-d(t)}e^{-\sqrt{2}z(t)}+e^{-2\sqrt{2}z(t)}\right]\right)\\
    +O\left(\norm{h}_{L^{2}}\left[\norm{g(t)}_{H^{1}}^{2}+\max_{j\in\{1,\,2\}}\md{\ddot x_{j}(t)}+\max_{j\in\{1,\,2\}}\md{\dot P_{j}(x)\dot x_{j}(t)}+\max_{j \in\{1,\,2\}} \md{P_{j}(t)}e^{-\sqrt{2}z(t)}+\md{P_{j}(t)\ddot x_{j}(t)}+\md{P_{j}(t)\dot x_{j}(t)^{2}}\right]\right).
\end{multline}
From the condition \eqref{orto2}, we deduce that
\begin{align*}
   \left\langle \partial^{2}_{t}r(t),\,\partial^{2}_{x}H^{x_{2}(t)}_{0,1}\right\rangle_{L^{2}}=\frac{d}{dt}\left[\dot x_{2}(t)\left \langle r(t),\,\partial^{3}_{x}H^{x_{2}(t)}_{0,1}
   \right \rangle_{L^{2}}\right]+\dot x_{2}(t)\left \langle \partial_{t}r(t),\,\partial^{3}_{x}H^{x_{2}(t)}_{0,1}
   \right \rangle_{L^{2}},\\
   \left\langle \partial^{2}_{t}r(t),\,\partial^{2}_{x}H^{x_{1}(t)}_{-1,0}\right\rangle_{L^{2}}=\frac{d}{dt}\left[\dot x_{1}(t)\left \langle r(t),\,\partial^{3}_{x}H^{x_{1}(t)}_{-1,0}
   \right \rangle_{L^{2}}\right]+\dot x_{1}(t)\left \langle \partial_{t}r(t),\,\partial^{3}_{x}H^{x_{1}(t)}_{-1,0}
   \right \rangle_{L^{2}},
\end{align*} which with the Theorem \ref{Stab} imply that there is a uniform constant $C>0$ such that
\begin{equation}\label{EO1}
   \md{\left\langle \partial^{2}_{t}r(t),\,\partial^{2}_{x}H^{x_{2}(t)}_{0,1}\right\rangle_{L^{2}}}\leq C\epsilon^{\frac{1}{2}}\norm{\overrightarrow{r(t)}},\,
    \md{\left\langle \partial^{2}_{t}r(t),\,\partial^{2}_{x}H^{x_{1}(t)}_{-1,0}\right\rangle_{L^{2}}}\leq C\epsilon^{\frac{1}{2}}\norm{\overrightarrow{r(t)}}.
\end{equation}
From \eqref{P1}, \eqref{P2} and \eqref{P3}, we obtain that 
$\norm{\overrightarrow{r(t)}}\lesssim \norm{\overrightarrow{g(t)}}.$
In conclusion, after we apply the partial differential equation \eqref{pdegeral1} in distributional sense to $\partial^{2}_{x}H^{x_{2}(t)}_{0,1},\,\partial^{2}_{x}H^{x_{1}(t)}_{-1,0},$ the estimates \eqref{P1}, \eqref{P2}, \eqref{P3}, \eqref{AB1}, \eqref{AB3} and \eqref{EO1} imply that there is a uniform constant $K_{1}>0$ such that if $\epsilon\ll 1$ enough, then for $j \in \{1,\,2\}$ we have that for $0\leq t\leq \frac{N\ln{(\frac{1}{\epsilon})}}{\epsilon^{\frac{1}{2}}}$
\begin{equation*}
    \md{\ddot P_{j}(t)+\dot x_{j}(t)^{2}}\leq K_{1}\Big(e^{-\sqrt{2}d(t)}+\epsilon^{\frac{3}{2}}\ln{\Big(\frac{1}{\epsilon}\Big)}^{M+1}\exp\Big(\frac{10C\epsilon^{\frac{1}{2}}t}{\ln{(\frac{1}{\epsilon})}}\Big)+\epsilon q(\epsilon)\Big),
\end{equation*}
from which we deduce for all $0\leq t\leq N\frac{\ln{(\frac{1}{\epsilon})}}{\epsilon^{\frac{1}{2}}}$ that
\begin{equation}\label{eqqq1}
    \md{\sum_{j=1}^{2}\ddot P_{j}(t)+\dot x_{j}(t)^{2}}\leq 2K_{1}\Big(e^{-\sqrt{2}d(t)}+\epsilon^{\frac{3}{2}}\ln{\Big(\frac{1}{\epsilon}\Big)}^{M+1}\exp\Big(\frac{10C\epsilon^{\frac{1}{2}}t}{\ln{(\frac{1}{\epsilon})}}\Big)+\epsilon q(\epsilon)\Big).
\end{equation}
Since $ \md{\sum_{j=1}^{2}\ddot P_{j}(t)}\geq -\md{\sum_{j=1}^{2}\ddot P_{j}(t)+\dot x_{j}(t)^{2}}+\sum_{j=1}^{2} \dot x_{j}(t)^{2},$
we deduce from the estimates \eqref{eqqq1} and \eqref{PP} that
\begin{equation}\label{finii1}
     \md{\sum_{j=1}^{2}\ddot P_{j}(t)}\geq \frac{\epsilon}{K}-\Big[e^{-\sqrt{2}z(t)}+\norm{\overrightarrow{g(t)}}^{2}\Big]-2K_{1}\Big[e^{-\sqrt{2}d(t)}+\epsilon^{\frac{3}{2}}\ln{\Big(\frac{1}{\epsilon}\Big)}^{M+1}\exp\Big(\frac{10C\epsilon^{\frac{1}{2}}t}{\ln{(\frac{1}{\epsilon})}}\Big)\Big]-2K_{1}\epsilon q(\epsilon).
\end{equation}
We recall that from the statement of Theorem \ref{trueTheo2} that
$e^{-\sqrt{2}d(t)}=\frac{v^{2}}{8}\sech{(\sqrt{2}vt+c)}^{2},$ with
$v=\Big(\frac{\dot z(0)^{2}}{4}+8 e^{-\sqrt{2}z(0)}\Big)^{\frac{1}{2}},$
which implies that $v\lesssim \epsilon^{\frac{1}{2}}.$ Since we have verified in Theorem \ref{Stab} that $e^{-\sqrt{2}z(t)}\lesssim \epsilon,$ the mean value theorem implies that
$\md{e^{-\sqrt{2}z(t)}-e^{-\sqrt{2}d(t)}}=O(\epsilon \md{z(t)-d(t)}),$
from which we deduce from \ref{AB1} that
\begin{equation*}
    \md{e^{-\sqrt{2}z(t)}-e^{-\sqrt{2}d(t)}}=O\Big(\epsilon^{2}\ln{\Big(\frac{1}{\epsilon}\Big)}^{M+1}\exp\Big(\frac{10C\epsilon^{\frac{1}{2}}t}{\ln{(\frac{1}{\epsilon})}}\Big)\Big).
\end{equation*}
In conclusion, if $\epsilon\ll 1$ enough, we obtain for $0\leq t \leq \frac{N\ln{(\frac{1}{\epsilon})}}{\epsilon^{\frac{1}{2}}}$ from \eqref{finii1} that
\begin{equation}\label{finii2}
    \md{\sum_{j=1}^{2}\ddot P_{j}(t)}\geq \frac{\epsilon}{K}-\Big[e^{-\sqrt{2}d(t)}+\norm{\overrightarrow{g(t)}}^{2}\Big]-4K_{1}\Big[e^{-\sqrt{2}d(t)}+\epsilon^{\frac{3}{2}}\ln{\Big(\frac{1}{\epsilon}\Big)}^{M+1}\exp\Big(\frac{10C\epsilon^{\frac{1}{2}}t}{\ln{(\frac{1}{\epsilon})}}\Big)\Big]-2K_{1}\epsilon q(\epsilon).
\end{equation}
 The conclusion of the demonstration will follow from studying separate cases in the choice of $0<v,\,c.$ We also observe that $K,\,K_{1}$ are uniform constants and the value of $N \in \mathbb{N}_{>0}$ can be chosen in the beginning of the proof to be as much large as we need.\\
\textbf{Case 1.}($v^{2}\leq \frac{8\epsilon}{(1+4K_{1})2K}.$)
From inequality \eqref{finii2}, we deduce that
\begin{equation*}
    \md{\sum_{j=1}^{2}\ddot P_{j}(t)}\geq \frac{\epsilon}{2K}-\norm{\overrightarrow{g(t)}}^{2}-4K_{1}\Big(\epsilon^{\frac{3}{2}}\ln{\Big(\frac{1}{\epsilon}\Big)}^{M+1}\exp\Big(\frac{10C\epsilon^{\frac{1}{2}}t}{\ln{(\frac{1}{\epsilon})}}\Big)\Big)-2K_{1}\epsilon q(\epsilon),
\end{equation*}
then, from \eqref{contradicthypo} we deduce for $0\leq t \leq \frac{\ln{(\frac{1}{\epsilon})}}{\epsilon^{\frac{1}{2}}}$ that if $\epsilon$ is small enough, then
$\md{\sum_{j=1}^{2}\ddot P_{j}(t)}\geq \frac{\epsilon}{4K},$
and so,
\begin{equation*}
    \md{\sum_{j=1}^{2}\dot P_{j}(t)}\geq \frac{\epsilon t}{4K}-\md{\sum_{j=1}^{2}\dot P_{j}(0)},
\end{equation*}
which contradicts the fact that \eqref{P2} and \eqref{contradicthypo} should be true for $\epsilon\ll 1.$\\
\textbf{Case 2.}($v^{2}\geq \frac{8\epsilon}{ (1+4K_{1}) 2K},\,\md{c}>2\ln{(\frac{1}{\epsilon})}.$)
It is not difficult to verify that for $0\leq t\leq \min( \frac{\md{c}}{2\sqrt{2}v},N\frac{\ln{(\frac{1}{\epsilon})}}{\epsilon^{\frac{1}{2}}}),$ we have that
$e^{-\sqrt{2}d(t)}\leq \frac{v^{2}}{8}\sech{(\frac{c}{2})}^{2}\lesssim \epsilon^{3},$
then estimate \eqref{finii2} implies that $\md{\sum_{j=1}^{2}\ddot P_{j}(t)}\geq \frac{\epsilon}{4K}$ is true in this time interval. Also, since now $v\cong \epsilon^{\frac{1}{2}},$ we have that
\begin{equation*}
   \frac{\ln{(\frac{1}{\epsilon})}}{\epsilon^{\frac{1}{2}}} \lesssim\frac{\md{c}}{2\sqrt{2}v},
\end{equation*}
so we obtain a contradiction by similar argument to the Case $1.$\\
\textbf{Case 3.}($v^{2}\geq \frac{8\epsilon}{ (1+4K_{1}) 2K},\,\md{c}\leq 2\ln{(\frac{1}{\epsilon})}.$)
For $1\ll N$ enough and $t_{0}=\frac{ (1+4K_{1})^{\frac{1}{2}} K^{\frac{1}{2}} \sqrt{2} \ln{(\frac{1}{\epsilon})}}{\epsilon^{\frac{1}{2}}},$ we have during the time interval $ \left\{t_{0}\leq t \leq 2\frac{(1+4K_{1})^{\frac{1}{2}} K^{\frac{1}{2}} \sqrt{2} \ln{(\frac{1}{\epsilon})}}{\epsilon^{\frac{1}{2}}}\right\}$ that $e^{-\sqrt{2}d(t)}\leq \frac{v^{2}}{8}\sech{\Big(2\ln{\Big(\frac{1}{\epsilon}\Big)}\Big)}^{2}\lesssim \epsilon^{5}.$ In conclusion, estimate \eqref{finii1} implies that $\md{\sum_{j=1}^{2}\ddot P_{j}(t)}\geq \frac{\epsilon}{4K}$ is true in this time interval. From the Fundamental Calculus Theorem, we have that\begin{center}
$\md{\sum_{j=1}^{2}\dot P_{j}(t)}\geq \frac{\epsilon (t-t_{0})}{4K}-\md{\sum_{j=1}^{2}\dot P_{j}(t_{0})}.$\end{center}
In conclusion, hypothesis \eqref{contradicthypo} and estimate \eqref{P2} imply for $T=2\frac{(1+2K_{1})^{\frac{1}{2}} K^{\frac{1}{2}} \sqrt{2} \ln{(\frac{1}{\epsilon})}}{\epsilon^{\frac{1}{2}}}$ that
\begin{equation*}
    \md{\sum_{j=1}^{2}\dot P_{j}(T)}\geq \frac{\epsilon^{\frac{1}{2}} (1+2K_{1})^{\frac{1}{2}}  \sqrt{2} \ln{(\frac{1}{\epsilon})} }{8K^{\frac{1}{2}}},
\end{equation*}
which contradicts the fact that \eqref{contradicthypo} and \eqref{P2} should be true, which finishes our proof.
\end{proof}
\begin{remark}
Indeed, we can use Theorem \ref{optimal} to verify that there is a sequence $(t_{n})_{n\in\mathbb{R}}$ such that $t_{n}\to +\infty$ and $\epsilon\lesssim\norm{\overrightarrow{g(t_{n})}}_{H^{1}\times L^{2}}.$
\end{remark}
\bibliographystyle{plain}
\bibliography{ref}

\begin{thebibliography}{10}

\bibitem{Dynamicsmultiple}
Fabrice Bethuel, Giandomenico Orlandi, and Didier Smets.
\newblock Dynamics of multiple degree \text{Ginzburg-Landau} vortices.
\newblock {\em Comptes Rendus Mathematique}, 342:837--842, 2006.

\bibitem{condensed}
Alan~R. Bishop and Toni Schneider.
\newblock {\em Solitons and Condensed Matter Physics}.
\newblock Springer-Verlag, 1978.

\bibitem{spectral}
David Borthwick.
\newblock {\em Spectral Theory-Basic Concepts and Applications}.
\newblock Springer-Verlag, 2020.

\bibitem{ode}
Earl Coddington and Norman Levinson.
\newblock {\em Theory of Ordinary Differential Equations}.
\newblock McGraw Hill Education, 1955.

\bibitem{VorticesGLS}
J.~E. Colliander and R.~L. Jerrard.
\newblock Vortex dynamics for the \text{Ginzburg-Landau-Schrödinger} equation.
\newblock {\em International Mathematical Reseach Notices}, 1998:333--358,
  1998.

\bibitem{kinkdelort}
Jean-Marc Delort and Nader Masmoudi.
\newblock On the stability of kink solutions of the $\phi^{4}$ model in $1 + 1$
  space time dimensions.
\newblock 2021.
\newblock {Hal:2862414}.

\bibitem{collision}
Patrick Dorey, Kieran Mersh, Tomasz Romanczukiewicz, and Yasha Shnir.
\newblock Kink-antikink collisions in the $\phi^6$ model.
\newblock {\em Physical Review Letters}, 107, 2011.

\bibitem{reducedwave}
Maciej Dunajski and Nicholas~S. Manton.
\newblock Reduced dynamics of \text{Ward} solitons.
\newblock {\em Nonlinearity}, 18(4):1677--1689, 2005.

\bibitem{kinkcollision}
Vakhid~A. Gani, Alexander~E. Kudryavtsev, and Mariya~A. Lizunova.
\newblock Kink interactions in the (1+1)-dimensional $\phi^6$ model.
\newblock {\em Physical Review D: Particles and fields}, 89, 2014.

\bibitem{physicsa2}
John~T. Giblin, Lam Hui, Eugene~A. Lim, and I-Sheng Yang.
\newblock How to run through walls: Dynamics of bubble and soliton collisions.
\newblock {\em Physical Review D: Particles and fields}, 82, 2010.

\bibitem{nonintegrable}
K.~A. Gorshkov and L.~A. Ostrovsky.
\newblock Interactions of solitons in nonintegrable systems: Direct
  perturbation method and applications.
\newblock {\em Physica D: Nonlinear Phenomena}, 3:428--438, 1981.

\bibitem{Vorticesdynamics}
S.~Gustafson and I.M. Sigal.
\newblock Effective dynamics of magnetic vortices.
\newblock {\em Advances in Mathematics}, 199:448--498, 2006.

\bibitem{physicsa1}
S.W. Hawking, I.G. Moss, and J.M. Stewart.
\newblock Bubble collisions in the very early universe.
\newblock {\em Physical Review D}, 26:2681--2713, 1982.

\bibitem{multison}
Jacek Jendrej and Gong Chen.
\newblock Kink networks for scalar fields in dimension $1+1$.
\newblock {\em Nonlinear Analysis}, 215, 2022.

\bibitem{dynamicsgl}
Robert~Leon Jerrard and Halil~Mete Soner.
\newblock Dynamics of \text{Ginzburg-Landau} vortices.
\newblock {\em Archive for Rational Mechanics and Analysis}, 142:99--125, 1998.

\bibitem{refined}
Robert~Leon Jerrard and Daniel Spirn.
\newblock Refined jacobian estimates and \text{Gross–Pitaevsky} vortex
  dynamics.
\newblock {\em Archive for Rational Mechanics and Analysis}, 190:425--475,
  2008.

\bibitem{munoz}
Michal Kowalczyk, Yvan Martel, and Claudio Muñoz.
\newblock Kink dynamics in the $\phi^{4}$ model: asymptotic stability for odd
  perturbations in the energy space.
\newblock {\em Journal of the American Mathematical Society}, 30:769--798,
  2017.

\bibitem{asympt}
Michal Kowalczyk, Yvan Martel, Claudio Muñoz, and Hanne Van~Den Bosch.
\newblock A sufficient condition for asymptotic stability of kinks in general
  (1+1)-scalar field models.
\newblock {\em Annals of PDE}, 7:1--98, 2021.

\bibitem{jkl}
Andrew Lawrie, Jacek Jendrej, and Michal Kowalczyk.
\newblock Dynamics of strongly interacting kink-antikink pairs for scalar
  fields on a line.
\newblock {\em Duke Mathematical Journal}, 2019.

\bibitem{coupled}
N.~S. Manton and J.~M. Speight.
\newblock Asymptotic interactions of critically coupled vortices.
\newblock {\em Communications in Mathematical Physics}, 236:535–555, 2003.

\bibitem{solitonss}
Nicholas Manton and Paul Sutcliffe.
\newblock {\em Topological Solitons}.
\newblock Cambridge Monographs on Mathematical Physics. Cambridge University
  Press, 2004.

\bibitem{collision1}
Yvan Martel and Frank Merle.
\newblock Inelastic interaction of nearly equal solitons for the quartic
  \textit{gKdV} equation.
\newblock {\em Inventiones Mathematicae}, 183(3):563--648, 2011.

\bibitem{gkdvmulti}
Yvan Martel, Frank Merle, and Tai-Peng Tsai.
\newblock Stability and asymptotic stability in the energy space of the sum of
  \text{N} solitons for subcritical \text{gKdV} equations.
\newblock {\em Communications in Mathematical Physics}, 231:347–373, 2002.

\bibitem{munozz}
Claudio Muñoz.
\newblock The \text{Gardner} equation and the stability of multi-kink solutions
  of the \textit{mKdV} equation.
\newblock {\em Discrete and Continuous Dynamical Systems}, 16(7):3811--3843,
  2016.

\bibitem{GL3}
Yu~N Ovchinnikov and I.~M. Sigal.
\newblock The \text{Ginzburg-Landau} equation \text{III.} vortices dynamics.
\newblock {\em Nonlinearity}, 11(5):1277--1294, 1998.

\bibitem{blowup}
Pierre Raphaël and Jeremy Szeftel.
\newblock Existence and uniqueness of minimal blow-up solutions to an
  inhomogeneous mass critical nls.
\newblock {\em Journal of the American Mathematical Society}, 24(2):471–546,
  2010.

\bibitem{gamma}
Etienne Sandier and Sylvia Serfaty.
\newblock Gamma-convergence of gradient flows and applications to
  \text{Ginzburg-Landau} vortex dynamics.
\newblock {\em Communications on Pure and Applied Mathematics}, 57:1627 --
  1672, 2004.

\bibitem{DynamicsHiggs}
D.~Stuart.
\newblock Dynamics of \text{Abelian} \text{Higgs} vortices in the near
  \text{Bogolmony} regime regime.
\newblock {\em Communications in Mathematical Physics}, 159:51--91, 1994.

\bibitem{geodesicYangmills}
D.~Stuart.
\newblock The geodesic approximation for \text{Yang-Mills-Highs} equations.
\newblock {\em Communications in Mathematical Physics}, 166:149--190, 1994.

\bibitem{adiabaticc}
D.~Stuart.
\newblock Analysis of the adiabatic limit for solitons in classical field
  theory.
\newblock {\em Proceedings of the Royal Society A}, 463:2753--2781, 2007.

\bibitem{dsipersivebook}
Terence Tao.
\newblock {\em Nonlinear dispersive equations: local and global analysis}.
\newblock AMS, 2006.

\bibitem{cosmic}
Alexander Vilekin and E.P.S Shellard.
\newblock {\em Cosmic Strings and Other Topological Defects}.
\newblock Cambridge University Press, 1994.

\end{thebibliography}

\end{document}